\documentclass{amsart}
\usepackage{geometry}                
\usepackage{graphicx}
\usepackage{amssymb, amsmath, amscd, amsthm}
\usepackage{epstopdf}
\usepackage{mathdots}
\usepackage[all]{xy}
\usepackage{appendix}
\usepackage[pdftex,bookmarks=true]{hyperref}
\def\C{\mathbb{C}}
\def\Z{\mathbb{Z}}

\def\Q{\mathbb{Q}}

\def\A{\mathbb{A}}

\def\beqn{\begin{equation}}
\def\eeqn{\end{equation}}
\def\beqnn{\begin{equation*}}
\def\eeqnn{\end{equation*}}
\def\beqna{\begin{eqnarray}}
\def\eeqna{\end{eqnarray}}
\def\beqnan{\begin{eqnarray*}}
\def\eeqnan{\end{eqnarray*}}

\newtheorem{theorem}{Theorem}[section]
\newtheorem{lemma}[theorem]{Lemma}
\newtheorem{proposition}[theorem]{Proposition}
\newtheorem{corollary}[theorem]{Corollary}
\newtheorem{conjecture}[theorem]{Conjecture}
\newtheorem{caution}[theorem]{Caution}
\newtheorem{definition}[theorem]{Definition}
\newtheorem{remark}[theorem]{Remark}

\theoremstyle{remark}

\numberwithin{equation}{section}

\begin{document}

\title{The Refined Gross-Prasad Conjecture for Unitary Groups}

\author{R. Neal Harris}
\email{rnharris@math.ucsd.edu}
\address{Department of Mathematics\\ University of California, San Diego\\ 9500 Gilman Drive \#0112 \\ La Jolla, CA\\ 92093-0112}


\begin{abstract}
Let $F$ be a number field, $\A_F$ its ring of ad\`eles, and let $\pi_n$ and $\pi_{n+1}$ be irreducible, cuspidal, automorphic representations of $SO_n(\mathbb{A}_F)$ and $SO_{n+1}(\mathbb{A}_F)$, respectively.  In 1991, Benedict Gross and Dipendra Prasad conjectured the non-vanishing of a certain period integral attached to $\pi_n$ and $\pi_{n+1}$ is equivalent to the non-vanishing of $L(1/2, \pi_n\boxtimes\pi_{n+1})$  \cite{gp}.  More recently, Atsushi Ichino and Tamotsu Ikeda gave a refinement of this conjecture as well as a proof of the first few cases ($n=2,3$) \cite{rgp}.   Their conjecture gives an explicit relationship between the aforementioned $L$-value and period integral.  We make a similar conjecture for unitary groups, and prove the first few cases.  The first case of the conjecture will be proved using a theorem of Waldspurger \cite{wald}, while the second case will use the machinery of the $\Theta$-correspondence.
\end{abstract}
\maketitle

\section{The Gross-Prasad Conjecture}\label{intro}
In \cite{gp}, Benedict Gross and Dipendra Prasad give a conjecture that relates the non-vanishing of a period integral to the non-vanishing of a certain $L$-value.  In this section, we'll discuss that conjecture, as well as a recent refinement due to Atsushi Ichino and Tamotsu Ikeda.

Let $F$ be a number field with ad\`ele ring $\A_F$, and let $V_{n}\subset V_{n+1}$ be quadratic spaces of dimensions $n$ and $n+1$ over $F$, respectively.  Assume that $n\geq 2$ and that $V_{n}$ is not a hyperbolic plane.  We consider the algebraic groups $\operatorname{SO}(V_{n})\subset \operatorname{SO}(V_{n+1})$ defined over $F$.  We denote $G_i := \operatorname{SO}(V_i)$.  Let $\pi_n$ and $\pi_{n+1}$ be irreducible tempered cuspidal automorphic representations of $G_n(\A_F)$ and $G_{n+1}(\A_F)$ respectively.  We fix ismorphisms $\pi_n \cong \otimes_v \pi_{n,v}$ and $\pi_{n+1} \cong \otimes_v \pi_{n+1,v}$.  Suppose that $\operatorname{Hom}_{G_n(k_v)}(\pi_{n+1,v}\otimes \pi_{n,v},\C)\neq 0$ for every place $v$ of $F$.

We briefly recall the notion of an automorphic $L$-function.  We encourage the interest reader to consult \cite{borel} for details.  For $G$ a reductive algebraic group over $F$, $\pi$ an irreducible admissible automorphic representation of $G(\A_F)$, and $r$ a smooth homomorphism $r:{^LG}\to GL_m(\C)$,\footnote{Here, ${^LG}$ is the so=called $L$-group of $G$.} one has the notion of an automorphic $L$-function.  As one expects, this $L$-function is given by a product of local $L$-factors:
$$
L(s,\pi,r) = \prod_v L_v(s,\pi_v,r).
$$

For $v$ outside a sufficiently large set of places (including all archimedean places), such that for $v\notin S$ all relevant data is unramified, one can understand the local $L$-factors using the Satake Isomorphism \cite{satake}.  To wit, for each local factor $\pi_v$, we have an associated conjugacy class $[t_v]$ where $t_v\in {^LG_v}$.  Then, we have the following definition for the local $L$-factors
$$
L_v(s,\pi_v, r) := \det\big(I_m - q_v^{-s}r(t_v)\big)^{-1}.
$$

We now state the Gross-Prasad Conjecture:
\begin{conjecture}[Original Gross-Prasad Conjecture] There exist vectors $\varphi_i\in\pi_i$ such that
$$
\int_{G_n(F)\backslash G_n(\A_F)} \varphi_{n+1}(g_n)\varphi_{n}(g_n)\ dg_n\neq 0
$$
if and only if
$$
L(1/2, \pi_{n+1}\boxtimes \pi_{n})\neq 0.
$$
The $L$-function here is the product $L$-function.
\end{conjecture}

Recently in \cite{rgp}, a refinement to this conjecture was proposed which gives a precise relationship between the period integral above and the $L$-value.

Consider the following $G_{n}(\A_F)\times G_{n}(\A_F)$-invariant functional
$$
\mathcal{P}: (V_{\pi_{n+1}}\boxtimes \bar{V}_{\pi_{n+1}})\otimes ({V}_{\pi_{n}}\boxtimes \bar{V}_{\pi_{n}})\to\C
$$
defined by
\beqn\label{pdef}
\mathcal{P}(\phi_1,\phi_2;f_1,f_2):=\left(\int_{[G_n]} \phi_1(g)f_1(g) dg\right)\cdot\left(\int_{[G_n]} \overline{\phi_2(g)f_2(g)} dg\right)
\eeqn
for $\phi_i\in V_{\pi_{n+1}}$ and $f_i\in V_{\pi_n}$.  If $\phi_1=\phi_2=\phi$ and $f_1=f_2=f$, we simply write $\mathcal{P}(\phi,f):=\mathcal{P}(\phi_1,\phi_2;f_1,f_2).$  We call $\mathcal{P}$ the global period.

We also construct another $G_n(\A_F)\times G_n(\A_F)$-invariant functional, this time constructed from local integrals.  For each place $v$ of $F$, denote $G_{i,v}:=G_i(F_v)$.  We fix local pairings
$$
\mathcal{B}_{\pi_{i,v}}:\pi_{i,v}\otimes\bar{\pi}_{i,v}\to\C
$$
such that
$$
\mathcal{B}_{\pi_i} = \prod_v \mathcal{B}_{\pi_{i,v}}
$$
where the $\mathcal{B}_{\pi_i}$ are the Petersson pairings
\beqnan
\mathcal{B}_{\pi_n}(f_1,f_2) &:=& \int_{[G_n]} f_1(g_n)\overline{f_2(g_n)}dg_n\\
\mathcal{B}_{\pi_{n+1}}(\phi_1,\phi_2) &:=& \int_{[G_{n+1}]} \phi_1(g_{n+1})\overline{\phi_2(g_{n+1})}dg_{n+1}
\eeqnan
and the $dg_i$ are Tamagawa measures on $G_i(\A_F)$.
For each $v$, we define a $G_{n,v}\times G_{n,v}$ invariant functional
$$
\mathcal{P}_v^\natural: (\pi_{n+1,v}\boxtimes \bar{\pi}_{n+1,v})\otimes ({\pi}_{n,v}\boxtimes \bar{\pi}_{n,v})
$$
by
$$
\mathcal{P}_v^\natural(\phi_{1,v},\phi_{2,v}; f_{1,v}, f_{2,v}) := \int_{G_{n,v}} \mathcal{B}_{\pi_{n+1,v}}(\pi_{n+1,v}(g_{n,v})\phi_{1,v},\phi_{2,v})\mathcal{B}_{\pi_{n,v}}(\pi_{n,v}(g_{n,v})f_1,f_2) dg_{n,v}.
$$
Here, the $dg_{n,v}$ are local Haar measures, chosen so that $\prod_v dg_{n,v}=dg_n$.

Again, we denote $\mathcal{P}_v^\natural(\phi_{v},\phi_{v}; f_{v}, f_{v})=: \mathcal{P}_v^\natural(\phi_v,f_v).$  We set
\beqnan
\Delta_{G_i} &:=& L(M^\vee_{i}(1),0)\\
\Delta_{G_{i,v}} &:=& L_v(M^\vee_{i}(1),0)
\eeqnan
where $M^\vee_{n+1}(1)$ is the twisted dual of the motive $M_{i}$ associated to $G_{i}$ by Gross in \cite{motive}.
It is a result of Ichino and Ikeda (Theorem 1.2 in \cite{rgp}) that the $\mathcal{P}_v^\natural$ converge absolutely if the $\pi_{i,v}$ are tempered.  Furthermore, when the $\mathcal{P}_v^\natural$ converge, we have
$$
\mathcal{P}_v^\natural(\phi_v,f_v) = \Delta_{G_{n+1,v}}\frac{L_v(1/2,\pi_{n,v}\boxtimes\pi_{n+1,v})}{L_v(1,\pi_{n,v},\operatorname{Ad})L_v(1,\pi_{n+1,v},\operatorname{Ad})}
$$
for unramified data $\phi_v,f_v$ satisfying conditions $(U1)-(U6)$ in \cite{rgp}.  So we normalize as follows:
$$
\mathcal{P}_v := \Delta_{G_{n+1,v}}^{-1}\frac{L_v(1,\pi_{n,v},\operatorname{Ad})L_v(1,\pi_{n+1,v},\operatorname{Ad})}{L_v(1/2,\pi_{n,v}\boxtimes\pi_{n+1,v})}\mathcal{P}_v^\natural.
$$
Now we have another $G_n(\A_F)\times G_n(\A_F)$-invariant functional:
$$
\prod_v\mathcal{P}_v: (V_{\pi_{n+1}}\boxtimes \bar{V}_{\pi_{n+1}})\otimes ({V}_{\pi_{n}}\boxtimes \bar{V}_{\pi_{n}})\to\C.
$$
The Refined Gross-Prasad Conjecture gives the explicit constant of proportionality between $\mathcal{P}$ and $\prod_v\mathcal{P}_v$:
\begin{conjecture}[Refined Gross-Prasad Conjecture]
$$
\mathcal{P}(\phi,f) = \frac{\Delta_{G_{n+1}}}{2^\beta}\frac{L(1/2,\pi_n\boxtimes\pi_{n+1})}{L(1,\pi_n,\operatorname{Ad})(1,\pi_{n+1},\operatorname{Ad})}\prod_v \mathcal{P}_v(\phi_v,f_v).
$$
Here, $\beta$ is an integer such that $2^\beta = |S_{\psi_{n+1}}|\cdot |S_{\psi_n}|$, where $\psi_i$ is the conjectural $L$-parameter for $\pi_i$, and $S_{\psi_i}:=\operatorname{Cent}_{\widehat{G_i}}(\operatorname{Im}(\psi_i))$ is the associated component group.
\end{conjecture}
This conjecture is known unconditionally for $n=2$ \cite{wald}, for $n=3$ \cite{triple}, and for $n=4$, assuming that $\pi_5$ is a $\Theta$-lift of a representation on $SO(4)$ \cite{so4so5}.

Our goal is to provide evidence for an analogous conjecture for unitary groups.  We remark that the original Gross-Prasad Conjecture has already been extended to unitary groups in \cite{ggp}. The conjecture for unitary groups is:
\begin{conjecture}[Refined Gross-Prasad Conjecture for Unitary Groups]\label{theconjecture}
Let $\pi_n$ and $\pi_{n+1}$ be irreducible, cuspidal, tempered, automorphic representations of $G_n(\A_F)$ and $G_{n+1}(\A_F)$, respectively.  Let $\mathcal{P}$ be as in line \ref{pdef}.  Then
$$
\mathcal{P} = \frac{\Delta_{G_{n+1}}}{2^\beta}\frac{L_E(1/2, BC(\pi_{n+1})\boxtimes BC(\pi_n))}{L_F(1,\pi_{n+1}, \operatorname{Ad})L_F(1,\pi_n,\operatorname{Ad})}\prod_v \mathcal{P}_v.
$$
Here, $BC(\pi_i)$ denotes the quadratic base-change of $\pi_i$ to a representation of $GL_i(\A_E)$.  Also, $\beta$ is an integer such that $2^\beta = |S_{\psi_{n+1}}|\cdot |S_{\psi_n}|$, where $\psi_i$ is the conjectural $L$-parameter of $\pi_i$, and $S_{\psi_i}$ is the associated component group.
\end{conjecture}
Michael Harris has recently found an application of the conjecture above \cite{adjoint}.

We remark that in \cite{rog}, the author constructs the quadratic base-change for automorphic representations of $U(n)$ to $GL_n$ for $n=2,3$.  So, in this case, we needn't assume the existence of the quadratic base-change.  Furthermore, the $\psi_i$ are not conjectural in this case.
\begin{remark} One might ask why the $L$-values on the RHS above are non-zero.  Indeed, in \cite{rgp} the authors comment that for orthogonal groups it is believed that $L(1,\pi_i,\operatorname{Ad})\neq 0$ for $\pi_i$ tempered.  However, we can say something stronger for unitary groups.  Namely, by results in \cite{ggp2}, we see that
$$
L_F(s, \pi_i, \operatorname{Ad}) = L_F(s, BC(\pi_i), \operatorname{As}^{(-1)^i}).
$$
Here, we are viewing $BC(\pi_i)$ as a representation of $GL_i(\A_F)$ via $\operatorname{Res}_{E/F}$.  The $L$-function on the RHS above is the `Asai' (if $i$ is even) or `twisted Asai' (if $i$ is odd) $L$-function.  Now, since the $\pi_i$ are assumed to be tempered, by Theorem 5.1 in \cite{shahidi} we have that $L_F(s, BC(\pi_i), \operatorname{As}^{(-1)^i})$ is holomorphic and nonzero at $s=1$.
\end{remark}

The rest of the paper is organized as follows: in Section \ref{localchapt}, we compute $\mathcal{P}_v$ for unramified data.  Then, in Section \ref{globalchapt}, we show how a result of Waldspurger can be used to obtain Conjecture \ref{theconjecture} for $n=1$ without too much pain.  Section \ref{triplechapt} focusses on Ichino's triple product formula, a tool we need to obtain Conjecture \ref{theconjecture} for $n=2$.  In Section \ref{thetachapt}, we introduce the theta correspondence for unitary groups, as well as give several versions of the Rallis Inner Product Formula.  Section \ref{seesawchapt} is local in nature; we develop a local seesaw identity, which allows us to relate the $\mathcal{P}_v$ to the local integrals considered by Ichino in his triple product formula.  Finally, Section \ref{finalchapt} puts everything together; we give a proof of Conjecture \ref{theconjecture} for $n=2$ (assuming the representation of $U(3)$ is a theta lift).

\section{Local Integrals of Matrix Coefficients}\label{localchapt} In this section we'll prove the convergence of the $\mathcal{P}_v^\natural$, assuming $\pi_{n,v}$ and $\pi_{n+1,v}$ are tempered.  This comes down to some standard arguments based on well-known bounds of matrix coefficients.  We'll also compute them for $v\notin S$, where $S$ is a sufficiently large finite set of `bad' places of $F$, including all even places and archimedean ones.  We also make the following assumptions for $v\notin S$ (cf. $(U1) - (U6)$ on page 5 of \cite{rgp}):
\begin{enumerate}
\item The extension $E/F$ is unramified at $v$.
\item $G_{i,v}$ is unramified over $F_v$.
\item $K_{i,v}\subset G_{i,v}$ is a hyperspecial maximal compact subgroup.
\item $K_{i,v}\subset K_{i+1,v}$.
\item $\pi_{i,v}$ is an unramified representation of $G_{i,v}.$
\item The local Haar measures $dg_{i,v}$ are chosen so that the $K_{i,v}$ have volume $1$.
\item The vectors $f_{i,v}\in\pi_{i,v}$ are $K_{i,v}$-fixed and $||f_{i,v}||=1$.
\end{enumerate}
Note that even for $v\in S$, we still fix a maximal compact subgroup $K_i\subset G_i$.  For the remainder of this section, we will omit $v$ from the notation, though everything is local.

Put
$$
L_{\pi_{i+1}, \pi_i}(s) := \frac{L_E(s, BC(\pi_{i+1})\boxtimes BC(\pi_{i}))}{L_F(s+1/2, \pi_{i+1}, \operatorname{Ad}) L_F(s+1/2, \pi_i, \operatorname{Ad})}.
$$
We also consider the matrix coefficient
$$
\Phi_{\varphi_i, \varphi_i'}(g_i) := \mathcal{B}_{\pi_i}(\pi_i(g)\varphi_i,\varphi_i')
$$
where the $\mathcal{B}_{\pi_i}$ are pairings with respect to which the $\pi_i$ are unitary, $g_i\in G_i$, and $K_{i,v}$-finite vectors $\varphi_i,\varphi_i'\in\pi_i$.  When $\varphi_i=\varphi_i'$, we simply refer to $\Phi_{\varphi_i}$.  We consider the following integral:
$$
\mathcal{P}(\varphi_{n+2}, \varphi_{n+1}) := \int_{G_{n+1}} \Phi_{\varphi_{n+2}}(g){\Phi_{\varphi_{n+1}}(g)} dg.
$$
We will establish convergence of the integral assuming temperedness of the $\pi_i$, and compute the integrals away from $S$.

\subsection{Convergence of Integral}  Note that we make no assumption that $v\notin S$, but for now, we assume that $v$ does not split in $E/F$.

Let $V$ be a hermitian space over $E$ of dimension $n$, and set $G:=U(V)$.  Let $V_{\text{an}}$ the the anisotropic kernel of $V$.  Let $d$ denote the dimension of $V_{\text{an}}$.  We have
$$ 
V=X\oplus V_{\text{an}}\oplus Y
$$
where $X$ and $Y$ are totally isotropic subspaces.  We set $r=\dim_E X=\dim_E Y$.  By fixing a basis for $X$, we obtain a minimal parabolic subgroup $P\subset G$.  The Levi factor $M\subset P$ is isomorphic to $(E^\times)^r\times U(V_\text{an})$ with maximal torus $T$ such that $M\supset T\cong (E^\times)^r$.  The split component $A\subset T\subset M$ is isomorphic to $(F^\times)^r$.  We denote an element $x\in A$ as $x=(x_1,x_2,\dots, x_{r})$.  The simple roots of $(P,A)$ are given by
$$
\alpha_1(x) = x_1x_2^{-1}, \dots, \alpha_{r-1}(x) = x_{r-1}x_{r}^{-1}
$$
and
$$
\alpha_{r}(x) = \begin{cases} x_{r} & \mbox{if } n\mbox{ is odd or } U(V)\mbox{ is not quasi-split}\\ x_{r}^{2} & \mbox{otherwise.}\end{cases}
$$
We view the $\alpha_i$ as elements of $\operatorname{Hom}(T,\mathbb{G}_m)$.  Via the natural projection $M\to T$, we may also view the $\alpha_i$ as characters of $M$.

If we denote by $\delta$ the modulus character of $P$, then we have
$$
\delta(x) = \prod_{i=1}^{r} |x_i|_E^{n+1-2i}.
$$
(See Proposition $1.2$ in \cite{khoury}.)
Fix a special maximal compact subgroup $K\subset G$.  Then we have a Cartan decomposition
$$
G = KM^+K
$$
where
$$
M^+ := \{m\in M: |\alpha_i(m)|\leq 1\text{ for } 1\leq i\leq r\}
$$
and $K$ is in good position relative to $M$.\footnote{$K$ is such that the unique point in the building of $G$ fixed by $K$ lies on the apartment of $A$.}  We also define $T^+:=T\cap M^+$.

We fix an embedding $\eta: G\hookrightarrow GL_m$ for some $m$.  We define a height function $$\sigma(g):=\max_{1\leq i,j\leq m}\{\log |\eta(g)_{ij}|, \log |\eta(g^{-1})_{ij}|\}.$$  Let $\Xi(g)$ be Harish-Chandra's spherical function given by
$$
\Xi(g) := \int_K h(kg)dk
$$  where $h\in\operatorname{ind}_P^G 1$ is the function that is identically $1$ on $K$.  It is known that there exist positive constants $A, B$ such that $$A^{-1}\delta^{1/2}(m)\leq \Xi(m)\leq A\delta^{1/2}(m)(1+\sigma(m))^B$$ for all $m\in M^+$.  Also, recall that a function $f(g)$ on $G$ is said to satisfy the weak inequality if
$$
|f(g)|\leq A\cdot \Xi(g)(1+\sigma(g))^B
$$
for some positive constants $A, B$, and all $g\in G$.  It is known that a matrix coefficient of a tempered representation satisfies the weak inequality.  (See \cite{weak}, for example.)

Let $V_{n+2}$ and $V_{n+1}$ be hermitian spaces of dimension $n+2$ and $n+1$, respectively.  Assume further that we have an embedding $\iota:V_{n+1}\hookrightarrow V_{n+2}$ of hermitian spaces.  Let $G_i=U(V_i)$ be their associated unitary groups, and let $P_i, M_i, A_i, K_i$ denote the respective minimal parabolic subgroups, Levi component, maximal split tori, and special maximal compact subgroups of $G_i$.  We view $G_{n+1}$ as a subgroup of $G_{n+2}$ via the embedding $\iota$.  Note that we may assume that $T_{n+1}\subset T_{n+2}, T_{n+1}^+\subset T_{n+2}^+, M_{n+1}\subset M_{n+2}, M_{n+1}^+\subset M_{n+2}^+$ in this case.

The main goal of this section is to prove the following:

\begin{proposition}\label{conv}  The integral $\mathcal{P}(\varphi_{i+1}, \varphi_i)$ converges absolutely.
\end{proposition}
\begin{proof}
This proof is just an adaptation of the analogous proposition in \cite{rgp}.  First, we note the following result from calculus:
\begin{lemma}\label{calc} Let $D, r_1,\dots,r_i>0$ and $r_{i+1},\dots, r_n<0$.  Then the integral
\beqnan
&&\int_{|x_1|\leq |x_2|\leq\dots\leq |x_i|\leq 1\leq |x_{i+1}|\leq\dots\leq |x_n|} |x_1|^{r_1}\dots |x_n|^{r_n}\\
&&\times (1-\sum_{j=1}^i \log |x_j|+\sum_{k=i+1}^n\log |x_k|)^D d^\times x_1\dots d^\times x_n
\eeqnan
converges absolutely.
\end{lemma}
We note that by a theorem of Silberger (\cite{silb}, page 149), the convergence of the integral above is reduced to the convergence of
\beqn\label{silbint}
\int_{M_{n+1}^+}\mu(m)\int_{K_{n+1}\times K_{n+1}} \Phi_{\varphi_{n+2}}(k_1mk_2) \Phi_{\varphi_{n+1}}(k_1mk_2) dk_1dk_2dm
\eeqn
where
$$
\mu(m) := \operatorname{Vol}(K_{n+1}mK_{n+1})/\operatorname{Vol}(K_{n+1}).
$$
In fact, if either of these two integrals converge, they are equal.  Furthermore, we have a positive constant $A$ such that
\beqn\label{mubound}
|\mu(m)|\leq A\cdot\delta_{n+1}^{-1}(m)
\eeqn
for all $m\in M_{n+1}^+$ (see \cite{silb}).

Since $\Phi_{\varphi_{n+1}}$ and $\Phi_{\varphi_{n+2}}$ are matrix coefficients for tempered representations, they satisfy the so-called weak inequality, which means that there are positive constants $B,C$ such that for all $g_i\in G_i$,
\beqn\label{matbound}
|\Phi_{\varphi_i}(g_i)| \leq B\cdot|\Xi_i(g_i)|(1+\sigma(g_i))^C
\eeqn
It is known that that there are positive constants $B',C'$
\beqn\label{hcbound}
|\Xi_i(m)|\leq B'\delta_i^{1/2}(m)(1+\sigma(m))^{C'}
\eeqn
for all $m\in M_i^+$ (see \cite{silb}).

Combining lines \ref{mubound}, \ref{matbound}, and \ref{hcbound}, we see that the convergence of the integral in \ref{silbint} is reduced to the convergence of
$$
\int_{M_{n+1}^+} \delta_{n+1}^{-1/2}(m)\delta_{n+2}^{1/2}(m)(1+\sigma(m))^Ddm
$$
for some positive constant $D$.

When $E$ is a field, we have $M_{n+1}^+ = T_{n+1}^+\times U(V_{n+1,\text{an}})$ and $U(V_{n+1,\text{an}})$ is compact, so the convergence of the integral above is reduced to the convergence of
$$
\int_{T_{n+1}^+} \delta_{n+1}^{-1/2}(t)\delta_{n+2}^{1/2}(t)(1+\sigma(t))^Ddt.
$$
Finally, this is reduced to the convergence of $$\int_{|x_1|\leq |x_2|\leq\dots\leq |x_{r_{n+1}}|\leq 1} |x_1x_2\cdot\dots\cdot x_{r_{n+1}}|^{1/2}\left(1-\sum_{i=1}^{r_{n+1}}\log|x_i|\right)^{D}\ d^\times x_1d^\times x_2\dots d^\times x_{r_{n+1}}
$$
which follows from Lemma \ref{calc}.

Now we suppose that $E=F\times F$.  In this case, we have $G_i\cong GL_i(F)$.  Note that in this case, we have $T^+_i = M^+_i$; however, we no longer have $T_{n+1}^+\subset T_{n+2}^+$.  With the right choice of bases, we can view
$$
T_{n+1}^+ = \{\operatorname{diag}(x_1, x_2, \dots, x_{n+1}, 1):x_i\in F^\times, |x_i|\leq |x_{i+1}|\}
$$
and
$$
T_{n+2}^+ = \{\operatorname{diag}(x_1, x_2, \dots, x_{n+2}):x_i\in F^\times, |x_i|\leq |x_{i+1}|\}
$$
both as subgroups of $GL_{n+2}(F)$.

For the moment, set $m := \operatorname{diag}(x_1, x_2, \dots, x_{n+1}, 1)$.  We see that $m\in T_{n+2}^+$ if and only if $|x_{n+1}|\leq 1$.  If $m\notin T_{n+2}^+$, then we cannot directly apply the bound on $\Xi_{n+2}$ that we used previously.  To remedy this, let $i$ be such that $|x_i|\leq 1\leq |x_{i+1}|$, and set $m':=\operatorname{diag}(x_1,\dots,x_i,1,x_{i+1},\dots,x_{n+1})$.  Then we see that there are $k_1,k_2\in K_{n+2}$ such that $m=k_1m'k_2$, and therefore $\Xi_{n+2}(m) = \Xi_{n+2}(m')$.  Furthermore, we see that $m'\in T_{n+2}^+$, and therefore the bound we used previously applies.

So, we see that in this case we are reduced to checking the convergence of
\beqnan
&&\int_{|x_1|\leq\dots\leq |x_{n+1}|\leq 1}|x_1\dots x_{n+1}|^{1/2}\left(1-\sum_{j=1}^{n+1} \log |x_j|\right)^Dd^\times{x_1}\dots d^\times x_{n+1}\\
&&+\int_{|x_1|\leq\dots\leq |x_n|\leq 1\leq |x_{n+1}|} |x_1\dots x_n x_{n+1}^{-1}|^{1/2}\\
&&\times\left(1-\sum_{j=1}^{n} \log |x_j|+\log |x_{n+1}|\right)^D d^\times{x_1}\dots d^\times x_{n+1}\\
&&\vdots\\
&&+\int_{1\leq |x_1|\leq\dots\leq |x_{n+1}|}|x_1^{-1}\dots x_{n+1}^{-1}|^{1/2}\\
&&\left(1+\sum_{j=1}^{n+1} \log |x_j|\right)^Dd^\times{x_1}\dots d^\times x_{n+1}.
\eeqnan
The convergence of each of these integrals follows from Lemma \ref{calc}.
\end{proof}

\subsection{Calculation of integrals in the unramified case}
In what follows, we assume $v\notin S$.  The purpose of the remainder of this chapter is to explicitly compute $\zeta(\Xi,\xi)$ and $S_{\Xi^{-1}, \xi^{-1}}(1)$ for such $v$.

Away from places in $S$, we are either in the non-split, but quasi-split case, or the split case.  In the non-split case, $E/F$ is an unramified quadratic extension of $p$-adic fields.  In the split case, we have $E=F\oplus F$.  Let $V_{n+1}\subset V_{n+2}$ be (quasi-split) hermitian spaces over $E$, and let $G_{n+1}\subset G_{n+2}$ be the associated unitary groups.\footnote{The mildly strange choice of notation $n+1$ and $n+2$ will be explained later.}

The calculation of the integrals $\mathcal{P}'(f_{\pi_{n+1}}, f_{\pi_{n+2}})$ will involve splitting them into the product of two pieces, each of which will be computed independently.  But first, we give a description of the representations $\pi_i$; away from $S$, these have a concrete description.

We set $l_i:=\lfloor i/2\rfloor$.  Let $\xi_1,\dots , \xi_{l_{n+1}}$ and $\Xi_{1},\dots, \Xi_{l_{n+2}}$ be unramified characters of $E^\times$, and let $\Xi_0$ and $\xi_0$ be unramified characters of $E_1$, where
$$
E_1 := \{x\in E^\times: N_{E/F}(x) = 1\}
$$
and $N_{E/F}$ is the relative norm map.  Note that if $E$ is a field, then $E_1$ is compact, and $\xi_0$ and $\Xi_0$ are trivial (since they're unramified).  However, if $E$ is not a field, then $E_1\cong F^\times$, and $\Xi_0,\xi_0$ need not be trivial.  In the split case, for $i\geq 1$, we have $\Xi_i = (\mu_i, \nu_i)$ and $\xi_i = (\theta_i, \phi_i)$, where each $\mu_i,\nu_i,\theta_i,\phi_i$ are all unramified characters of $F^\times$.

If $n+1$ is odd, then we have $T_{n+1}\cong E_1\times (E^\times)^{l_{n+1}}$, and $T_{n+2}\cong (E^\times)^{l_{n+2}}$.  Then we define characters $\xi := (\xi_0, \xi_1,\dots, \xi_{l_{n+1}})$ and $\Xi:= (\Xi_1,\dots, \Xi_{l_{n+2}})$ of $T_{n+1}$ and $T_{n+2}$, respectively.  If $n+1$ is even, then $T_{n+1}\cong (E^\times)^{l_{n+1}}$ and $T_{n+2}\cong E_1\times (E^\times)^{l_{n+2}}$ and we have $\xi=(\xi_1,\dots, \xi_{l_{n+1}})$ and $\Xi=(\Xi_0, \Xi_1,\dots, \Xi_{l_{n+2}})$.

Let $B_i=T_iN_i$ be Borel subgroups, where the $N_i$ are unipotent radicals.  Then we can view $\xi$ and $\Xi$ as characters of $B_{n+1}$ and $B_{n+2}$ by extending them by $1$ to $N_i$.

Away from $S$, we may assume that
\beqnan
\pi_{n+1} &\cong& \operatorname{Ind}_{B_{n+1}}^{G_{n+1}} \xi\\
\pi_{n+2} &\cong& \operatorname{Ind}_{B_{n+2}}^{G_{n+2}} \Xi
\eeqnan

Note that in the split case, we have $G_{i}\cong GL_i(F)$, and we have
$$
{\xi} = (\theta_1,\dots,\theta_{l_{n+1}}, \mu_0, \phi_{l_{n+1}}^{-1},\dots, \phi_1^{-1})\text{ or }(\theta_1,\dots,\theta_{l_{n+1}}, \phi_{l_{n+1}}^{-1},\dots, \phi_1^{-1})
$$
and 
$$
{\Xi} = (\mu_1,\dots, \mu_{l_{n+2}}, \nu_{l_{n+2}}^{-1}, \dots, \nu_1^{-1})\text{ or }(\mu_1,\dots, \mu_{l_{n+2}}, \Xi_0, \nu_{l_{n+2}}^{-1}, \dots, \nu_1^{-1}),
$$
according to whether $n$ is odd or even, respectively.

We choose the $\mathcal{B}_{\pi_i}$ so that the $\Phi_{f_{\pi_i}}$ are the spherical matrix coefficients normalized such that $\Phi_{f_{\pi_i}}(k_i)=1$ for all $k_i\in K_i$.  Then we have the following explicit formulae:
$$
\Phi_{f_{\pi_i}}(g_i) = \int_{K_{i}} f_{\pi_i}(k_{i}g_{i})\ dk_{i}
$$
for all $g_i\in G_i$.

We consider the function
$$
F(g_{n+2}):= \int_{G_{n+1}} \Phi_{f_{\pi_{n+2}}}(g_{n+2}^{-1}g_{n+1})\Phi_{f_{\pi_{n+1}}}(g_{n+1})\ dg_{n+1}
$$
for $g_{n+2}\in G_{n+2}$.  We're interested in computing $F(1)$.  While it may seem a bit silly to invent this function if we're only interested in its value at the identity, the reason for this definition will become clear soon.

Let $G_{n+1}^\Delta$ denote the diagonal copy of $G_{n+1}$ in $G_{n+2}\times G_{n+1}$.  Then we note that ${G_{n+2}\times G_{n+1}/ G_{n+1}^\Delta}$ is a spherical variety.  This means that it has a unique open orbit under the action of $B_{n+2}\times B_{n+1}$ on the left.  Also, we note that
$$
B_{n+2}\times B_{n+1}\backslash G_{n+2}\times G_{n+1} / G_{n+1}^\Delta\cong B_{n+2}\backslash G_{n+2} / B_{n+1}.
$$
Recalling that we have Iwasawa decompositions
$$
G_i = B_iK_i,
$$
we let $\eta_{n+2}\in K_{n+2}$ be a representative for the open $B_{n+2}\times B_{n+1}$ orbit on $G_{n+2}$.  Let $Y_{\Xi,\xi}$ be the function on $G_{n+2}$ determined by the following conditions:
\begin{enumerate}
\item $Y_{\Xi,\xi}(b_{n+2}g_{n+2}b_{n+1}) = (\Xi^{-1}\delta_{n+2}^{1/2})(b_{n+2})(\xi\delta_{n+1}^{-1/2})(b_{n+1})Y_{\Xi,\xi}(g_{n+2})$ for all $b_i\in B_i$.
\item $Y_{\Xi,\xi}(\eta_{n+2})=1$
\item $Y_{\Xi,\xi}(g_{n+2}) = 0$ if $g_{n+2}\not\in B_{n+2}\eta_{n+2} B_{n+1}$.
\end{enumerate}

Here, $\delta_i$ denotes the modulus character for $B_i$.  We define the following two functions on $G_{n+2}$:
\beqnan
T_{\Xi,\xi}(g_{n+2}) &:=& \begin{cases} \int_{G_{n+1}} f_{\pi_{n+2}}(g_{n+2} g_{n+1})f_{\pi_{n+1}}(g_{n+1})\ d_{g_{n+1}} & \mbox{if } g_{n+2} \in B_{n+2}\eta_{n+2}B_{n+1}\\ 0 & \mbox{otherwise}\end{cases}\\
S_{\Xi,\xi}(g_{n+2}) &:=& \int_{K_{n+2}}\int_{K_{n+1}} Y_{\Xi,\xi}(k_{n+2} g_{n+2}^{-1} k_{n+1})\ dk_{n+1}d_{k_{n+2}}.
\eeqnan
We have $T_{\Xi,\xi}(g_{n+2}) = T_{\Xi,\xi}(\eta_{n+2})Y_{\Xi^{-1},\xi^{-1}}(g_{n+2})$ since $T_{\Xi,\xi}$ satisfies conditions $(1)$ and $(3)$ for $Y_{\Xi^{-1},\xi^{-1}}$.  Also, we note that $T_{\Xi,\xi}(\eta_{n+2})$ does not depend on the choice of representative $\eta_{n+2}$.  So we denote $\zeta(\Xi,\xi) := T_{\Xi,\xi}(\eta_{n+2})$.

Relating these to the integral that is the subject of this chapter, we have
\beqnan
\lefteqn{F(g_{n+2})}\\
&:=&\int_{G_{n+1}} \Phi_{f_{\pi_{n+2}}}(g_{n+2}^{-1}g_{n+1})\Phi_{f_{\pi_{n+1}}}(g_{n+1})\ dg_{n+1}\\ &=& \int_{G_{n+1}}\int_{K_{n+2}}\int_{K_{n+1}} f_{\pi_{n+2}}(k_{n+2}g_{n+2}^{-1}g_{n+1})f_{\pi_{n+1}}(k_{n+1}g_{n+1})\ dk_{n+1} dk_{n+2} dg_{n+1}\\
&=& \int_{G_{n+1}}\int_{K_{n+2}}\int_{K_{n+1}} f_{\pi_{n+2}}(k_{n+2}g_{n+2}^{-1}k_{n+1}g_{n+1})f_{\pi_{n+1}}(g_{n+1})\ dk_{n+1} dk_{n+2} dg_{n+1}\\
&=& \int_{K_{n+2}}\int_{K_{n+1}} T_{\Xi,\xi}(k_{n+2}g_{n+2}^{-1}k_{n+1})\ dk_{n+1}dk_{n+2}\\
&=& T_{\Xi,\xi}(\eta_{n+2}) \int_{K_{n+2}}\int_{K_{n+1}} Y_{\Xi^{-1},\xi^{-1}}(k_{n+2}g_{n+2}^{-1}k_{n+1})\ dk_{n+1}dk_{n+2}\\
&=& \zeta(\Xi,\xi) S_{\Xi^{-1},\xi^{-1}}(g_{n+2}).
\eeqnan
Regarding convergence, we note (as in \cite{rgp}), that $F(g_{n+2})$ is convergent for $\Xi$ and $\xi$ sufficiently close to the unitary axis.  (Indeed, Proposition \ref{conv} holds for such $\Xi,\xi$.)  So, we see that $T_{\Xi,\xi}(k_{n+2}g_{n+2}^{-1}k_{n+1})$ is convergent for almost all $k_{n+2},k_{n+1}$ such that $k_{n+2}g_{n+2}^{-1}k_{n+1}\in B_{n+2}\eta_{n+2} B_{n+1}$.  But since $T_{\Xi,\xi}(g)$ is convergent for some $g\in B_{n+2}\eta_{n+2} B_{n+1}$ if and only if it is convergent for all $g\in B_{n+2}\eta_{n+2} B_{n+1}$, we see that $T_{\Xi,\xi}(\eta_{n+2})$ is convergent.

\subsubsection{Calculation of $\zeta(\Xi,\xi)$}
In this section, we will actually have occasion to consider three different hermitian spaces simultaneously; let $V_n\subset V_{n+1}\subset V_{n+2}$ be hermitian spaces of dimensions $n, n+1$ and $n+2$ over $E$, and let $G_i$ be their respective unitary groups.  Until specified otherwise, we do not distinguish between the non-split and split cases.  That is, we do not view the split unitary groups as general linear groups over $F$, but still as isometry groups of hermitian spaces over $E$.

If $n=2m$, then we write
\beqnan
V_n &=& \langle e_1, e_2, \dots, e_m, f_1, f_2, \dots, f_m\rangle\\
V_{n+1} &=& V_n \oplus \langle e_{m+1} + f_{m+1}\rangle\\
V_{n+2} &=& V_{n+1} \oplus \langle e_{m+1} -  f_{m+1}\rangle
\eeqnan
where all of the $e_i, f_i$ are isotropic vectors.  Furthermore, if we let $h$ denote the hermitian form on $V_{n+2}$, we have $h(e_i, f_j) = \delta_{ij}$.  Then we have the Borel subgroups $B_i \subset G_i$ where $B_i := \operatorname{Stab}_{G_i} \mathcal{F}_i$ where the $\mathcal{F}_i$ are the following flags of isotropic spaces:
\beqnan
\mathcal{F}_n &:=& \{0\}\subsetneq \langle e_1\rangle \subsetneq \langle e_1, e_2\rangle \subsetneq \dots \subsetneq \langle e_1,\dots e_m\rangle\subsetneq V_n\\
\mathcal{F}_{n+1} &:=& \{0\}\subsetneq \langle e_1\rangle \subsetneq \langle e_1, e_2\rangle \subsetneq \dots \subsetneq \langle e_1,\dots e_m\rangle\subsetneq V_{n+1}\\
\mathcal{F}_{n+2} &:=& \{0\}\subsetneq \langle e_{m+1}\rangle \subsetneq \langle e_{m+1}, e_1\rangle \dots\subsetneq \langle e_{m+1}, e_1,\dots e_m\rangle\subsetneq V_{n+2}.
\eeqnan
If $n=2m+1$, then we have
\beqnan
V_n &=& \langle e_1, e_2, \dots, e_m, f_1, f_2, \dots, f_m, e_{m+1} + f_{m+1}\rangle\\
V_{n+1} &=& V_n\oplus \langle e_{m+1} - f_{m+1}\rangle\\
V_{n+2} &=& V_{n+1}\oplus \langle e_{m+2}+f_{m+2}\rangle.
\eeqnan
Again, all of the $e_i, f_i$ are isotropic, and we have $h(e_i, f_j)=\delta_{ij}$.  In this case, the flags are:
\beqnan
\mathcal{F}_n &:=& \{0\}\subsetneq\langle e_1\rangle \subsetneq \langle e_1, e_2\rangle \subsetneq \dots \subsetneq \langle e_1,\dots e_m\rangle\subsetneq V_n\\
\mathcal{F}_{n+1} &:=&  \{0\}\subsetneq\langle e_{m+1}\rangle \subsetneq \langle e_{m+1}, e_1\rangle \subsetneq \dots \subsetneq \langle e_{m+1},e_1\dots e_m\rangle\subsetneq V_{n+1}\\
\mathcal{F}_{n+2} &:=& \{0\}\subsetneq \langle v\rangle \subsetneq \langle v, e_1\rangle\subsetneq \dots \subsetneq \langle v, e_1, \dots, e_m\rangle\subsetneq V_{n+2}
\eeqnan
where $v=e_{m+2}+f_{m+2} + e_{n+1} - f_{n+1}.$

Let $K_n\subset K_{n+1}\subset K_{n+2}$ be hyperspecial maximal compact subgroups of the $G_i$, such that $G_i = B_iK_i$.  For $i=n,n+1$, we consider the action of $B_{i+1}\times B_i$ on $G_{i+1}$ by $(b_{i+1}, b_i)\cdot g_{i+1} := b_{i+1}g_{i+1}b_i^{-1}$.  We let $\eta_{i+1}\in K_{i+1}$ be a representative for the unique open dense orbit.  We now prove a proposition that relates the representatives of the open orbits of $B_{n+2}\times B_{n+1}$ on $G_{n+2}$ and $B_{n+1}\times B_n$ on $G_{n+1}$.
\begin{proposition}  $\eta_{n+1}^{-1}\in K_{n+1}\subset K_{n+2}$ is a representative for the open orbit of $B_{n+2}\times B_{n+1}$ acting on $G_{n+2}$.
\end{proposition}

\begin{proof}
To show that $\eta_{n+1}^{-1}$ is a representative for the open $B_{n+2}\times B_{n+1}$ orbit in $G_{n+2}$, it suffices to check that $B_{n+2}\cap \eta_{n+1}^{-1} B_{n+1}\eta_{n+1}$ is trivial.

Suppose that $n=2m$.  First, we show that $B_{n+2}\cap G_{n+1}\subset B_n$.  To see this, take any $b\in B_{n+2}\cap G_{n+1}$.  Since $b\in G_{n+1}$, we know that $b$ fixes $e_{m+1}-f_{m+1}$.  But since $b\in B_{n+2}$, we know that
$$
b\cdot e_{m+1} = ae_{m+1}
$$
for some $a\in E^\times$ and
$$
b\cdot f_{m+1} = \overline{a}^{-1}f_{m+1} + \sum_{i=1}^{m+1} a_ie_i+\sum_{i=1}^m c_if_i
$$
 for $a_i, c_i\in E$.  This means that we have
$$
b\cdot(e_{m+1} - f_{m+1}) = ae_{m+1} - \overline{a}^{-1}f_{m+1} + \sum_{i=1}^{m+1} a_ie_i+\sum_{i=1}^m c_if_i.
$$
But since $b$ fixes $e_{m+1} - f_{m+1}$, this means that $a=1$ and $a_i=c_i=0$ for all $i$.  So we see that $b$ fixes both $e_{m+1},f_{m+1}$ and is therefore in $G_n$.  But since it preserves $\mathcal{F}_{n+2}$ and $e_{m+1}$, we see that it preserves $\mathcal{F}_n$, and is therefore in $B_n$.

Now we have that 
$$
B_{n+2} \cap \eta_{n+1}^{-1}B_{n+1}\eta_{n+1} = B_{n+2}\cap \eta_{n+1}^{-1} B_{n+1}\eta_{n+1}\cap G_{n+1} = B_n \cap \eta_{n+1}^{-1}B_{n+1}\eta_{n+1} = \{1\}.
$$

Now suppose that $n=2m+1$.  Again, we will show that $B_{n+2}\cap G_{n+1}\subset B_n$.  First, we show that $B_{n+2}\cap G_{n+1}\subset G_n$.  To see this, we need only show that $b\in B_{n+2}\cap G_{n+1}$ fixes $e_{m+1}-f_{m+1}$.  Well, since $b\in G_{n+1}$, we know that it fixes $e_{m+2}+f_{m+2}$.  Also, since $b\in B_{n+2}$, we know that it scales $v$.  So we have:
\beqnan
b\cdot v &=& b(e_{m+2}+f_{m+2}+e_{m+1}-f_{m+1})\\
&=& b(e_{m+2}+f_{m+2}) + b(e_{m+1} - f_{m+1})\\
&=& e_{m+2}+f_{m+2} + b(e_{m+1}-f_{m+1})\\
&=& av\ \text{ for some } a\in E^\times\\
&=& ae_{m+2}+af_{m+2} + ae_{m+1} - af_{m+1}.
\eeqnan
But since $b\in G_{n+1}$, we know that $b\cdot (e_{m+1} - f_{m+1})\in V_{n+1}$, and so $a=1$.  This gives that $b$ fixes $e_{m+1} - f_{m+1}$ and is therefore in $G_{n+1}$.  We also see that it fixes $v$.  This, together with the fact that it preserves $\mathcal{F}_{n+2}$ immediately gives that it preserves $\mathcal{F}_n$, and is therefore in $B_n$.  Then just as in the case where $n=2m$, we see that
$$
B_{n+2} \cap \eta_{n+1}^{-1}B_{n+1}\eta_{n+1} = B_{n+2}\cap \eta_{n+1}^{-1} B_{n+1}\eta_{n+1}\cap G_{n+1} = B_n \cap \eta_{n+1}^{-1}B_{n+1}\eta_{n+1} = \{1\}.
$$
\end{proof}

Now, we take $V$ to be an $n$-dimensional hermitian space with associated hermitian form $\langle\cdot,\cdot\rangle_V$.  Let $V^-$ be the same space as $V$, but with the hermitian form $-\langle\cdot,\cdot\rangle_V$.  Also, let $Q = \operatorname{Span}\{ e, f\}$ be a hyperbolic plane, where $e$ and $f$ are isotropic vectors and $\langle e,f\rangle_Q = 1$.  We also consider the hermitian space $W := V\oplus V^{-}\oplus Q$.

We will be considering the following groups defined over $F$:
\beqnan
G_n &:=& U(V)\cong U(V^-)\\
G_{n+1} &:=& U(V^-\oplus \langle e+f\rangle)\\
G_{n+2} &:=& U(V^-\oplus Q)\\
G &:=& U(W)
\eeqnan
where in the first line, we identity $V$ and $V^-$ as vector spaces (but not as hermitian spaces) via the identity map.

Note that we have inclusions $G_n\subset G_{n+1}\subset G_{n+2}$.  Let $B_n, B_{n+1}, B_{n+2}$ be Borel subgroups as in the beginning of this section.  Let $T_i, N_i$ be the corresponding tori and unipotent radicals, and let $K_i$ be hyperspecial maximal compact subgroups such that $G_i=B_iK_i$. Let $\tilde{\Xi}, \xi, \Xi$ be characters of $T_n,T_{n+1}, T_{n+2}$ respectively, where

\beqnan
\tilde{\Xi} &:=& (\Xi_1,\Xi_2,\dots, \Xi_{\lfloor n/2\rfloor})\\
\xi &:=& (\xi_1,\xi_2,\dots, \xi_{\lfloor (n+1)/2 \rfloor})\\
\Xi &:=& (\Xi_1,\dots, \Xi_{\lfloor (n+2)/2 \rfloor})
\eeqnan
and each $\Xi_i, \xi_i$ is an unramified character of $E^\times$.  Recall that $l:= \lfloor (n+2)/2\rfloor$.

As before, we extend these characters to $B_i$ by $1$ along $N_i$.  Denote by $\pi_i$ the corresponding unramified principal series representation, and let $f_{\pi_i}\in \pi_i$ be the corresponding spherical vector, normalized so that $f_{\pi_i}(k_i)=1$ for all $k_i\in K_i$.

Let $\iota:V\to V^-$ be the identity map.  Then we define the following subspace of $W$:
$$
V^\dagger := \operatorname{Span}\{v_i-\iota(v_i), e\}.
$$

Note that $V^\dagger$ is a maximal isotropic subspace of $W$.  Denote by $G\supset P := \operatorname{Stab}_G V^\dagger$ the corresponding maximal parabolic subgroup.  The Levi subgroup $M\subset P$ is isomorphic to $GL(V^\dagger)$.  Denote by $N\subset P$ the unipotent radical.  We consider the following induced representation of $G$:
$$
I(\Xi_l) := \operatorname{Ind}_{P}^G (\Xi_l\circ {\det}_{V^\dagger})
$$
where the induction is normalized.  Let $f_0\in I(\Xi_l)$ be the spherical vector, normalized so that $f_0(1)=1$.  We consider $f_0|_{G_n\times G_{n+2}}$, and denote this by $\tilde{f_0}$.

We define the following integral:
$$
\Lambda(f_0, f_{\pi_n})(g_{n+2}) := \int_{G_n} \tilde{f_0}(g_n, g_{n+2}) \pi_n(g_n)f_{\pi_n}\ dg_n.
$$
We remark that since $\Xi_l$ is unramified, we have $\Xi_l = |\cdot |_E^s$ for some $s\in\C$.  The integral $\Lambda$ converges for $\operatorname{Re}(s)>>0$, and we use analytic continuation to define $\Lambda$ elsewhere.

Note that

$$
\Lambda(f_0,f_{\pi_n})\in \operatorname{Ind}_{GL(V^\dagger\cap Q)\times G_n}^{G_{n+2}} \Xi_l\otimes \pi_n,
$$
and so by transitivity of induction, we may view $\Lambda(f_0,f_{\pi_n})$ as an element of $\pi_{n+2}$.

\begin{proposition}\label{prev} For any $k\in K_{n+2}$ and $n\geq 1$ we have
$$\Lambda(f_0, f_{\pi_n})(k) = \frac{L_E(1, BC(\pi_n)\otimes\Xi_l)}{\prod_{j=1}^n L_F(1+j, \Xi_l\chi_{E/F}^{n-j})}f_{\pi_n},$$
where $\chi_{E/F}$ is the quadratic character attached to the extension $E/F$ by class field theory.  Recall that in the case where $E$ is not a field, this is the trivial character.
\end{proposition}

\begin{proof}
First, we note that $\Lambda$ is constant as a function on $K_{n+2}$.  To see this, we simply note that $f_0$ is a spherical vector.

Now, we consider $\widehat{f_0} := f_0|_{G_n\times G_n}$.  Then we see that
$$
\widehat{f_0}\in \operatorname{Ind}_{P\cap (G_n\times G_n)}^{G_n\times G_n} \left(\frac{\delta_{P}}{\delta_{P\cap (G_n\times G_n)}}\right)^{1/2}(\Xi_l\circ \operatorname{det}_{V^\dagger})
$$

Now, we know that $$\delta_{P}(p) = |\operatorname{det}_{V^\dagger}(p)|^{n+1}$$ and $$\delta_{P\cap (G_n\times G_n)}(p) = |\operatorname{det}_{V^\dagger}(p)|^{n}$$ for $p\in P\cap (G_n\times G_n)$. So, by taking $G_n=G$ in Proposition 3 in \cite{lr},  we see that

$$
\Lambda(f_0,f_{\pi_n})(1) = \Lambda(\widehat{f_0}, f_{\pi_n})(1) = \frac{L_E(1, BC(\pi_n)\otimes\Xi_l)}{\prod_{j=1}^n L_F(1+j, \Xi_l\chi_{E/F}^{n-j})}f_{\pi_n}.
$$
(This is analogous to Theorem 1.1 on page 16 of \cite{l:orthog}.)
\end{proof}

We recall the action of $B_{i+1}\times B_i$ on $G_{i+1}$ by $(b_{i+1}, b_i)g_{i+1} := b_{i+1}g_{i+1}b_i^{-1}$.  As we mentioned before, there is a unique open and dense orbit under this action.  We let $\eta_{i+1}\in K_{i+1}$ be a representative of this orbit.

Using the previous proposition, we have the following inductive relationship between $\zeta(\Xi,\xi)$ and $\zeta(\xi,\tilde{\Xi})$:
\begin{proposition} For $n\geq 1$ we have
$$\zeta(\Xi,\xi) =  \frac{L_E(1/2, BC(\pi_{n+1})\otimes\Xi_{l})}{L_E(1, BC(\pi_n)\otimes\Xi_l)L_F(1, \chi_{E/F}^n\otimes\Xi_l)}\zeta(\xi,\tilde{\Xi}).$$
\end{proposition}
\begin{proof}  First, we note the following:
\beqnan
\lefteqn{\int_{G_{n+1}} \tilde{f_0}(1, g_{n+1}) \int_{G_n} f_{\pi_{n+1}}(\eta_{n+1}g_ng_{n+1})f_{\pi_n}(g_n)dg_ndg_{n+1}}\\
&=&\frac{L_E(1/2, BC(\pi_{n+1})\otimes\Xi_l)}{\prod_{j=1}^{n+1}L_F(j, \Xi_l\chi_{E/F}^{n+1-j})}\zeta(\xi,\tilde{\Xi})
\eeqnan
which follows from Proposition $3$ in \cite{lr}.\footnote{Note that in \cite{lr} the authors actually consider an integral over $G_{n+1}^\Delta\backslash (G_{n+1}\times G_{n+1})$, where $G_{n+1}^\Delta$ is the diagonally embedded copy of $G_{n+1}$.  By identifying this with $G_{n+1}$, we are led to consider the integral above instead.}  To see this, we consider the pairing $\mathcal{T}:\pi_{n+1}\otimes \pi_{n}\to\C$ given by
$$
\mathcal{T}(f_1, f_2) :=  \int_{G_n} f_1(\eta_{n+1}g_n)f_2(g_n)dg_n.
$$
Clearly, we have $\mathcal{T}(f_{\pi_{n+1}}, f_{\pi_n})=\zeta(\xi,\tilde{\Xi})$.  Now, note that $g_{n+1}\mapsto \tilde{f}_0(1,g_{n+1})$ is bi-invariant under $K_{n+1}$.  Choosing $dk$ so that $\int_{K_{n+1}}dk = 1$, we have
\beqnan
\lefteqn{\int_{G_{n+1}} \tilde{f}_0(1, g_{n+1}) \mathcal{T}(g_{n+1}f_{\pi_{n+1}}, f_{\pi_n})dg_{n+1}}\\ &=& \int_{G_{n+1}}\int_{K_{n+1}} \tilde{f}_0(1, kg_{n+1}) \mathcal{T}(g_{n+1} f_{\pi_{n+1}}, f_{\pi_n})dk dg_{n+1}\\
&=& \int_{G_{n+1}} \tilde{f}_0(1, g_{n+1})\int_{K_{n+1}}\mathcal{T}(k^{-1}g_{n+1} f_{\pi_{n+1}}, f_{\pi_n})dk dg_{n+1}.
\eeqnan
The inner integral gives a smooth linear form on $\pi_{n+1}$; more specifically,
$$
f\mapsto \int_{K_{n+1}} \mathcal{T}(k^{-1} f, f_{\pi_n}) dk
$$
gives a map $L\in \pi_{n+1}^\vee$.  In fact, $L(g_{n+1}f_{\pi_{n+1}})$ gives a matrix coefficient on $\pi_{n+1}$ which is bi-invariant under $K_{n+1}$.  So, we see that there is a constant $A$ such that
$$
L(g_{n+1}f_{\pi_{n+1}}) = A\cdot\mathcal{B}_{\pi_{n+1}}(g_{n+1}f_{\pi_{n+1}}, f_{\pi_{n+1}}).
$$
To compute $A$, we simply take $g_{n+1}=1$, and we obtain $\zeta(\xi,\tilde{\Xi})=A$.  Now we use the result in \cite{lr} to obtain the claim at the beginning of the proof.

Now, using Proposition \ref{prev}, we observe the following:
\beqnan
\lefteqn{\int_{G_{n+1}} \tilde{f}_0(1,g_{n+1})\int_{G_n} f_{\pi_{n+1}}(\eta_{n+1} g_n g_{n+1}) f_{\pi_n}(g_n) dg_ndg_{n+1}} \\
&=& \int_{G_{n+1}} \int_{G_n} f_{\pi_{n+1}}(\eta_{n+1} g_n g_{n+1}) \tilde{f}_0(g_n,g_ng_{n+1}) f_{\pi_n}(g_n) dg_ndg_{n+1}\\
&=& \int_{G_{n+1}}\int_{G_n} f_{\pi_{n+1}}(\eta_{n+1}g_{n+1})\tilde{f}_0(g_n, g_{n+1}) f_{\pi_n}(g_n) dg_n dg_{n+1}\\
&=& \int_{G_{n+1}} f_{\pi_{n+1}}(\eta_{n+1}g_{n+1}) \int_{G_n} \tilde{f}_0(g_n, g_{n+1}) f_{\pi_n}(g_n)dg_n dg_{n+1}\\
&=& \int_{G_{n+1}} f_{\pi_{n+1}}(\eta_{n+1}g_{n+1})\Lambda(f_0, f_{\pi_n})(g_{n+1})(1)dg_{n+1}\\
&=& \frac{L_E(1, BC(\pi_n)\otimes\Xi_l)}{\prod_{j=1}^n L_F(1+j, \Xi_l\chi_{E/F}^{n-j})}\int_{G_{n+1}}f_{\pi_{n+1}}(\eta_{n+1}g_{n+1}) f_{\pi_{n+2}}(g_{n+1}) dg_{n+1}\\
&=& \frac{L_E(1, BC(\pi_n)\otimes\Xi_l)}{\prod_{j=1}^n L_F(1+j, \Xi_l\chi_{E/F}^{n-j})}\int_{G_{n+1}}f_{\pi_{n+2}}(\eta_{n+1}^{-1}g_{n+1}) f_{\pi_{n+1}}(g_{n+1}) dg_{n+1}\\
&=& \frac{L_E(1, BC(\pi_n)\otimes\Xi_l)}{\prod_{j=1}^n L_F(1+j, \Xi_l\chi_{E/F}^{n-j})} \zeta(\Xi,\xi).
\eeqnan

Note that we've used the fact that $\eta_{n+1}^{-1}$ is a representative for the open $B_{n+2}\times B_{n+1}$-orbit in $G_{n+2}$.  Also, we remark that the calculation above is similar to that carried out for orthogonal groups in \cite{l:orthog}.

Combining this with the first identity mentioned completes the proof.
\end{proof}

By induction on $n$, this gives us the following in the non-split (but quasi-split) case:

\begin{corollary}
Suppose $E$ is a field.  If $n$ is even, then
\beqnan
\zeta(\Xi,\xi) &=& \prod_{1\leq i<j\leq n/2+1} L_E(1/2, \xi_i\Xi_j)L_E(1/2,\xi_i^{-1}\Xi_j)L_E(1,\Xi_i\Xi_j)^{-1}L_E(1,\Xi_i^{-1}\Xi_j)^{-1}\\
&& \times \prod_{1\leq i\leq j \leq n/2} L_E(1/2, \Xi_i\xi_j)L_E(1/2,\Xi_i^{-1}\xi_j)\\
&& \times \prod_{1\leq i<j\leq n/2} L_E(1,\xi_i\xi_j)^{-1} L_E(1,\xi_i^{-1}\xi_j)^{-1}\\
&&\times \prod_{i=1}^{n/2} L_E(1/2, \chi_{E/F} \xi_i)^{-1} L_E(1,\xi_i)^{-1}.
\eeqnan
If $n$ is odd, then
\beqnan
\zeta(\Xi,\xi) &=& \prod_{1\leq i\leq j\leq (n+1)/2} L_E(1/2, \xi_i\Xi_j)L_E(1/2, \xi_i^{-1}\Xi_j)\\
&& \times \prod_{1\leq i<j\leq (n+1)/2} L_E(1, \Xi_i\Xi_j)^{-1}L_E(1, \Xi_i^{-1}\Xi_j)^{-1} L_E(1/2, \Xi_i\xi_j)L_E(1/2, \Xi_i^{-1}\xi_j)\\
&& \times \prod_{1\leq i<j\leq (n+1)/2} L_E(1,\xi_i\xi_j)^{-1}L_E(1,\xi_i^{-1}\xi_j)^{-1}\\
&& \times\prod_{i=1}^{(n+1)/2} L_E(1/2,\chi_{E/F}\Xi_i)^{-1}L_E(1,\Xi_i)^{-1}.
\eeqnan
\end{corollary}
\begin{proof}
We check the base cases.  The inductive steps follow from the previous proposition.

The base case is computing $\zeta((\Xi_1), (\xi_0))$, where $\xi_0$ is the trivial character.  The two groups involved in this calculation are $G_2$ and $G_1$ with principal series representations $\pi_2$ and $\pi_1$ respectively.  Let $f_{(\Xi_1)}$ be the normalized spherical vector in $\pi_2$, and let $f_{(\xi_0)}$ be the normalized spherical vector in $\pi_1$.   Now, since $G_1=K_1$ is compact, we see that $f_{(\xi_0)}$ is the constant function equal to $1$.  Now, let $\eta_2\in K_2$ be a representative for the open $B_2\times B_1$ orbit in $G_2$.  Then we have that
\beqnan
\zeta((\Xi_1), (\xi_0)) &=& \int_{G_1} f_{(\Xi_1)}(\eta_2g_1) f_{(\xi_0)}(g_1)\ dg_1\\
&=& \int_{K_1} f_{(\Xi_1)}(\eta_2g_1)\ dg_1\\
&=& \int_{K_1}\ dg_1\text{ (since }\eta_2g_1\in K_2)\\
&=& 1.
\eeqnan
\end{proof}
We state the result in the split case as a separate corollary.
\begin{corollary}
Suppose that $E=F\times F$.  If $n$ is even, then
\beqnan
\zeta(\Xi,\xi) &=& \prod_{1\leq i<j\leq l_{n+2}} L_F(1/2, \theta_i\mu_j)_F(1/2, \phi_i^{-1}\mu_j)L_F(1/2, \theta_i^{-1}\nu_j)L_F(1/2, \phi_i\nu_j)\\
&&\times\prod_{1\leq i\leq j\leq l_{n+1}} L_F(1/2,\mu_i\theta_j)L_F(1/2, \nu_i^{-1}\theta_j)L_F(1/2,\mu_i^{-1}\phi_j)L_F(1/2, \nu_i\phi_j)\\
&&\times \prod_{i=1}^{l_{n+2}}L_F(1/2, \xi_0\mu_i)L_F(1/2, \xi_0^{-1}\nu_i)\\
&&\times \prod_{1\leq i<j\leq l_{n+2}} L_F(1, \mu_i^{-1}\mu_j)^{-1}L_F(1, \nu_i\mu_j)^{-1}L_F(1, \mu_i\nu_j)^{-1}L_F(1,\nu_i^{-1}\nu_j)^{-1}\\
&&\times\prod_{1\leq i<j\leq l_{n+1}} L_F(1, \theta_i^{-1}\theta_j)^{-1}L_F(1, \phi_i\theta_j)^{-1}L_F(1, \theta_i\phi_j)^{-1}L_F(1, \phi_i^{-1}\phi_j)^{-1}\\
&&\times\prod_{i=1}^{l_{n+2}}L_F(1, \mu_i\nu_i)^{-1}\prod_{i=1}^{l_{n+1}} L_F(1, \xi_0^{-1}\theta_i)^{-1}L_F(1, \xi_0\phi_i)^{-1}L_F(1, \theta_i\phi_i)^{-1}.
\eeqnan
If $n$ is odd, then
\beqnan
\zeta(\Xi,\xi) &=& \prod_{1\leq i\leq j\leq l_{n+2}} L_F(1/2, \theta_i\mu_j)L_F(1/2, \phi_i^{-1}\mu_j)L_F(1/2, \theta_i^{-1}\nu_j)L_F(1/2, \phi_i\nu_j)\\
&& \times\prod_{1\leq i<j\leq l_{n+1}} L_F(1/2, \mu_i\theta_j)L_F(1/2, \nu_i^{-1}\theta_j)L_F(1/2, \mu_i^{-1}\phi_j)L_F(1/2, \nu_i\theta_j)\\
&& \times\prod_{i=1}^{l_{n+1}} L_F(1/2, \Xi_0\theta_i)L_F(1/2, \Xi_0^{-1}\phi_i)\\
&&\times \prod_{1\leq i<j\leq l_{n+2}} L_F(1, \mu_i^{-1}\mu_j)^{-1}L_F(1, \nu_i\mu_j)^{-1}L_F(1, \mu_i\nu_j)^{-1}L_F(1,\nu_i^{-1}\nu_j)^{-1}\\
&&\times\prod_{1\leq i<j\leq l_{n+1}} L_F(1, \theta_i^{-1}\theta_j)^{-1}L_F(1, \phi_i\theta_j)^{-1}L_F(1, \theta_i\phi_j)^{-1}L_F(1, \phi_i^{-1}\phi_j)^{-1}\\
&&\times\prod_{i=1}^{l_{n+2}}L_F(1, \mu_i\nu_i)^{-1}L_F(1, \Xi_0^{-1}\mu_i)^{-1}L_F(1, \Xi_0\nu_i)^{-1}\prod_{i=1}^{l_{n+1}}L_F(1, \theta_i\phi_i)^{-1}.
\eeqnan
\end{corollary}
\begin{proof}
Again, we need only check the base case.  This time, we're computing $\zeta((\mu_1, \nu_1^{-1}), (\xi_0))$ or $\zeta((\theta_1, \phi_1^{-1}), (\Xi_0))$, depending on whether $n$ is even or odd, respectively.

We compute $\zeta((\theta_1, \phi_1^{-1}), (\Xi_0))$.  We have
$$
\zeta((\theta_1, \phi_1^{-1}), (\Xi_0)) = \int_{G_1} f_{(\theta_1, \phi_1^{-1})}(a) f_{(\Xi_0)}(a)\ da.
$$

Now, we now that $G_1\cong F^\times$.  Also, $G_1$ embeds in $G_2$ in the following way:
$$
F^\times\owns a\mapsto \left(\begin{matrix} a+1/2 & a-1/2\\ a-1/2 & a+1/2\end{matrix}\right)\in G_2.
$$

(Note that since $B_2\cap B_1=\{1\}$, we take $\eta_2$ to be trivial.)

Now, if $a\in \mathcal{O}_F^\times$, then we have that 
$$
\left(\begin{matrix} a+1/2 & a-1/2\\ a-1/2 & a+1/2\end{matrix}\right)\in K_2\supset K_1
$$
and therefore $f_{(\theta_1, \phi_1^{-1})}(a)=f_{(\Xi_0)}(a)=1$ in this case.

If $a\in \mathcal{O}_F-\mathcal{O}_F^\times$, then we have the following Iwasawa decomposition:
$$
\left(\begin{matrix} a+1/2 & a-1/2\\ a-1/2 & a+1/2\end{matrix}\right) = \left(\begin{matrix} -a&-1\\&1\end{matrix}\right) \left(\begin{matrix} -2 & -2\\ a-1/2&a+1/2\end{matrix}\right).
$$
If $a\notin\mathcal{O}_F$, then we have the following Iwasawa decomposition:
$$
\left(\begin{matrix} a+1/2 & a-1/2\\ a-1/2 & a+1/2\end{matrix}\right) = \left(\begin{matrix} 1/2 & a\\ & a\end{matrix}\right) \left(\begin{matrix} 2&-2\\ 1-\frac{1}{2a}& 1+\frac{1}{2a}\end{matrix}\right).
$$
Using this, we can realize the integral as a geometric series:
\beqnan
 \int_{G_1} f_{(\theta_1, \phi_1^{-1})}(a) f_{(\Xi_0)}(a)\ da &=& 1+\sum_{n=1}^\infty \left(\frac{\theta_1\Xi_0}{q_F^{1/2}}\right)^n + \left(\frac{\phi_1}{\Xi_0q_F^{1/2}}\right)^n\\
 &=& \frac{1-q_F^{-1}\theta_1\phi_1}{(1-q_F^{-1/2}\phi_1\Xi_0^{-1})(1-q_F^{-1/2}\theta_1\Xi_0)}\\
 &=& L_F(1/2, \phi_1\Xi_0^{-1})L_F(1/2, \theta_1\Xi_0)L_F(1, \theta_1\phi_1)^{-1}.
\eeqnan
\end{proof}

\subsubsection{Calculation of $S_{\Xi^{-1},\xi^{-1}}(1)$}
The calculation of $S_{\Xi^{-1}, \xi^{-1}}(1)$ follows from Michael Khoury's Ph.D. dissertation if the unitary groups are non-split (but still quasi-split).  If the groups are split, then the calculation follows from unpublished work of Kato, Murase, and Sugano.
\paragraph{Quasi-Split Case.}
Recall that in this case, $E$ is the quadratic unramified extension of $F$.  Let $\varpi$ be a uniformizer for $F$, which -- since $E$ is unramified over $F$ -- is also a uniformizer for $E$.  Also, we remind the reader that $l_n := \lfloor n/2\rfloor$.  Let $\xi$ and $\Xi$ be as before.  Denote by $A_i\subset T_i$ the maximal split tori.

We begin the calculation by defining some convenient members of \\ $\Z[q_E^{\pm 1/2}, \Xi_1(\varpi)^{\pm 1},\Xi_2(\varpi)^{\pm 1},\dots, \xi_1(\varpi)^{\pm 1},\xi_2(\varpi)^{\pm 1},\dots].$  (Recall that $q_E$ is the cardinality of the residue field of $E$.)  If $n+1$ is even (hereafter known as Case A), then $A_{n+1}\cong A_{n+2}\cong (F^\times)^{l_{n+1}}$, and we set:
\beqnan
b(\Xi, \xi)^{-1} &:=& \prod_{i=1}^{l_{n+1}} L_E(1/2, \xi_i)\prod_{1\leq i\leq j\leq l_{n+1}} L_E(1/2, \Xi_i\xi_j)L_E(1/2, \Xi_i\xi_j^{-1})\\&&\times\prod_{1\leq j<i\leq l_{n+1}} L_E(1/2, \Xi_i\xi_j)L_E(1/2, \Xi_i^{-1}\xi_j)\\
d_1(\Xi)^{-1} &:=& \prod_{i=1}^{l_{n+1}} L_E(0, \Xi_i^2)\prod_{1\leq i < j\leq l_{n+1}} L_E(0, \Xi_i\Xi_j)L_E(0, \Xi_i\Xi_j^{-1})\\
d_0(\xi)^{-1} &:=& \prod_{i=1}^{l_{n+1}} L_E(0, \xi_i)\prod_{1\leq i<j\leq l_{n+1}} L_E(0, \xi_i\xi_j)L_E(0, \xi_i\xi_j^{-1}).
\eeqnan
We remark that while we've given formulae for $b^{-1}, d_1^{-1}, d_0^{-1}$ above, it's actually $b,d_1,d_0$ that are in the ring $\Z[q_E^{\pm 1/2}, \Xi_1(\varpi)^{\pm 1},\Xi_2(\varpi)^{\pm 1},\dots, \xi_1(\varpi)^{\pm 1},\xi_2(\varpi)^{\pm 1},\dots].$

If $n+1$ is odd (hereafter known as Case B), then we have $A_{n+1}\cong (F^\times)^{l_{n+1}}$ and $A_{n+2}\cong (F^\times)^{l_{n+2}}$.  We set:
\beqnan
b(\Xi, \xi)^{-1} &:=& \prod_{i=1}^{l_{n+2}} L_E(1/2, \Xi_i)\prod_{1\leq i\leq j\leq l_{n+1}} L_E(1/2, \Xi_i\xi_j)L_E(1/2, \Xi_i\xi_j^{-1})\\&&\times\prod_{1\leq j<i\leq l_{n+2}} L_E(1/2, \Xi_i\xi_j)L_E(1/2, \Xi_i^{-1}\xi_j)\\
d_1(\Xi)^{-1} &:=& \prod_{i=1}^{l_{n+2}} L_E(0, \Xi_i)\prod_{1\leq i < j\leq l_{n+2}} L_E(0, \Xi_i\Xi_j)L_E(0, \Xi_i\Xi_j^{-1})\\
d_0(\xi)^{-1} &:=& \prod_{i=1}^{l_{n+1}} L_E(0, \xi_i^2)\prod_{1\leq i<j\leq l_{n+1}} L_E(0, \xi_i\xi_j)L_E(0, \xi_i\xi_j^{-1}).
\eeqnan
Recall that the local $L_E$ factors are defined as $L_E(s, \chi):= \big(1-q_E^{-s}\chi(\varpi)\big)^{-1}$.

We also define $$c(\Xi, \xi) := \frac{b(\Xi, \xi)}{d_1(\Xi)d_0(\xi)}$$ as an element of $\Q(q_E^{1/2}, \Xi_1(\varpi), \dots, \xi_1(\varpi),\dots)$. Let $W_{n+2}$ and $W_{n+1}$ be the Weyl groups $W(G_{n+2}, A_{n+2})$ and $W(G_{n+1}, A_{n+1})$.  If $n+1$ is even, both of these groups are isomorphic to $(\Z/2\Z)^{l_{n+1}}\rtimes S_{l_{n+1}}$.  If $n+1$ is odd, then $W_{n+2}\cong (\Z/2\Z)^{l_{n+2}}\rtimes S_{l_{n+2}}\not\cong W_{n+1}\cong (\Z/2\Z)^{l_{n+1}}\rtimes S_{l_{n+1}}$.  These groups act on elements of $A_i$ (and therefore unramified characters of $A_i$) by permutation (the $S_l$ factor) and inversion (the $(\Z/2\Z)^l$) factor.

Finally, define $$A_{\Xi, \xi} := \sum_{w'\in W_{n+2}, w\in W_{n+1}} c(w'\Xi, w\xi)$$ as an element of $\Q(q_E^{1/2}, \Xi_1(\varpi), \dots, \xi_1(\varpi),\dots)$.

Theorem 11.4 in \cite{khoury} says the following:
\begin{theorem} Let $w_\ell\in W_{n+1}$ and $w_\ell'\in W_{n+2}$ be the long elements, and let $\mathcal{B}_i\subset K_i$ be Iwahori subgroups.  (Recall that the $K_i$ were fixed in defining $S_{\Xi^{-1}, \xi^{-1}}$.)  Then
$$
S_{\Xi^{-1},\xi^{-1}}(1) = \zeta(\Xi^{-1},\xi^{-1})q_E^{l(w_\ell)+l(w_\ell')}\operatorname{Vol}(\mathcal{B}_{n+1})\operatorname{Vol}(\mathcal{B}_{n+2})A_{\Xi^{-1},\xi^{-1}}.  
$$
Here, the volumes are computed with respect to the Haar measures which give the $K_i$ volume $1$.
\end{theorem}

Having introduced them already, we remind the reader that $\mathcal{B}_i = N_{i,(1)}^- T_{i,(0)}N_{i,(0)}$, where $N_{i,(0)} = N_i\cap K_i$, $T_{i,(0)} = T_i\cap K_i$, $N_i^-$ is the unipotent radical of the parabolic opposite that of $N_i$, and $N_{i,(1)}^-$ is the subgroup of $N_{i}^-\cap K_i$ whose elements' off-diagonal entries lie in the ideal generated by $\varpi$.  The volumes are
\beqnan
\operatorname{Vol}(\mathcal{B}_i) = \frac{\prod_{j=1}^i (q_F-(-1)^j)}{\prod_{j=1}^i(q_F^j-(-1)^j)}.
\eeqnan

\begin{lemma} $A_{\Xi, \xi}$ is independent of $\Xi$ and $\xi$.
\end{lemma}

\begin{proof}
Adapting the proof from \cite{rgp}, we first consider Case A.  We define the following Weyl vectors:
\beqnan
\rho_{n+2} &:=& (l_{n+1}, l_{n+1}-1, \dots, 1)\\
\rho_{n+1} &:=& (l_{n+1}-1/2, l_{n+1}-3/2,\dots, 1/2).
\eeqnan
Note that $\rho_{n+1}$ is half the sum of the positive roots of type $B_{l_{n+1}}$, while $\rho_{n+2}$ is half the sum of the positive roots of type $C_{l_{n+1}}$.  In what follows, we use the notation $\Xi^\rho = \prod\Xi_i^{\rho_i}.$

We define more members of $\Z[q_E^{\pm 1/2}, \Xi_1(\varpi),\dots,\xi_1(\varpi),\dots]$.  We set:
\beqnan
\mathcal{D}_\Xi &:=& \Xi^{-\rho_{n+2}}d_1(\Xi) = \sum_{w'\in W_{n+2}}\operatorname{sgn}(w')\cdot (w'\Xi)^{-\rho_{n+2}}\\
\mathcal{D}_\xi &:=& \xi^{-\rho_{n+1}}d_0(\xi) = \sum_{w\in W_{n+1}}\operatorname{sgn}(w)\cdot (w\xi)^{-\rho_{n+1}}.
\eeqnan
Then we see that $\mathcal{D}_{w'\Xi} = \operatorname{sgn}(w')\mathcal{D}_\Xi$ and $\mathcal{D}_{w\xi} = \operatorname{sgn}(w)\mathcal{D}_\xi.$  We introduce one more member of $\Z[q_E^{\pm 1/2}, \Xi_1(\varpi),\dots,\xi_1(\varpi),\dots]$.  Setting: $$B_{\Xi,\xi} := \Xi^{-\rho_{n+2}}\xi^{-\rho_{n+1}}b(\Xi,\xi)$$ we see that
\beqn\label{Asum}
A_{\Xi,\xi} = (\mathcal{D}_\Xi\mathcal{D}_\xi)^{-1}\sum_{w'\in W_{n+2}, w\in W_{n+1}} \operatorname{sgn}(w)\operatorname{sgn}(w') B(w'\Xi,w\xi).
\eeqn

Now write
$$
B_{\Xi,\xi}=\sum_{\lambda\in \Z^{l_{n+1}},\mu\in (\frac{1}{2}\Z)^{l_{n+1}}} c_{\lambda,\mu}\Xi^\lambda\xi^\mu
$$
for some coefficients $c_{\lambda,\mu}\in \Z[q_E^{\pm1/2}]$ (almost all of which are $0$, of course).  We say a monomial is \emph{regular} if its stabilizer under the action of the Weyl group is trivial; otherwise we call a monomial \emph{singular}.  We show that all regular monomials in $B_{\Xi,\xi}$ are in the orbit of $\Xi^{\rho_{n+2}}\xi^{\rho_{n+1}}.$

Note that it is sufficient to show that that we have $|\lambda_i|\leq l_{n+1}$ and $|\mu_i|\leq l_{n+1}-1/2$ and that none of the $\mu_i$ are integral.  (All such monomials are either singular or in the orbit of $\Xi^{\rho_{n+2}}\xi^{\rho_{n+1}}.$)

Note that
\beqnan
B_{\Xi,\xi} &=& \prod_{1\leq j\leq l_{n+1}} (\xi_j^{-1/2}-q_E^{-1/2}\xi_j^{1/2})\prod_{1\leq i,j\leq l_{n+1}} (1-q_E^{-1/2}\Xi_i\xi_j)\\ &&\times\prod_{1\leq i\leq j\leq l_{n+1}} (\Xi_i^{-1} - q_E^{-1/2}\xi_j^{-1})\prod_{1\leq j<i\leq l_{n+1}} (\xi_j^{-1}-q_E^{-1/2}\Xi_i^{-1}).
\eeqnan
It is clear from this that all $\mu_i$ are half-integral but not integral.

We check that $|\lambda_i|\leq l_{n+1}$.  Choose $i_0\in\{1,2,\dots, l_{n+1}\}.$  The positive contribution of $\Xi_{i_0}$ comes from
$$
\prod_{1\leq j\leq l_{n+1}}(1-q_E^{-1/2}\Xi_{i_0}\xi_j)
$$
and the negative contribution comes from
$$
\prod_{i_0\leq j\leq l_{n+1}}(\Xi_{i_0}^{-1}-q_E^{-1/2}\xi_j^{-1})\prod_{1\leq j< i_0}(\xi_j^{-1}-q_E^{-1/2}\Xi_{i_0}^{-1}).
$$
From this we see that $|\lambda_{i_0}|\leq l_{n+1}.$

Now we check that $|\mu_j|\leq l_{n+1}-1/2.$  Pick some $j_0\in\{1,2,\dots, l_{n+1}\}$.  The positive contribution of $\mu_{j_0}$ comes from
$$
(\xi_{j_0}^{-1/2}-q_E^{-1/2}\xi_{j_0}^{1/2})\prod_{1\leq i\leq l_{n+1}}(1-q_E^{-1/2}\Xi_i\xi_{j_0}).
$$
From this we can see that $|\mu_{j_0}|\leq l_{n+1}+1/2.$  Since we know that $\mu_{j_0}$ is not integral, we now must show that $|\mu_{j_0}|\neq l_{n+1}+1/2$.

Suppose there is some regular monomial $c_{\lambda,\mu}\Xi^{\lambda}\xi^{\mu}$ such that $|\mu_{j_0}| = l_{n+1}+1/2.$  We will show that $|\lambda_{i}|<l_{n+1}$ for all $i$.  This will contradict the fact that the monomial is regular.

A monomial with $\mu_{j_0}=l_{n+1}+1/2$ appears in the the following product 
\beqnan
c_{j_0}\xi_{j_0}^{l_{n+1}+1/2}\prod_{j\neq j_0} (\xi_i^{-1/2}-q_E^{-1/2}\xi_i^{1/2})\prod_{\stackrel{j\neq j_0}{1\leq i\leq l_{n+1}}}(1-q_E^{-1/2}\Xi_i\xi_j)\\
\times\prod_{\stackrel{j\neq j_0}{1\leq i\leq j\leq l_{n+1}}}(\Xi_i^{-1}-q_E^{-1/2}\xi_j^{-1})\prod_{\stackrel{j\neq j_0}{1\leq j<i\leq l_{n+1}}}(\xi_j^{-1}-q_E^{-1/2}\Xi_i^{-1})
\eeqnan
where $c_{j_0}\in \Z[q_E^{\pm 1/2}].$

We claim that for any $i_0$, we cannot have $|\lambda_{i_0}|=l_{n+1}$. If $i_0\leq j_0$, we get $l_{n+1}-i_0$ copies of $\Xi_{i_0}^{-1}$ from the third product, and $i_0-1$ copies of $\Xi_{i_0}^{-1}$ from the fourth product, and no more such factors anywhere else.  Furthermore, we only get at most $l_{n+1}-1$ copies of $\Xi_{i_0}$ from the second product, and no more such factors anywhere else.

A similar argument works to show that we cannot have $\mu_{j_0}=-l_{n+1}-1/2.$  So, all regular monomials in $B_{\Xi,\xi}$ are in the orbit of $\Xi^{\rho_{n+2}}\xi^{\rho_{n+1}}.$

Now, all singular monomials are stabilized by a collection of pairs of Weyl-group elements of opposite sign, and will therefore vanish from (\ref{Asum}).  Furthermore, since $A_{\Xi,\xi}$ is clearly Weyl-invariant, all regular monomials will appear with the same constant coefficient $c$, which is independent of both $\Xi$ and $\xi.$  So we have
$$
A_{\Xi,\xi} =c\cdot (\mathcal{D}_\Xi\mathcal{D}_\xi)^{-1}\sum_{w'\in W_{n+2}, w\in W_{n+1}} \operatorname{sgn}(w)\operatorname{sgn}(w')(w'\Xi)^{-\rho_{n+1}}(w\xi)^{-\rho_{n}} = c.
$$

The proof in Case B proceeds the same as in Case A.  One needs to take
$$
\rho_{n+2} := (l_{n+2}-1/2,l_{n+2}-3/2,\dots , 1/2)
$$
and 
$$
\rho_{n+1}:=(l_{n+2}-1,l_{l+2}-2,\dots,1)
$$
in this case.
\end{proof}

Now we compute $A_{\Xi,\xi}$.

\begin{proposition}
We have $$A_{\Xi,\xi} = \left(L(1,\chi)\zeta(2)L(3,\chi)\dots L(n,\chi)\zeta(n+1)\right)^{-1}$$ if $n+1$ is even and $$A_{\Xi,\xi} = \left(L(1,\chi)\zeta(2)L(3,\chi)\dots L(n+1,\chi)\right)^{-1}$$ if $n+1$ is odd.  Here, $\chi$ is the quadratic character associated to the extension $E/F$ and all $L$ and $\zeta$ factors are with respect to $q_F$.
\end{proposition}

Note that this is simply saying that $A_{\Xi, \xi}^{-1} = L\big(0, M_{n+1}^\vee(1)\big)$, where $M_{n+1}^\vee(1)$ is the twisted dual of the motive $M_{n+1}$ associated to $G_{n+1}$ by Gross \cite{motive}.

\begin{proof}
We prove this only in Case A.  The proof proceeds similarly in Case B.

We set 
\beqnan
\widehat{\Xi} &=& (q_E^{-l_{n+2}}, q_E^{-(l_{n+2}-1)},\dots, q_E^{-1})\\
\widehat{\xi} &=& (q_E^{-l_{n+2}+1/2}, q_E^{-l_{n+2}+3/2},\dots,q_E^{-1/2})
\eeqnan
and compute $A_{\widehat{\Xi},\widehat{\xi}}.$

Note that $b(w'\widehat{\Xi},w\widehat{\xi}) = 0$ if and only if at least one of the following is true:
\beqnan
(w'\widehat{\Xi})_i(w\widehat{\xi})_j &=& q_E^{1/2}\text{ for some }i,j\\
(w'\widehat{\Xi})_i(w\widehat{\xi})_j^{-1} &=& q_E^{1/2}\text{ for some }i\leq j\\
(w'\widehat{\Xi})_i^{-1}(w\widehat{\xi})_j &=& q_E^{1/2}\text{ for some } i>j\\
(w\widehat{\xi})_j &=& q_E^{1/2}\text{ for some }j.
\eeqnan
We show that $b(w'\widehat{\Xi}, w\widehat{\xi})\neq 0\implies w',w=1.$  To see this, note that there are elements $\sigma,\tau$ of the symmetric group $S_{l_{n+2}}$ and $\varepsilon_i,\varepsilon_i'\in\{\pm 1\}$ such that 
\beqnan
w'\widehat{\Xi} &=& \left(\widehat{\Xi}_{\sigma(1)}^{\varepsilon_1'}, \widehat{\Xi}_{\sigma(2)}^{\varepsilon_2'},\dots, \widehat{\Xi}_{\sigma(l_{n+2})}^{\varepsilon_{l_{n+2}}'}\right)\\
w\widehat{\xi} &=& \left( \widehat{\xi}_{\tau(1)}^{\varepsilon_1},\widehat{\xi}_{\tau(2)}^{\varepsilon_2},\dots,\widehat{\xi}_{\tau(l_{n+2})}^{\varepsilon_{l_{n+2}}}\right).
\eeqnan
(Note that $l_{n+2}=l_{n+1}=\frac{n+1}{2}$ in this case.)  Set $r_a := \sigma^{-1}(l_{n+2} + 1 - a)$ and $s_b := \tau^{-1}(l_{n+2}+1-b).$  To avoid the fourth condition above, we see that we must have $\varepsilon_{s_1} = 1$.  But then, to avoid the first condition, we must have $\varepsilon_{r_1}'=1$ as well.  Continuing to avoid the first condition, we see that $\varepsilon_i=\varepsilon_j'=1$ for all $1\leq i,j,\leq l_{n+2}$.  We are left to show that both $\sigma$ and $\tau$ are trivial.  This is equivalent to showing that $r_i\leq s_i$ for all $i$, and that $s_{i+1}<r_i$.  Suppose that $s_i < r_i$ for some $i$.  Then we'd have $(w'\widehat{\Xi})_{r_i}^{-1}(w\widehat{\xi})_{s_i} = q_E^{1/2}$, which is the third condition above.  Similarly, if $s_{i+1}\geq r_i$, then we'd have $(w'\widehat{\Xi})_{r_i}(w\widehat{\xi})_{s_{i+1}}^{-1} = q_E^{1/2}$, which is the second condition above.  So we have
$$
s_1\geq r_1 > s_2\geq r_2>\dots > s_{l_{n+2}}\geq r_{l_{n+2}}
$$
which means that both $w$ and $w'$ are trivial.
This means that
\beqn
A_{\widehat{\Xi}, \widehat{\xi}} = \frac{b(\widehat{\Xi}, \widehat{\xi})}{d_1(\widehat{\Xi})d_0(\widehat{\xi})}.\label{shortA}
\eeqn

We denote $A_l:=A_{\widehat{\Xi^{(l)}}, \widehat{\xi^{(l)}}}$ where $\widehat{\Xi}^{(l)} := (q_E^{-l}, q_E^{-(l-1)}, \dots, q_E^{-1})$ and $\widehat{\xi}^{(l)} := (q_E^{-l+1/2},\dots, q_E^{-1/2}).$  By (\ref{shortA}), and a straightforward calculation, we see that
$$
A_{l+1} = A_l\left(L(2l+1,\chi)\zeta(2l+2)\right)^{-1}.
$$
By induction on $l$, the proof is complete.

The proof in Case B proceeds in the same way by setting $$\widehat{\Xi} := (q_E^{-(l_{n+2}-1/2)}, q_E^{-(l_{n+2}-3/2)},\dots, q_E^{-1/2})$$ and $$\widehat{\xi} := (q_E^{-(l_{n+2}-1)},q_E^{-(l_{n+2}-2)},\dots,q_E^{-1}).$$
\end{proof}

\paragraph{The Split Case.}
Recall that in the split case, the unitary groups are just general linear groups.  So, we consider the groups $G_{n+2}$ and $G_{n+1}$ where $G_i := GL_i(F)$.\footnote{At this point, the reason for sticking with the choice of $n+1$ and $n+2$ is only for consistency with the non-split case.}  Let $B_i, T_i$ and $N_i$ denote the standard Borel subgroups of upper triangular matrices, tori of diagonal matrices, and subgroups of upper triangular unipotent matrices (unipotent radicals).  Now, let $$\xi = (\xi_1,\dots, \xi_{n+1})$$ and $$\Xi = (\Xi_1,\dots, \Xi_{n+2})$$ where each $\xi_i,\Xi_i$ are unramified characters of $F^\times$.  We see that $\xi$ and $\Xi$ can be viewed as characters of $T_{n-1}$ and $T_n$ in the obvious manner; we extend them to characters of $B_{n+1}$ and $B_{n+2}$ by triviality on the $N_i$.  Then we define
$$
I(\xi) := \operatorname{Ind}_{B_{n+1}}^{G_{n+1}}\xi
$$
and
$$
I(\Xi) := \operatorname{Ind}_{B_{n+2}}^{G_{n+2}} \Xi
$$
where the induction is normalized.  Let $\eta$ be a representative for the unique open dense orbit of $B_{n+2}\times B_{n+1}$ on $G_{n+2}$.  Now, we recall the function $S_{\Xi,\xi}$ on $G_{n+2}$.  We have
$$
S_{\Xi,\xi}(g) := \int_{K_{n+1}\times K_{n+2}} Y_{\Xi,\xi}(k_{n+2}g^{-1}k_{n+1})\ dk_{n+1}dk_{n+2}
$$
where $Y_{\Xi,\xi}$ is the function defined on $G_{n+2}$ by
\begin{enumerate}
\item $Y_{\Xi,\xi}(b_{n+2}gb_{n+1}) = (\Xi^{-1}\delta_{n+2}^{1/2})(b_{n+2})(\xi\delta_{n+1}^{-1/2})(b_{n+1})Y_{\Xi,\xi}(g)$ for all $b_{n+2}\in B_{n+2}$ and $b_{n+1}\in B_{n+1}$.
\item $Y_{\Xi,\xi}(\eta) = 1$
\item $Y_{\Xi,\xi}(g) = 0$ for $g\not\in B_{n+2}\eta B_{n+1}$.
\end{enumerate}

As mentioned earlier, we're interested in computing $S_{\Xi,\xi}$ at the identity.  We have the following result of \cite{kmsgln}:
\begin{theorem}
$$
S_{\Xi,\xi}(1) = q_F^{l(w_\ell)+l(w_\ell')}\frac{\prod_{1\leq i < j\leq n+2}L_F(1/2, \xi_i\Xi_{n-j+3})\prod_{1\leq j\leq i<n+2}L_F(1/2,\xi_i^{-1}\Xi_{n-j+3}^{-1})}{\prod_{i=1}^{n+1}\zeta_F(i)\prod_{1\leq i<j\leq n+1}L_F(1,\xi_i\xi_j^{-1})\prod_{1\leq i<j\leq n+2}L_F(1,\Xi_i\Xi_j^{-1})}.
$$
\end{theorem}

\subsubsection{Concluding the unramified calculations}
Now that we've computed both $S_{\Xi,\xi}(1)$ and $\zeta(\Xi,\xi)$, we have actually computed the local integrals in the unramified case.  In this section, we show that they are essentially a product of local $L$-factors.

If we let $\pi_n$ and $\pi_{n+1}$ be unramified principal series representations of $G_n$ and $G_{n+1}$, respectively.  Let $\xi=(\xi_1, \xi_2,\dots, \xi_{\lfloor n/2\rfloor})$ and $\Xi=(\Xi_1,\dots, \Xi_{\lfloor(n+1)/2\rfloor})$ be the relevant characters of $B_n$ and $B_{n+1}$.  We first consider the standard local $L$-factor.  In the quasi-split case, for $n=2l$, we have
\beqnan
\lefteqn{L_E(s, BC(\pi_n)\otimes BC(\pi_{n+1}), \operatorname{st})}\\
 &=& \prod_{1\leq i < j\leq l} L_E(s, \xi_i\Xi_j)L_E(s, \xi_i^{-1}\Xi_j)L_E(s, \xi_i\Xi_j^{-1})L_E(s, \xi_i^{-1}\Xi_j^{-1}) \\
&&\times \prod_{1\leq j \leq i\leq l} L_E(s, \xi_i\Xi_j)L_E(s, \xi_i^{-1}\Xi_j)L_E(s, \xi_i\Xi_j^{-1})L_E(s, \xi_i^{-1}\Xi_j^{-1})\\
&&\times \prod_{i=1}^l L_E(s, \xi_i) L_E(s, \xi_i^{-1}).
\eeqnan
If $n=2l-1$, we have
\beqnan
\lefteqn{L_E(s, BC(\pi_n)\otimes BC(\pi_{n+1}), \operatorname{st})}\\
 &=& \prod_{1\leq i < j\leq l} L_E(s, \xi_i\Xi_j)L_E(s, \xi_i^{-1}\Xi_j)L_E(s, \xi_i\Xi_j^{-1})L_E(s, \xi_i^{-1}\Xi_j^{-1}) \\
&&\times \prod_{1\leq j \leq i\leq l-1} L_E(s, \xi_i\Xi_j)L_E(s, \xi_i^{-1}\Xi_j)L_E(s, \xi_i\Xi_j^{-1})L_E(s, \xi_i^{-1}\Xi_j^{-1})\\
&&\times \prod_{i=1}^l L_E(s, \Xi_i) L_E(s, \Xi_i^{-1}).
\eeqnan

Now we consider the adjoint local $L$-factors.  In the quasi-split case, for $n=2l$, we have
\beqnan
L_F(s, \pi_n, \operatorname{Ad}) &=& \zeta_F(s)^lL_F(s, \chi_{E/F})^l\\
&&\times\prod_{1\leq i<j\leq l} L_F(2s, \xi_i\xi_j) L_F(2s, \xi_i^{-1}\xi_j)L_F(2s, \xi_i^{-1}\xi_j^{-1})L_F(2s, \xi_i\xi_j^{-1})\\
&& \times \prod_{i=1}^l L_F(s, \xi_i)L_F(s, \xi_i^{-1})
\eeqnan
and
\beqnan
\lefteqn{L_F(s, \pi_{n+1}, \operatorname{Ad})}\\
 &=& \zeta_F(s)^{l} L_F(s, \chi_{E/F})^{l+1}\\
&&\times \prod_{1\leq i<j\leq l} L_F(2s, \Xi_i\Xi_j) L_F(2s, \Xi_i^{-1}\Xi_j)L_F(2s, \Xi_i^{-1}\Xi_j^{-1})L_F(2s, \Xi_i\Xi_j^{-1})\\
&& \times \prod_{i=1}^l L_F(s, \chi_{E/F}\Xi_i)L_F(s, \chi_{E/F}\Xi_i^{-1})L_F(2s, \Xi_i)L_F(2s,\Xi_i^{-1}).
\eeqnan
In the quasi-split case, for $n=2l-1$, then we have
\beqnan
\lefteqn{L_F(s, \pi_n, \operatorname{Ad})}\\
 &=& \zeta_F(s)^{l-1}L_F(s, \chi_{E/F})^{l}\\
&&\times\prod_{1\leq i<j\leq l-1} L_F(2s, \xi_i\xi_j) L_F(2s, \xi_i^{-1}\xi_j)L_F(2s, \xi_i^{-1}\xi_j^{-1})L_F(2s, \xi_i\xi_j^{-1})\\
&& \times \prod_{i=1}^{l-1} L_F(s, \chi_{E/F}\xi_i)L_F(s, \chi_{E/F}\xi_i^{-1})L_F(2s, \xi_i)L_F(2s, \xi_i^{-1})
\eeqnan
and
\beqnan
\lefteqn{L_F(s, \pi_{n+1}, \operatorname{Ad})}\\
 &=& \zeta_F(s)^{l} L_F(s, \chi_{E/F})^{l}\\
&&\times \prod_{1\leq i<j\leq l} L_F(2s, \Xi_i\Xi_j) L_F(2s, \Xi_i^{-1}\Xi_j)L_F(2s, \Xi_i^{-1}\Xi_j^{-1})L_F(2s, \Xi_i\Xi_j^{-1})\\
&& \times \prod_{i=1}^l L_F(s, \Xi_i)L_F(s,\Xi_i^{-1}).
\eeqnan

Now we discuss the split case.  Recall that at a split place, we have $E=F\oplus F$, and $V=V_1\oplus V_2$, where each $V_i$ is an $F$-vector space.  Also, we have $U(V)\cong GL(V_1)$ via the map $(g_1, g_2)\mapsto g_1$.  So, the representations $\pi_n$ and $\pi_{n+1}$ are unramified spherical series representations of $GL_n(F)$ and $GL_{n+1}(F)$, respectively.  If $\Xi = (\Xi_1, \dots , \Xi_{n+1})$ and $\xi=(\xi_1, \dots, \xi_n)$, where each of the $\Xi_i$ and $\xi_i$ are unramified characters of $F^\times$, then we have
\beqnan
L_E(s, BC(\pi_n)\otimes BC(\pi_{n+1}), \operatorname{st}) &=& L_F(s, \pi_n\otimes\pi_{n+1}, \operatorname{st})L_F(s, \pi_n^\vee\otimes\pi_{n+1}^\vee, \operatorname{st})\\
&=& \prod_{\stackrel{1\leq i\leq n}{1\leq j\leq n+1}} L_F(s, \xi_i\Xi_j)L_F(s, \xi_i^{-1}\Xi_j^{-1}).
\eeqnan

The adjoint $L$-factors are as follows:
$$
L_F(s, \pi_n, \operatorname{Ad}) = \zeta_F(s)^n\prod_{1\leq i\neq j\leq n} L_F(s, \xi_i\xi_j^{-1})
$$
and
$$
L_F(s, \pi_{n+1}, \operatorname{Ad}) = \zeta_F(s)^{n+1}\prod_{1\leq i\neq j\leq n+1} L_F(s, \Xi_i\Xi_j^{-1}).
$$

So using the results from the previous sections, we have the following:
\begin{theorem}\label{localcalc}  For $v\not\in S$
\beqnan
\mathcal{P}'(f_{\pi_{n+2}}, f_{\pi_{n+1}})&=&\zeta(\Xi,\xi)S_{\Xi^{-1},\xi^{-1}}(1)\\ 
&=& L\big(M_{n+2}^\vee(1), 0\big)\frac{L_E(1/2, BC(\pi_{n+2})\boxtimes BC(\pi_{n+1}))}{L_F(1, \pi_{n+2}, \operatorname{Ad})L_F(1, \pi_{n+1}, \operatorname{Ad})}\\
&=& \Delta_{G_{n+2}}L_{\pi_{n+2}, \pi_{n+1}}(1/2).
\eeqnan
\end{theorem}

\section{The Refined Gross-Prasad Conjecture for $U(1)\times U(2)$}\label{globalchapt}
In this chapter, we give a proof of Conjecture \ref{theconjecture} for $U(1)\times U(2)$.  As mentioned before, in this case the conjecture follows almost immediately from a theorem of Waldspurger.  However, this theorem does not deal directly with representations of unitary groups, but instead of a quaternion algebra.

To bridge this gap, we use a result that says that any irreducible, cuspidal, automorphic representation $\pi$ of $U(2)$ can be lifted (in the appropriate sense of `lift') to an irreducible, cuspidal, automorphic representation $\tilde{\pi}$ of $GU(2)$.  Then, by using a an isomorphism of algebraic groups that relates $GU(2)$ to a quaternion algebra $B$, we have that $\tilde{\pi}\cong \Sigma\boxtimes \eta$, where $\Sigma$ is a representation of $B^\times$, and $\eta$ is a Hecke character of $\A_E^\times$.  After getting our hands on a representation of a quaternion algebra, we let Waldspurger's theorem finish the job.

\subsection{The groups $U(1)\subset U(2)\subset GU(2)$}
Let $F$ be a number field, and let $E$ be a quadratic extension of $F$.  Let $B$ be a quaternion algebra defined over $F$, with a fixed embedding $E\hookrightarrow B$ of $F$-algebras.  We view $B$ as a $2$-dimensional vector space over $E$ via left multiplication.  Let $^-:B\to B$ be the standard involution.  Let $\langle\tau\rangle=\operatorname{Gal}(E/F).$  We note that $\overline{e}=\tau(e)$ for all $e\in E$.  With this involution, we define trace and norm maps in the usual way:
\beqnan
N_B(x) &:=& x\overline{x}\in F\\
Tr_B(x) &:=& x+\overline{x}\in F
\eeqnan
for all $x\in B$.  We have $b\in B$ of trace $0$ which normalizes $E$, and whose conjugation action on $E$ is $\tau$.  Any other member of $B$ with these properties is of the form $\lambda b$, with $\lambda\in E$.  This gives us
$$
B = E\oplus E\cdot b.
$$

Now, we can define a non-degenerate hermitian form on B as follows:
$$
\langle x, y\rangle_B := \text{projection of } x\overline{y}\text{ onto the } E\text{ factor via the decomposition above}.
$$

Since we have a hermitian form on $B$, we can consider the unitary group $U(B)$ and the similitude group $GU(B)$.  Furthermore, we have
$$
GU(B)\cong (B^\times\times E^\times)/(\Delta F^\times)
$$
where $\Delta F^\times$ denotes the diagonally embedded copy of $F^\times$ in $B^\times\times E^\times$.  The action of $B^\times\times E^\times$ on $B$ is given by
$$
(b,e)(x) := exb^{-1}
$$
The similitude character is given by
$$
(b,e)\mapsto N_{E/F}(e)N_B(b)^{-1}.
$$
So, we see that
$$
U(B) = \{\widehat{(b,e)}: N_B(b) = N_{E/F}(e)\}.
$$
Here, $\widehat{(b,e)}$ denotes the equivalence class of $(b,e)$ modulo the diagonally embedded $F^\times$.  We also see the center is given by
\beqn\label{zenter}
Z_{U(B)} = \{\widehat{(f,e)}: N_B(f) = N_{E/F}(e), f\in F^\times\} = \{\widehat{(1,e)}: N_{E/F}(e) = 1\}.
\eeqn

We also consider the line $L_B := E\cdot b\subset B$, and we view the associated unitary group $U(L_B)$ as $E_1$ in $GU(B)$.

Now, for any pair of unitary groups $U(1)\subset U(2)$ defined over $F$, there is a quaternion algebra $B$ over $F$ and embedding $E\hookrightarrow B$ such that $U(1)\cong U(L_B)$ and $U(2)\cong U(B)$.  For ease of notation, we will refer to the unitary groups as $G_1$ and $G_2$, and the unitary similitude group as $\widetilde{G_2}$.

We view $G_1, G_2$ and $\widetilde{G_2}$ as algebraic groups over $F$.  Let $(\pi_1, V_{\pi_1})$ and $(\pi_2, V_{\pi_2})$ be irreducible, cuspidal, tempered automorphic representations of $G_1(\A_F)$ and $G_2(\A_F)$, respectively.  

\subsection{Extending Cusp Forms}
By Theorem 4.13 in \cite{hiraga}, we have the following result about extending cusp forms from $G_2(\A_F)$ to $\widetilde{G_2}(\A_F)$.

\begin{theorem}\label{lift} There is an irreducible, cuspidal, automorphic representation $(\widetilde{\pi_2}, V_{\widetilde{\pi_2}})$ of $\widetilde{G_2}(\A_F)$ such that $V_{\widetilde{\pi_2}}|_{G_2(\A_F)}\supset V_{\pi_2}$.
\end{theorem}

\begin{caution} Note that the restriction above is that of functions, and not restriction of the representation.
\end{caution}

Let $(\widetilde{\pi_2}, V_{\widetilde{\pi_2}})$ be a representation of $\widetilde{G_2}(\A_F)$ given by the theorem above.  Then we have
$$
\widetilde{\pi_2}\cong \Sigma\boxtimes\eta
$$
where $\Sigma$ is a cuspidal irreducible automorphic representation of $B^\times(\A_F)$, and $\eta$ is a Hecke character of $\A_E^\times$, and $\omega_\Sigma\eta|_{\A_F^\times}\equiv 1$, where $\omega_\Sigma$ is the central character of $\Sigma.$  As usual, fix isomorphisms $\Sigma\cong\otimes_v \Sigma_v$ and $\eta\cong\otimes_v\eta_v$.

We observe that we have
$$
\omega_{\pi_2} = \eta|_{\A_{E,1}^\times}.
$$

We denote by $\Sigma'$ the representation of $\operatorname{GL}_2(\A_F)$ associated to $\Sigma$ by the Jacquet-Langlands correspondence.

Now, for $f\in V_{\pi_2}$, we denote by $\widetilde{f}\in V_{\widetilde{\pi_2}}$ a cusp form such that $\widetilde{f}|_{G_2(\A_F)} = f.$  We write
$$
\tilde{f} = \tilde{f}_{\Sigma}\otimes\eta.
$$
As a consequence of this decomposition, we note that for any $f_1,f_2\in V_{\pi_2}$, the corresponding $\tilde{f}_1,\tilde{f}_2\in V_{\widetilde{\pi_2}}$ satisfy $\tilde{f}_1(z)\overline{\tilde{f}_2(z)}=1$ for all $z\in Z_{GU(B)}$, so that $\tilde{f}_1\overline{\tilde{f}_2}$ is a function on $Z_{GU(B)}\backslash GU(B)$.  Furthermore, since $\tilde{f}_{1,\Sigma}(z) = \tilde{f}_{2,\Sigma}(z)$ for all $z$ in the center $Z_{B^\times}(\A_F)$, we see that $\tilde{f}_{1,\Sigma}\overline{\tilde{f}_{2,\Sigma}}$ is a function on $\mathbb{P}B^\times(\A_F) := (Z_{B^\times}\backslash B^\times)(\A_F)$.

\subsection{Waldspurger's Theorem}
We give a brief discussion of Waldspurger's theorem to be used in the proof of Conjecture \ref{theconjecture} for $n=1$.  We take $B,\Sigma$, and $\Sigma'$ as defined above.  Let $T$ be the torus given by the embedding $E\hookrightarrow B$, so that $T(F)$ is identified with $E^\times\subset B^\times(F)$. Let $\chi$ be a Hecke character of $T(\A_F)$, and $f\in\Sigma$.  Let $Z_B$ denote the center of $B^\times$.  The period that Waldspurger considers is
$$
\widetilde{\mathcal{P}}(f,\chi):= \left|\int_{Z_B(\A_F)T(F)\backslash T(\A_F)} f(t)\chi^{-1}(t)dt\right|^2
$$
where $dt$ is the Tamagawa measure, which gives $\operatorname{Vol}(Z_B(\A_F)T(F)\backslash T(\A_F))=2.$  Let $\mathcal{B}_{\Sigma_v}$ and $\mathcal{B}_{\chi_v}$ be local pairings.  We also choose local measures $dt_v$ so that $dt = \prod_v dt_v$, as usual.  Then set
\beqnan
\lefteqn{\alpha_v(f_v, \chi_v) :=}&&\\
&& \left(\frac{\zeta_{F_v}(2)L_{F_v}(1/2, \Sigma_v'\otimes\chi_v^{-1})}{L_{F_v}(1,\Sigma_v', \operatorname{Ad})L_{F_v}(1, \chi_{E_v/F_v})}\right)^{-1}\int_{T_v} \mathcal{B}_{\Sigma_v}(\Sigma_v(t_v)f_v, f_v)\mathcal{B}_{\chi_v}(\chi_v^{-1}(t_v)\chi_v, \chi_v)dt_v.
\eeqnan

We let $\mathcal{B}_{\Sigma}$ be the Petersson inner product on $\Sigma$, where the integral is taken over $[\mathbb{P}B^\times]$. That is
$$
\mathcal{B}_\Sigma(f_1, f_2) := \int_{[\mathbb{P}B^\times]} f_1(b)\overline{f_2(b)} db
$$
where $db$ is the Tamagawa measure.

Then Waldspurger's theorem (see \cite{wald}, page 222) is the following:
\begin{theorem} Suppose that $\Sigma$ has trivial central character.  Then
$$
\frac{\widetilde{\mathcal{P}}(f,\chi)}{\mathcal{B}_\Sigma(f,f)} = \frac{\zeta_F(2) L_E(1/2, BC(\Sigma')\otimes\chi^{-1})}{2 L_F(1, \Sigma', \operatorname{Ad}) L_F(1,\chi_{E/F})}\prod_v \frac{\alpha_v(f_v, \chi_v)}{\mathcal{B}_{\Sigma_v}(f_v, f_v)\mathcal{B}_{\chi_v}(\chi_v, \chi_v)}.
$$
\end{theorem}
\begin{remark} The reader will notice that the formulation of Waldspurger's theorem we give looks slightly different than that given by Waldspurger himself.  In \cite{wald}, he chooses the global Haar measure\footnote{Waldspurger also refers to this as the Tamagawa measure.  This collision of terminology is unfortunate, and we hope the reader suffers minimal confusion as a result.} such that $\operatorname{Vol}(Z_B(\A_F)T(F)\backslash T(\A_F))=2L_F(1,\chi_{E/F}),$ and he chooses local measures compatibly with respect to this.  With our choice of measures, the formulation above is equivalent.  He also does not include the local pairing $\mathcal{B}_{\chi_v}$ anywhere in the result.  Our inclusion of $\mathcal{B}_{\chi_v}$ -- both in the definition of the $\alpha_v$ and in the denominator of the product on the RHS of the theorem -- does not change anything.
\end{remark}

\subsection{Proof of Conjecture \ref{theconjecture} for $U(1)\times U(2)$}
Waldspurger's theorem does the bulk of the work in proving Conjecture \ref{theconjecture} for $n=1$.  As usual, we let $\pi_i$ denote an irreducible, tempered, cuspidal, automorphic representation of $G_i(\A_F)$.

Let $\mathcal{B}_\Sigma$ be as above, and $\mathcal{B}_{\pi_i}, \mathcal{B}_{\widetilde{\pi_2}}$ are defined as follows, where all global measures are the appropriate Tamagawa measure:
\beqnan
\mathcal{B}_{\pi_i}(f_1,f_2) &:=& \int _{[G_i]} f_1(g)\overline{f_2(g)} dg\\
\mathcal{B}_{\widetilde{\pi_2}}(\tilde{f}_1, \tilde{f}_2) &:=& \int_{Z_{\widetilde{G_2}}(\A_F)\widetilde{G_2}(F)\backslash \widetilde{G_2}(\A_F)} \tilde{f}_1(g)\overline{\tilde{f}_2(g)}dg.
\eeqnan
Note that for $\tilde{f}_1,\tilde{f}_2\in \widetilde{\pi_2}$, we have $\mathcal{B}_{\widetilde{\pi_2}}(\tilde{f}_1, \tilde{f}_2) = \mathcal{B}_\Sigma(\tilde{f}_{1,\Sigma},\tilde{f}_{2,\Sigma})$.

We choose local pairings $\mathcal{B}_{\pi_{i,v}}$ compatibly with the associated global pairings, so that $\prod_v \mathcal{B}_{\pi_{i,v}} = \mathcal{B}_{\pi_i}$.  However, we choose the local pairings $\mathcal{B}_{\Sigma_v}$ so that for $f_i\in \pi_2$ -- which we extend to $\tilde{f}_i\in \widetilde{\pi_2}$ by Theorem \ref{lift} -- we have $ \mathcal{B}_{\Sigma_v}(\tilde{f}_{1,\Sigma_v}, \tilde{f}_{2,\Sigma_v}) = \mathcal{B}_{\pi_{2,v}}(f_{1,v}, f_{2,v})$, where $\tilde{f}_{i,v} = \tilde{f}_{i,\Sigma_v}\otimes\eta.$  

We now have enough in place to prove Conjecture \ref{theconjecture} for $n=1$.  Recall that Waldspurger's theorem assumes that the central character of $\Sigma$ is trivial.  However, in \cite{grosszag}, the authors remove this assumption.

\begin{theorem}  Let $f\in\pi_2$ and $\tilde{f}=\tilde{f}_\Sigma\otimes\eta\in\widetilde{\pi_2}$ such that $\tilde{f}|_{G_2} = f$.  Let $\theta\in \pi_1$ be a unitary character of $G_1(\A_F)$, which we are viewing as the norm-one elements of $\A_E^\times$.  Then
$$
\mathcal{P}( f|_{G_1}, \theta) = \frac{\Delta_{G_2} L_E(1/2, BC(\pi_2)\boxtimes BC(\pi_1))}{4|X({\pi_2})|L_F(1, \pi_2, \operatorname{Ad})L_F(1, \theta, \operatorname{Ad})}\prod_{v} \mathcal{P}_v(f_{v}, \theta_{v})
$$
where $X(\pi_2)$ is the set of automorphic characters $\omega$ of $GU(2)(\A_F)/U(2)(\A_F)$ such that $\widetilde{\pi_2}\otimes\omega \cong \widetilde{\pi_2}$.  We remind the reader that $\Delta_{G_2} := L_F(1,\chi_{E/F})\zeta_F(2)$.
\end{theorem}
\begin{proof}
We have the following:
\beqnan
 \int_{[G_1]} f(g)\overline{\theta(g)}\ dg &=& \int_{[G_1]} \tilde{f}(g)\overline{\theta(g)}\ dg\\
&=& \int_{Z_B(\A_F)T(F)\backslash T(\A_F)} \tilde{f}_{\Sigma}(g)\eta(g)\overline{\theta(g)}\ dg\\
&=& \int_{Z_B(\A_F)T(F)\backslash T(\A_F)} \tilde{f}_{\Sigma}(g)\overline{\eta^{-1}(g)\theta(g)}\ dg.
\eeqnan
So, Waldspurger's theorem gives
$$
\frac{\mathcal{P}(f|_{G_1}, \overline{\theta})}{\mathcal{B}_{\Sigma}(\tilde{f}_\Sigma,\tilde{f}_\Sigma)} = \frac{\zeta_F(2)L_E(1/2, BC(\Sigma')\otimes\eta BC(\theta^{-1}))}{2L_F(1,\Sigma',\operatorname{Ad})L_F(1,\chi_{E/F})}\prod_v \frac{\alpha_v(f_{\Sigma,v}, \overline{\eta_v^{-1}\theta_v})}{\mathcal{B}_{\Sigma_v}(\tilde{f}_{\Sigma_v}, \tilde{f}_{\Sigma_v}) \mathcal{B}_{\eta_{v}}(\eta_v^{-1}\theta_v,\eta_v^{-1}\theta_v)}.
$$
Noting that
$$
\alpha_v(\tilde{f}_{\Sigma_v}, \overline{\eta^{-1}_v\theta_v}) = \mathcal{P}_v(f_{v}, \overline{\theta_v})
$$
and
$$
\prod_v \mathcal{B}_{\eta_{v}}(\eta_v^{-1}\theta_v,\eta_v^{-1}\theta_v) = 2
$$
we have
$$
\frac{\mathcal{P}(f|_{G_1}, \overline{\theta})}{\mathcal{B}_{\Sigma}(\tilde{f}_\Sigma,\tilde{f}_\Sigma)} = \frac{\zeta_F(2)L_E(1/2, BC(\Sigma')\boxtimes\eta BC(\theta^{-1}))}{4L_F(1,\Sigma',\operatorname{Ad})L_F(1,\chi_{E/F})}\prod_v \frac{\mathcal{P}_v(f_{v}, \overline{\theta_v})}{\mathcal{B}_{\Sigma_v}(\tilde{f}_{\Sigma_v}, \tilde{f}_{\Sigma_v})}.
$$
We remark that $BC(\theta^{-1})$ is the character of $\A_E^\times/\A_F^{\times}$ given by $$BC(\theta^{-1})(x) = \frac{\theta^{-1}(x)}{\theta^{-1}(\tau(x))}$$ where $\operatorname{Gal}(E/F)$ is generated by $\tau$.

Recall that $\mathcal{B}_\Sigma(\tilde{f}_\Sigma, \tilde{f}_\Sigma) = \mathcal{B}_{\widetilde{\pi_2}}(\tilde{f}, \tilde{f})$ and $\mathcal{B}_{\Sigma_v}(\tilde{f}_{\Sigma_v}, \tilde{f}_{\Sigma_v}) = \mathcal{B}_{\pi_{2,v}}(f_v, f_v)$.  As for the $L$-values, we have
$$
L_{F}(s, \Sigma', \operatorname{Ad}) = L_{F}(s, \pi_{2}, \operatorname{Ad})L_{F}(s, \chi_{E/F})^{-1}
$$
and
$$
L_E(1/2, BC(\Sigma')\otimes\eta BC(\theta^{-1})) = L_{E}(s, BC(\pi_{2})\boxtimes BC(\pi_1^\vee)).
$$
This gives:
$$
\frac{\mathcal{P}(f|_{G_1}, \overline{\theta})}{\mathcal{B}_{\widetilde{\pi_2}}(\tilde{f},\tilde{f})} = \frac{\Delta_{G_2}L_{E}(s, BC(\pi_{2})\boxtimes BC(\pi_1^\vee))}{4L_F(1,\pi_2,\operatorname{Ad})L_F(1,\chi_{E/F})}\prod_v \frac{\mathcal{P}_v(f_{v}, \overline{\theta_v})}{\mathcal{B}_{\pi_{2,v}}({f}_{v}, {f}_{v})}.
$$
By Remark 4.20 of \cite{hiraga}, we have:
$$
\frac{\mathcal{B}_{\pi_2}(f, f)}{\operatorname{Vol}(G_2(F)\backslash G_2(\A_F))} = |X({\pi_2})|\cdot\frac{\mathcal{B}_{\widetilde{\pi_2}}(\tilde{f}, \tilde{f})}{\operatorname{Vol}(Z_{\widetilde{G_2}}(\A_F)\widetilde{G_2}(F)\backslash \widetilde{G_2}(\A_F))}.
$$
The volumes are:
$$
\operatorname{Vol}(G_2(F)\backslash G_2(\A_F)) =\operatorname{Vol}(Z_{\widetilde{G_2}}(\A_F)\widetilde{G_2}(F)\backslash \widetilde{G_2}(\A_F)) = 2.
$$
Finally, we note that $|S_{\psi_1}|=2$ and $|S_{\psi_2}|=2\cdot |X(\pi_2)|$, so that
$$
|S_{\psi_1}|\cdot |S_{\psi_2}| = 4\cdot |X(\pi_2)|
$$
as stated in Conjecture \ref{theconjecture}.  This completes the proof.
\end{proof}


\section{Ichino's Triple Product Formula}\label{triplechapt}
We now begin introducing the machinery necessary to prove Conjecture \ref{theconjecture} for $n=2$.  The first tool that we introduce is a result due to Ichino: the so-called triple product formula.  Like Waldspurger's theorem and the Refined Gross-Prasad Conjectures already mentioned, Ichino's formula gives an explicit relationship between a period integral and a particular $L$-value.

Let $\tau_1, \tau_2$ and $\tau_3$ be irreducible, cuspidal representations of $G_2$.  Denote by $\omega_i$ the central character of $\tau_i$.  We require that $\omega_1\omega_2\omega_3\equiv 1$.  Recall that from Theorem \ref{lift} we have corresponding representations $\tilde{\tau}_1, \tilde{\tau}_2, \tilde{\tau}_3$ of $\widetilde{G_2}.$  Also recall that we have
$$
\tilde{\tau_i} \cong \Sigma_i\boxtimes\eta_i.
$$
In order to make use of Ichino's formula, we must ensure that the central characters of the $\Sigma_i$ multiply to give the trivial character.  The following lemma ensures that we can choose the $\tilde{\tau}_i$ extending the $\tau_i$ such that this holds.
\begin{lemma} There exist $\tilde{\tau}_i$ extending the $\tau_i$ such that the corresponding $\eta_i$ satisfy $\eta_1\eta_2\eta_3\equiv 1$.
\end{lemma}
\begin{proof} We note that since $\omega_1\omega_2\omega_3\equiv 1,$ we already have that $\eta_1\eta_2\eta_3|_{\A_{E,1}^\times} = 1.$  But this means that $\eta_1\eta_2\eta_3=\chi\circ N_{E/F}$ for some Hecke character $\chi$ of $\A_F^\times$.  So, by twisting one of the $\eta_i$ -- say $\eta_1$ -- by $\chi^{-1}\circ N_{E/F}$, we can achieve the desired result.  Note that in order to ensure that $\tilde{\tau}_1$ still extends $\tau_1$, we must also twist $\Sigma_1$ by $\chi\circ N_B$.
\end{proof}

An immediate consequence of the previous result is that we can choose the $\tilde{\tau}_i$ such that the central characters $\omega_{\Sigma_i}$ of the $\Sigma_i$ satisfy $\omega_{\Sigma_1}\omega_{\Sigma_2}\omega_{\Sigma_3} \equiv 1.$  To see this, we simply note that
$$
\omega_{\Sigma_1}\omega_{\Sigma_2}\omega_{\Sigma_3} = (\eta_1\eta_2\eta_3)^{-1}|_{\A_F^\times}\equiv 1.
$$
So, having chosen the $\tilde{\tau}_i$ in this way, we are entitled to make use of Ichino's formula.

Let $\tilde{f}_i\in\tilde{\tau}_i$.  In \cite{triple}, Ichino proves:
\beqn\label{icheq}
\left|\int_{\mathbb{P}B^\times(F)\backslash \mathbb{P}B^\times(\A_F)} (\tilde{f}_{1,\Sigma_1}\tilde{f}_{2, \Sigma_2}\tilde{f}_{3,\Sigma_3})(b)\ db\right|^2 = \frac{\zeta_F(2)^2L_F(1/2, \Sigma')}{8\left(\prod_{i=1}^3|X(\tau_i)|\right)L_F(1, \Sigma', \operatorname{Ad})}\prod_{v} \mathcal{J}_v
\eeqn
where
$$
L_F(s,\Sigma') := L_F(s,\Sigma_1'\boxtimes \Sigma_2'\boxtimes\Sigma_3')
$$
is the triple-product $L$-function,
$$
L_F(s,\Sigma',\operatorname{Ad}) := L_F(s,\Sigma_1',\operatorname{Ad}) L_F(s, \Sigma_2',\operatorname{Ad})L_F(s,\Sigma_3',\operatorname{Ad}).
$$
and the $\mathcal{J}_v$ are defined as follows:
\beqnan
\mathcal{J}_v &:=& \frac{L_{F_v}(1, \Sigma_v', \operatorname{Ad})}{\zeta_{F_v}(2)^2L_{F_v}(1/2, \Sigma_v')}\times\\
&&\int_{\mathbb{P}B^\times_v} \mathcal{B}_{\Sigma_{1,v}}(\Sigma_{1,v}(b_v) \tilde{f}_{\Sigma_1,v}, \tilde{f}_{\Sigma_1,v})\mathcal{B}_{\Sigma_{2,v}}( \Sigma_{2,v}(b_v) \tilde{f}_{\Sigma_2,v}, \tilde{f}_{\Sigma_2,v})\\
&&\mathcal{B}_{\Sigma_{3,v}} (\Sigma_{3,v}(b_v) \tilde{f}_{\Sigma_3,v}, \tilde{f}_{\Sigma_3,v})\ db_v.
\eeqnan
Here we make what will seem -- for the moment -- to be a strange normalization; we require that $\prod_{i=1}^{3}\left(|X(\tau_i)|^{-1}\cdot\prod_v \mathcal{B}_{\Sigma_i,v}\right)$ is the Petersson inner product on $\mathbb{P}B^\times.$

We will spend the remainder of this chapter using this formula to derive one more suited for our purposes.  For the remainder of the chapter, we assume that at least one of the $\tilde{\tau}_i$ -- say $\tilde{\tau}_3$ -- is dihedral with respect to $E/F$; in other words, we assume that $\Sigma_3\cong \Sigma_3\otimes\chi_{E/F}$.

At every place $v$ of $F$, we consider the following subgroup of $B_v^\times$:
$$
(B_v^\times)^+ := \{b_v\in B_v^\times: N_{B_v}(b_v)\in N_{E_v/F_v}(E_v^\times)\}.
$$
Note that $(B_v^\times)^+$ is not the $F_v$-points of an algebraic group over $F$.  We also write
$$
B^\times(\A_F)^+ := \{(b_v)_v\in B^\times(\A_F): b_v\in (B_v^\times)^+\text{ for every } v\}.
$$
We will denote
$$
(\mathbb{P}B_v^\times)^+ := \{\hat{b}_v\in \mathbb{P}B_v^\times : b_v\in (B_v^\times)^+\}.
$$
It is easy to see that this is well-defined.  Note that $(B_v^\times)^+$ contains the center $Z_{B_v}$ of $B_v^\times$, so that $$(\mathbb{P}B^\times_v)^+\cong (B_v^\times)^+/ Z_{B_v}.$$  We also set
$$
\mathbb{P}B^\times(\A_F)^+ := \{\widehat{(b_v)}_v\in \mathbb{P}B^\times(\A_F) : b_v\in (B_v^\times)^+\text{ for all } v\}.
$$

We now give two lemmas toward converting Ichino's triple product formula to a formula involving only data for $U(2)$.  The first of these is a local result which will allow us to relate integrals of matrix coefficients of the $\tilde{\tau}_i$ over $GU(2)_v$ to integrals of matrix coefficients of the $\tau_i$ over $U(2)_v$.

\begin{lemma}\label{gu2locmat} Let $\tilde{\tau}_v=\Sigma_v\boxtimes\eta_v$ be an irreducible representation of $GU(2)_v$.  Suppose that $\Sigma_v$ is dihedral with respect to $E_v/F_v$.  Viewing $\tilde{\tau}_v$ as a representation of $U(2)_v\subset GU(2)_v$ by restriction, we have
$$
\tilde{\tau}_v = \tau_v^+\oplus \tau_v^-
$$
where $\tau_v^+$ and $\tau_v^-$ are inequivalent and both irreducible.  Furthermore, for a vector $x^+\in \tau^+, \langle\cdot,\cdot\rangle$ any $GU(2)_v$-invariant inner product on $\tilde{\tau}_v$, and $f$ a function on $GU(2)_v$ such that $\langle g\cdot x^+, x^+\rangle\cdot f(g)$ is a function on $Z_{GU(2)_V}\backslash GU(2)_v$, we have
$$
\int_{Z_{GU(2)_v}\backslash GU(2)_v} \langle g\cdot x^+, x^+\rangle f(g) dg = \int_{Z_{GU(2)_v}\backslash GU(2)_v^+} \langle g\cdot x^+, x^+\rangle f(g) dg
$$
where $GU(2)_v^+ := Z_{GU(2)_v}U(2)_v.$
\end{lemma}
\begin{proof} The fact that the restriction of $\tilde{\tau}_v$ to $U(2)_v$ decomposes as described follows from the fact that $\Sigma_v$ is dihedral.  Now, setting $GU(2)_v^- := GU(2)_v - GU(2)_v^+$ we have
$$
GU(2)_v = GU(2)_v^+\cup GU(2)_v^+c
$$
for some $c\in GU(2)_v^-$ that interchanges $\tau_v^+$ and $\tau_v^-$.  (The fact that $c$ interchanges $\tau_v^+$ and $\tau_v^-$ follows from the fact that neither is invariant under $GU(2)_v$.)  So, we have the following identity:
\beqnan
\lefteqn{\int_{Z_{GU(2)_v}\backslash GU(2)_v} \langle g\cdot x^+, x^+\rangle f(g) dg}&&\\ &=& \int_{Z_{GU(2)_v}\backslash GU(2)_v^+} \langle g\cdot x^+, x^+\rangle f(g) dg + \int_{Z_{GU(2)_v}\backslash GU(2)_v^+} \langle gc\cdot x^+, x^+\rangle f(gc) dg.
\eeqnan
We claim the second integral on the RHS vanishes.  Indeed, $c\cdot x^+\in\tau_v^-$, and $g\cdot x^-\in\tau_v^-$ for any $x^-\in \tau_v^-$ and $g\in GU(2)_v^+$.  Since $\tau_v^-$ and $\tau_v^+$ are orthogonal, this completes the proof.
\end{proof}
The next result is a global analogue of the previous one.  While we already know that for a cusp form $f$ on $U(2)(\A_F)$, there is a cusp form $\tilde{f}$ on $GU(2)(\A_F)$ which restricts to $f$, this is not quite good enough to compare Ichino's triple product integral to one over $[U(2)]$.  We need a $\tilde{f}$ which vanishes away from `the complement of $U(2)$ in $GU(2)$'.  The following lemma says that such a $\tilde{f}$ exists.

\begin{lemma}\label{extend} Let $\tau$ be an irreducible, cuspidal, automorphic representation of $U(2)(\A_F)$, and let $\tilde{\tau}=\Sigma\boxtimes\eta$ be an irreducibile, cuspidal, automorphic representation of $GU(2)(\A_F)$ extending $\tau$.  Suppose that $\Sigma$ is dihedral with respect to $E/F$.  Then for $f\in \tau$, there is a $\tilde{f}=\tilde{f}_\Sigma\otimes\eta\in\tilde{\tau}$ such that $\tilde{f}_\Sigma$ vanishes away from $B^\times(F)B^\times(\A_F)^+$.
\end{lemma}
\begin{proof} First, we fix a decomposition $\tilde{\tau}\cong\otimes_v \tilde{\tau}_v$.  Now, as in the previous lemma, we have $\tilde{\tau}_v=\tau_v^+\oplus\tau_v^-$ as representations of $U(2)_v$.  We adjust the labelings so that for almost all $v$, $\tau_v^+$ contains the spherical vector of $\tilde{\tau}_v$.  We note that as a $U(2)$-module we have
$$
\tilde{\tau} \cong \bigoplus_S \tau_S
$$
where $S$ runs over all finite sets of places of $F$, and
$$
\tau_S := (\otimes_{v\in S}\ \tau_v^-)\otimes (\otimes_{v\notin S}\ \tau_v^+).
$$
We note that $S\neq S'\implies \tau_S\not\cong \tau_{S'}$.  So, there is a unique $S_0$ such that $\tau\cong\tau_{S_0}$.

We denote by $\mathcal{A}_0(GU(2))$ and $\mathcal{A}_0(U(2))$ the spaces of cusp forms on $GU(2)(\A_F)$ and $U(2)(\A_F)$, respectively.  Under the natural restriction map $\operatorname{Res}:\mathcal{A}_0(GU(2))\to\mathcal{A}_0(U(2))$, we see that $\tau_{S_0}$ is sent isomorphically to $\tau$.  So, there is a unique $\tilde{f}=\tilde{f}_\Sigma\otimes\eta\in\tau_{S_0}\subset\tilde{\tau}$ such that  $\operatorname{Res}(\tilde{f})=f$.  We now claim that $\tilde{f}_\Sigma$ vanishes away from $B^\times(F)B^\times(\A_F)^+$.

Toward proving this claim, we consider the map
\beqnan
\phi : \mathcal{A}_0(GU(2)) &\to& \mathcal{A}_0(GU(2))\\
f_\Sigma\otimes\eta &\mapsto& f_\Sigma\chi_{E/F}\otimes\eta.
\eeqnan
If we let $R$ be the action given by right translation of $GU(2)(\A_F)$ on $\mathcal{A}_0(GU(2))$, then we see that $\phi$ intertwines the action of $R$ and $R\otimes \chi_{E/F}$.  In particular, since $\chi_{E/F}$ is trivial on $U(2)(\A_F)\subset GU(2)(\A_F)$, we see that $\phi$ commutes with the action of $U(2)(\A_F)$.

Now, let $\tilde{V}\subset\mathcal{A}_0(GU(2))$ denote the underlying space of functions for $\tilde{\tau}$.  We claim that
$$
\phi(\tilde{V})=\tilde{V}.
$$
This follows from the fact that $\tilde{\tau}$ is dihedral with respect to $E/F$, and that $GU(2)$ enjoys the multiplicity-one property.

So, we see that we have a $GU(2)(\A_F)$-equivariant isomorphism
$$
\phi:(\tilde{V}, R)\to (\tilde{V}, R\otimes\chi_{E/F}).
$$
Noting that $\phi^2=1$, and that $\phi$ is clearly not a scalar, we have a decomposition
$$
\tilde{V}=\tilde{V}^+\oplus \tilde{V}^-
$$
where $\tilde{V}^+$ is the $+1$ eigenspace for $\phi$, and $\tilde{V}^-$ is the $-1$ eigenspace.

Now, note that if $\tilde{f}\in\tilde{V}^+$, $\tilde{f}_\Sigma$ vanishes away from $B^\times(F)B^\times(\A_F)^+$.  Similarly, if $\tilde{f}\in\tilde{V}^-$, $\tilde{f}_\Sigma$ vanishes on $B^\times(F)B^\times(\A_F)^+$.  Note that this is simply saying that the kernel of $\operatorname{Res}|_{\tilde{V}}$ is precisely $\tilde{V}^-$.

Let $\tilde{V}_{S_0}\subset \tilde{V}$ be the space of functions affording the representation $\tau_{S_0}$ described above.  To finish the proof, we must show that $\tilde{V}_{S_0}\subset \tilde{V}^+$.  Since $\phi$ commutes with the action of $U(2)(\A_F)$ on $\tilde{V}_{S_0}$, we see that either $\tilde{V}_{S_0}\subset\tilde{V}^+$ or $\tilde{V}_{S_0}\subset\tilde{V}^-$.  But the latter is impossible, since $\operatorname{Res}(\tilde{V}^-)=0$.
\end{proof}
\begin{remark}\label{phirem} Note that the map $\phi$ in the preceding proof has a local analogue.  Indeed:  if $\tau$ is dihedral, then for each $v$, we have $GU(2)_v$-equivariant isomorphisms $$\phi_v : \tilde{\tau}_v\to \tilde{\tau}_v\otimes\chi_{E_v/F_v}.$$  Note that we may normalize these $\phi_v$ so that they are well-determined up to a sign.  By Schur's Lemma, we see that $\phi=\pm\otimes_v\phi_v$.  By adjusting the sign of one of the $\phi_v$, we may assume that $\phi=\otimes_v\phi_v$.  Let $\tau_v^+$ and $\tau_v^-$ denote the $+1$ and $-1$ eigenspaces for $\phi_v$, respectively.  Recall that there is a unique set $S_0$ of places of $F$ such that $(V_{S_0},\tau_{S_0})\subset \tilde{V}^+$ affords the representation $\tau$.  Note that by the above, $|S_0|$ is even.  By adjusting the $\phi_v$ for $v\in S_0$, we can assume $S_0=\emptyset$, and we can still have $\phi=\otimes_v\phi_v$.  The fact that $S_0=\emptyset$ means that we have $\tau_v=\tau_v^+$ for all $v$.
\end{remark}

We observe that $\mathbb{P}B^\times(F)\mathbb{P}B^\times(\A_F)^+$ is a subgroup of index $2$ in $\mathbb{P}B^\times(\A_F)$.  Similarly, the space $\mathbb{P}B^\times(F)\backslash \mathbb{P}B^\times(F)\mathbb{P}B^\times(\A_F)^+$ has `index $2$' in $\mathbb{P}B^\times(F)\backslash\mathbb{P}B^\times(\A_F)$.  Denoting $\mathbb{P}B^\times(F)^+:= \mathbb{P}B^\times(F)\cap \mathbb{P}B^\times(\A_F)^+$, we see that we may identify the spaces $$\mathbb{P}B^\times(F)^+\backslash \mathbb{P}B^\times(\A_F)^+ \cong \mathbb{P}B^\times(F)\backslash \mathbb{P}B^\times(F)\mathbb{P}B^\times(\A_F)^+.$$  With this in mind, we may view $\mathbb{P}B^\times(F)^+\backslash \mathbb{P}B^\times(\A_F)^+$ as a space of `index $2$' in $\mathbb{P}B^\times(F)\backslash \mathbb{P}B^\times(\A_F)$.

With this is mind, we have the following consequence of Ichino's formula and the preceding lemma:
\begin{proposition}\label{triple} If $\tilde{f}_3$ is chosen such that $\tilde{f}_{3,\Sigma_3}$ vanishes off $B^\times(F)B^\times(\A_F)^+$, then
$$
\left|\int_{\mathbb{P}B^\times(F)^+\backslash \mathbb{P}B^\times(\A_F)^+} (\tilde{f}_{1,\Sigma_1}\tilde{f}_{2, \Sigma_2}\tilde{f}_{3,\Sigma_3})(b)\ db\right|^2 = \frac{\zeta_F(2)^2L_F(1/2, \Sigma')}{8 \left(\prod_{i=1}^3 |X(\tau_i)|\right)  L_F(1, \Sigma', \operatorname{Ad})}\prod_{v} \mathcal{J}_v.
$$
\end{proposition}
Our final task for this chapter is to reinterpret the result above completely in terms of data on unitary groups.  Having chosen extensions $\tilde{\tau}_i$ of the $\tau_i$, we have the decompositions $$\tilde{\tau}_i\cong\otimes_v\tilde{\tau}_{i,v}$$ and $$\tau_i\cong\otimes_v\tau_{i,v}.$$  We see that as representations of $U(2)$, we have $\tau_{i,v}\subset \tilde{\tau}_{i,v}.$  Furthermore, by Remark \ref{phirem} we have that $\tau_{3,v}=\tau_{3,v}^+.$

We have already fixed the pairings $\mathcal{B}_{\Sigma_{i,v}}$; now we also fix pairings $\mathcal{B}_{\eta_{i,v}}$ on the $\eta_{i,v}$ such that $\prod_v \mathcal{B}_{\eta_{i,v}}$ gives the Petersson pairing on $\eta_i$.  We see that the tensor product of $\mathcal{B}_{\Sigma_{i,v}}$ and $\mathcal{B}_{\eta_{i,v}}$ yields a pairing on $\tilde{\tau}_{i,v}$, which we denote by $\mathcal{B}_{\tilde{\tau}_{i,v}}$.  We also consider its restriction to the subspace $\tau_{i,v}$, which we denote by $\mathcal{B}_{\tau_{i,v}}$.

\begin{remark}  By Remark $4.20$ of \cite{hiraga} and our choice of $\mathcal{B}_{\Sigma_{i,v}}$ earlier, we have that the products $\prod_v\mathcal{B}_{\tau_{i,v}}$ give the respective Petersson inner products on $\tau_i$.
\end{remark}

Let $\tilde{f}_i\in\tilde{\tau}_i$ be cusp forms extending $f_i\in\tau_i$.  Suppose that $f_i=\otimes_v f_{i,v}$.  Without loss of generality, we may assume that the $\tilde{f}_i$ are also pure tensors (because $\operatorname{Res}$ restricts to an isomorphism on the subspaces of $\mathcal{A}_0(GU(2))$ that afford the $\tau_i$).  Suppose that $\tilde{f}_3$ is chosen such that $\tilde{f}_{3,\Sigma_3}$ vanishes away from $B^\times(F)B^\times(\A_F)^+$.

We can now give a reinterpretation of the previous proposition, using data for $G_2$ instead of $\widetilde{G_2}.$  We remark that since $\tau_3$ is assumed to be dihedral with respect to $E/F$, we have that $BC(\tau_3)$ is isomorphic to the principal series representation $\pi(\chi_1,\chi_2)$ of $GL_2(\A_E)$ for some Hecke characters $\chi_i$ of $\A_E^\times.$

Note that with the choices we have made, we can rewrite the $\mathcal{J}_v$ as integrals of matrix coefficients of the $\tau_i$.  Indeed, we define
\beqnan
\mathcal{I}_v &:=& \frac{L_{F_v}(1, \tau_{1,v}, \operatorname{Ad}) L_{F_v}(1, \tau_{2,v}, \operatorname{Ad}) L_{F_v}(1, \tau_{3,v}, \operatorname{Ad})}{\zeta_{F_v}(2)^2 L_{E_v}(1/2, BC(\tau_{1,v})\boxtimes BC(\tau_{2,v})\boxtimes \chi_{1,v}) L_{F_v}(1,\chi_{E_v/F_v})^3}\times\\
&&\int_{Z_{U(2)_v}\backslash U(2)_v} \mathcal{B}_{\tau_{1,v}}(\tau_{1,v}(g_v)f_{1,v}, f_{1,v})\mathcal{B}_{\tau_{2,v}}(\tau_{2,v}(g_v)f_{2,v}, f_{2,v})\\ 
&& \mathcal{B}_{\tau_{3,v}}(\tau_{3,v}(g_v)f_{3,v}, f_{3,v})dg_v.
\eeqnan
We remark that this is well-defined since $\omega_1\omega_2\omega_3\equiv 1.$  We also note that the constants in front of the respective integrals in the definitions of $\mathcal{J}_v$ and $\mathcal{I}_v$ agree by Propositions \ref{piad}, \ref{sigmaad}, \ref{muad}, and \ref{bigsig} in the appendix.  Recall that we have chosen $\tilde{f}_3$ so that $\tilde{f}_{3,v}\in \tau_v^+$ for all $v$.  Then by Lemma \ref{gu2locmat} we see that
\beqnan
&& \int_{\mathbb{P}B^\times_v} \mathcal{B}_{\Sigma_{1,v}}(\Sigma_{1,v}(b_v) \tilde{f}_{\Sigma_1,v}, \tilde{f}_{\Sigma_1,v})\mathcal{B}_{\Sigma_{2,v}}( \Sigma_{2,v}(b_v) \tilde{f}_{\Sigma_2,v}, \tilde{f}_{\Sigma_2,v})\\
&&\mathcal{B}_{\Sigma_{3,v}} (\Sigma_{3,v}(b_v) \tilde{f}_{\Sigma_3,v}, \tilde{f}_{\Sigma_3,v})\ db_v\\
=&& \int_{(\mathbb{P}B^\times_v)^+} \mathcal{B}_{\Sigma_{1,v}}(\Sigma_{1,v}(b_v) \tilde{f}_{\Sigma_1,v}, \tilde{f}_{\Sigma_1,v})\mathcal{B}_{\Sigma_{2,v}}( \Sigma_{2,v}(b_v) \tilde{f}_{\Sigma_2,v}, \tilde{f}_{\Sigma_2,v})\\
&&\mathcal{B}_{\Sigma_{3,v}} (\Sigma_{3,v}(b_v) \tilde{f}_{\Sigma_3,v}, \tilde{f}_{\Sigma_3,v})\ db_v\\
=&& \int_{Z_{U(2)_v}\backslash U(2)_v} \mathcal{B}_{\tau_{1,v}}(\tau_{1,v}(g_v)f_{1,v}, f_{1,v})\mathcal{B}_{\tau_{2,v}}(\tau_{2,v}(g_v)f_{2,v}, f_{2,v})\\ 
&& \mathcal{B}_{\tau_{3,v}}(\tau_{3,v}(g_v)f_{3,v}, f_{3,v})dg_v.
\eeqnan
Here, $dg_v$ is the Haar measure derived from $db_v$ under the isomorphism $(\mathbb{P}B^\times)^+\cong Z_{U(2)_v}\backslash U(2)_v$.
So, we see that $\mathcal{J}_v=\mathcal{I}_v$, and the $\mathcal{I}_v$ are defined completely in terms of data for $U(2)_v$.

Finally, we give the following restatement of Ichino's triple product formula.  Note that since $\tau_3$ is assumed to be dihedral with respect to $E/F$, we have $BC(\tau_3)\cong \pi(\chi_1, \chi_2)$, the principal series representation on $GL_2(\A_E)$.
\begin{corollary}\label{ichtrip} Let $f_i\in\tau_i$ and $\tilde{f}_i\in\tilde{\tau}_i$ be as above.  Then
$$
\left| \int_{[G_2]} (f_1f_2f_3)(g)\ dg\right|^2 = \frac{\zeta_F(2)^2L_E(1/2,BC(\tau_1)\boxtimes BC(\tau_2)\boxtimes \chi_1)L_F(1,\chi_{E/F})^3}{2\prod_{i=1}^3 |X(\tau_i)| L_F(1,\tau_1,\operatorname{Ad})L_F(1,\tau_2,\operatorname{Ad})L_F(1,\tau_3,\operatorname{Ad})}\prod_v \mathcal{I}_v.
$$
\end{corollary}
\begin{proof}
First we show that
$$
\int_{[G_2]} (f_1f_2f_3)(g)dg = 2\int_{\mathbb{P}B^\times(F)^+\backslash\mathbb{P}B^\times(\A_F)^+} (\tilde{f}_{1,\Sigma_1}\tilde{f}_{2,\Sigma_2}\tilde{f}_{3,\Sigma_3})(b) db.
$$
Since the product of the central characters of the $\tau_i$ is trivial, we see that $f_1f_2f_3$ is really a function on $G_2(F)Z_{G_2}(\A_F)\backslash G_2(\A_F)$.  So we have
$$
\int_{[G_2]} (f_1f_2f_3)(g)dg = \operatorname{Vol}([Z_{G_2}]) \int_{G_2(F)Z_{G_2}(\A_F)\backslash G_2(\A_F)} (f_1f_2f_3)(g)dg.
$$
Here, $\operatorname{Vol}(Z_{G_2})$ is computed with respect to the Tamagawa measure on $Z_{G_2}$.
Note there is a natural identification of the spaces $G_2(F)Z_{G_2}(\A_F)\backslash G_2(\A_F)$ and $\mathbb{P}B^\times(F)^+\backslash \mathbb{P}B^\times(\A_F)^+$.  This gives
$$
\int_{G_2(F)Z_{G_2}(\A_F)\backslash G_2(\A_F)} (f_1f_2f_3)(g)dg = \int_{\mathbb{P}B^\times(F)^+\backslash \mathbb{P}B^\times(\A_F)} (\tilde{f}_{1,\Sigma_1}\tilde{f}_{2,\Sigma_2}\tilde{f}_{3,\Sigma_3})(b)db.
$$
Since $\operatorname{Vol}([Z_{G_2}])=2$, this proves the claim.

Finally, by invoking Propositions \ref{piad}, \ref{sigmaad}, \ref{muad}, and Corollary \ref{sigcor} in the appendix, the proof is complete.
\end{proof}

\section{The $\Theta$-Correspondence for Unitary Groups}\label{thetachapt}

While the first case $(n=1)$ of Conjecture \ref{theconjecture} follows rather easily from Waldspurger's theorem, a proof of the conjecture for $n=2$ will require significantly more work.  In fact, we will only be able to prove the conjecture for a restricted class of representations. 

In this section, we will first introduce the Weil Representation and $\Theta$-correspondence.  Given two unitary groups $U(V)$ and $U(W)$ and a cuspidal automorphic representation $\pi$ of $U(V)(\A_F)$, the $\Theta$-correspondence constructs for us a cuspidal automorphic representation of $U(W)(\A_F)$.

The $\Theta$-correspondence is useful for us because it relates the period integral we wish to compute to a more familiar integral: Ichino's triple product integral.  The relevant cartoon is the following seesaw diagram:
$$
\xymatrix{U(V_1\oplus V_2)\ar @{-}[d]\ar @{-}[dr]& U(W)\times U(W)\ar @{-}[dl]\ar @{-} [d]\\ U(V_1)\times U(V_2) & U(W)}
$$
Here, the vertical bars denote containment, and the oblique bars denote members of a so-called \emph{dual reductive pair} in $Sp((V_1\oplus V_2)\otimes W)$, which we define presently:
\begin{definition} Let $G,H\subset Sp(W)$ be subgroups.  Then $(G,H)$ is called a \emph{dual reductive pair} if
\begin{itemize}
\item $G$  is the centralizer of $H$ in $Sp(W)$, and vice versa.
\item The actions of $G$ and $H$ on $W$ are completely reducible.
\end{itemize}
\end{definition}

Let $\pi_W,\pi_{V_1},\pi_{V_2}$ be cuspidal automorphic representations of $U(W),U(V_1)$ and $U(V_2)$.  Then the $\Theta$-correspondence gives us representations $\Theta(\pi_W)$ of $U(V_1\oplus V_2)$ and $\Theta(\pi_{V_1})\otimes \Theta(\pi_{V_2})$ of $U(W)\times U(W)$.

Seesaw duality tells us that integrating a cusp form $\theta(f_{W})\in \Theta(\pi_W)$ against a pair of cusp forms $f_{V_1}$ and $f_{V_2}$ is the same as integrating $\theta(f_{V_1})$ and $\theta(f_{V_2})$ against $f_W$.  The first integral described is essentially our period integral, while the second is the triple-product integral.  This can be succinctly described by the following commutative diagram:
\beqn\label{comm}
\xymatrix{\omega_{\psi,\gamma}\otimes\bar{\pi}_{V_1}\otimes\bar{\pi}_{V_2}\otimes\bar{\pi}_W\ar^{\mathcal{T}}[r]\ar^{\mathcal{T'}}[d]& \Theta(\pi_{V_1})\otimes\Theta(\pi_{V_2})\otimes\bar{\pi}_W\ar^{\mathcal{I}}[d]\\ \bar{\pi}_{V_1}\otimes\bar{\pi}_{V_2}\otimes\Theta(\pi_W)\ar^{\mathcal{P}}[r]& \C}
\eeqn
Here, $\mathcal{T}$ and $\mathcal{T'}$ denote $\Theta$ lifts, $\mathcal{P}$ is our global period integral, and $\mathcal{I}$ is the triple product integral.

Recall that the Refined Gross-Prasad Conjecture relates $\mathcal{P}$ to $\prod_v \mathcal{P}_v$, the product of integrals of local matrix coefficients.  We know from various multiplicity-one results that:
\beqnan
\mathcal{T}&\approx& \otimes_v \mathcal{T}_v\\
\mathcal{T}'&\approx&\otimes_v\mathcal{T}'_v\\
\mathcal{I}&\approx& \prod_v \mathcal{I}_v
\eeqnan
where $\approx$ denotes equality up to a constant of proportionality.  If we can compute the exact constants of proportionality, then this and diagram \ref{comm} will provide us with the constant of proportionality between $\mathcal{P}$ and $\prod_v\mathcal{P}_v$.

So how do we find these constants of proportionality?  For $\mathcal{T}$ and $\mathcal{T}'$, we need several incarnations of the Rallis inner-product formula.  For two of these Rallis inner-product formulae, we simply invoke results of Michael Harris.  For the other, we will give a proof, which follows from Victor Tan's regularized Siegel-Weil Formula.  For $\mathcal{I}$, we use Ichino's work on the triple product integral.

The first part of this chapter will be spent introducing both the local and global theta correspondences.  For the local correspondence, we refer the reader to \cite{coho1} and \cite{hks}.  However, before discussing the $\Theta$-correspondence, we'll give a brief overview of the Weil Representation, both globally and locally.

\subsection{The Weil Representation for Unitary Groups}

As before, $E/F$ is a quadratic extension of number fields.  Let $V$ be a hermitian space over $E$ of dimension $m$, and $W$ a skew-hermitian space over $E$ of dimension $n$.

For brevity, we will treat the global and local Weil representation simultaneously.  For an algebraic group $G$, where the same statements can be made globally and locally, and if there is no risk of confusion, we will not make reference to both $G(F_v)$ and $G(\A_F)$, but only to $G$.

We denote
\beqnan
G &:=& U(V)\\
H &:=& U(W)
\eeqnan
viewed as algebraic groups over $F$.  We also consider the space
$$
\mathbb{W} := \operatorname{Res}_{E/F} V\otimes_E W
$$
along with a complete polarization
$$
\mathbb{W} = \mathbb{X}\oplus \mathbb{Y}.
$$

Denote by $\langle\cdot,\cdot\rangle_V$ and $\langle\cdot,\cdot\rangle_W$ the hermitian and skew-hermitian forms for $V$ and $W$ respectively. We equip $\mathbb{W}$ with the following symplectic form:
$$
\langle\cdot,\cdot\rangle_{\mathbb{W}} := \operatorname{tr}_{E/F}\left(\langle\cdot,\cdot\rangle_V\otimes {\langle\cdot,\cdot\rangle}_W\right).
$$
So, we consider the associated isometry group $Sp(\mathbb{W})$, along with the metaplectic cover $\widetilde{Sp}(\mathbb{W}).$  We have the following short exact sequence:
$$
1\to \C^\times\to\widetilde{Sp}(\mathbb{W})\to Sp(\mathbb{W})\to 1.
$$

Now, after fixing a an additive character $\psi:\A_F/F\to\C^\times$ (globally) or $\psi:F_v\to\C^\times$ (locally),
we have a Schr\"odinger model of the Weil Representation of $\widetilde{Sp}(\mathbb{W})$ on $\mathcal{S}(\mathbb{X})$, where $\mathcal{S}$ denotes the space of Schwartz-Bruhat functions.

A priori, we have embeddings
\beqna
\iota_W:G &\hookrightarrow& Sp(\mathbb{W})\label{firstemb}\\
\iota_V :H &\hookrightarrow& Sp(\mathbb{W})\label{secondemb}
\eeqna
which induces a map
$$
\iota_{W,V}:G\times H\to Sp(\mathbb{W}).
$$
\begin{remark}While the maps $\iota_W$ and $\iota_V$ are embeddings, the induced map $\iota_{W,V}$ is never an embedding.

\end{remark}
In order to obtain a splitting homomorphism $G\times H\to \widetilde{Sp}(\mathbb{W})(\A_F)$, we need a pair of characters $(\gamma_V,\gamma_W)$ of $\A_E^\times/E^\times$ (or $E_v$ in the local case) such that
\beqnan
\gamma_V|_{\A_F^\times\text{ or }F_v^\times}&=&\chi_{E/F}^m\\
\gamma_W|_{\A_F^\times\text{ or }F_v^\times}&=&\chi_{E/F}^n
\eeqnan
where, as usual, the character $\chi_{E/F}$ is the quadratic character of $\A_F^\times/F^\times$ or $F_v$ associated to $E/F$ by class field theory.
These characters give us splitting homormorphisms
\beqnan
\iota_{W,\gamma_W}:G&\to&\widetilde{Sp}(\mathbb{W})\\
\iota_{V,\gamma_V}:H&\to&\widetilde{Sp}(\mathbb{W})
\eeqnan
which then induces
$$
\iota_{W,\gamma_W,V,\gamma_V}:G\times H\to \widetilde{Sp}(\mathbb{W}).
$$

With this splitting map, we can compose with $\omega_\psi$ to get a Weil representation of $G\times H$ realized on $\mathcal{S}(\mathbb{X})$:
$$
\omega_{W,\gamma_W,V,\gamma_V,\psi} := \omega_\psi\circ \iota_{W,\gamma_W,V,\gamma_V}.
$$

For simplicity, we shall abide by a certain convention when choosing splitting characters.  Globally, let $\gamma$ be a character of $\A_E^\times/E^\times$ such that $\gamma|_{\A_F^\times}=\chi_{E/F}$.  When choosing $\gamma_W$ and $\gamma_V$ as above, we simply take $\gamma_W := \gamma^{\dim W}$ and $\gamma_V = \gamma^{\dim V}$.  We also follow the analogous local convention.

\begin{remark} The convention above does not lead to any real loss of generality, if we also consider twists of the theta-lifts by characters.
\end{remark}

\subsection{The Local $\Theta$-Correspondence}  For this section, we fix a place $v$ of $F$ and omit it from the notation, so that $F=F_v$.  As usual, $E$ is a quadratic extension of $F$; in the case that $v$ splits, we have $E=F\oplus F$.  Let $\chi_{E/F}$ be the character associated to $E/F$ by class field theory.  In the split case, $\chi_{E/F}$ is trivial. 

\subsubsection{Howe Duality}
Suppose that $(G,G')$ is a dual reductive pair of unitary groups in some symplectic group $Sp(\mathbb{W})$.  After fixing the characters $\psi$ and $\gamma$ as described above, we obtain a Weil representation $(\omega_{\psi,\gamma},\mathcal{S})$ of $G\times G'$. Let $\pi$ be an irreducible admissible representation of $G$.  We let $\mathcal{S}(\pi)$ be the maximal quotient of $\mathcal{S}$ on which $G$ acts as a multiple of $\pi$.  Then we have:
$$
\mathcal{S}(\pi)\cong\pi\otimes \Theta(\pi)
$$
where $\Theta(\pi)$ is a representation of $G'$.  We simply set $\Theta(\pi)=0$ if $\pi$ does not occur as a quotient of $\mathcal{S}$.  The \emph{Howe Duality Principle} states:
\begin{enumerate}
\item $\Theta(\pi)$ is a finitely generated admissible representation of $G'$.
\item $\Theta(\pi)$ has a unique proper maximal $G'$-invariant subspace and a unique irreducible quotient $\theta(\pi)$.
\item The correspondence $\pi\mapsto\theta(\pi)$ gives a bijection between the irreducible admissible representations of $G$ and $G'$ that occur as quotients of $\mathcal{S}$.
\end{enumerate}
The first assertion is known, due to \cite{kudla} and \cite{mvw}.  The last two assertions are known for $v\neq 2$, due to \cite{wald2}.  Furthermore, for $v=2$, the second two assertions are easily checked in the low rank examples considered here.

\subsection{The Global $\Theta$-Correspondence}

Globally, the $\Theta$-correspondence is realized using $\Theta$-series.  For any $\varphi\in\mathcal{S}(\mathbb{X}(\A_F))$, we define the theta kernel as
$$
\theta(g,h,\varphi) := \sum_{\lambda\in \mathbb{X}(F)}\omega_{W,\gamma_W,V,\gamma_V,\psi}(g,h)(\varphi)(\lambda).
$$

If $f$ is some cusp form on $G(\A_F)$, we define:
\beqn\label{strange}
\theta(f,\varphi)(h) := \int_{[G]} \theta(g,h,\varphi)\overline{f(g)}\ dg
\eeqn
where $dg$ is the Tamagawa measure.

With all of this, we can define the $\Theta$-lift of a cuspidal representation of $G$.
\begin{definition}
Let $\pi$ be a cuspidal automorphic representation of $G(\A_F)$, then
$$
\Theta_{V,W,\gamma_W,\gamma_V,\psi}(\pi) = \{\theta(f,\varphi): f\in\pi, \varphi\in\mathcal{S}(\mathbb{X}(\A_F))\}
$$
is the $\Theta$-lift of $\pi$ with data $(\gamma_W,\gamma_V,\psi).$
\end{definition}

\begin{remark} The reader may balk at the definition given in line \ref{strange}.  By integrating $\overline{f}$ (instead of merely $f$) against the theta series, we ensure that $\pi$ and $\Theta(\pi)$ have the same central characters.
\end{remark}

\subsection{The Rallis Inner Product Formula}
The goal of this section is to compute the constant of proportionality between $\mathcal{T}$ and $\prod_v\mathcal{T}_v$.  This involves several versions of the Rallis Inner Product Formula, which relates the Petersson inner product of two vectors to that of their $\Theta$-lifts.  We will need to use three different versions of the Rallis Inner Product formula, one for lifts from $U(1)$ to $U(2)$, one for lifts from $U(2)$ to $U(2)$, and one for lifts from $U(2)$ to $U(3)$.

Before giving the Rallis Inner Product Formulae we shall need, we give a brief discussion of the doubling method.

\subsubsection{The Doubling Method} We remind the reader that $V$ is a hermitian space over $E$ of dimension $m$, and $W$ is a skew-hermitian space of dimension $n$.  We will also consider the space $V^-$, which is the same space as $V$, but with hermitian form $-\langle\cdot,\cdot\rangle_V$.  We note that $U(V)=U(V^-)$.

The seesaw diagram relevant to this discussion is the following:
\beqn
\xymatrix{U(V\oplus V^-)\ar @{-}[d]\ar @{-}[dr]& U(W)\times U(W)\ar @{-}[dl]\ar @{-} [d]\\ U(V)\times U(V^-) & U(W)^\Delta}\label{doublingss}
\eeqn
which we shall call the \emph{doubling seesaw}.  $U(W)^\Delta$ simply denotes the diagonally embedded copy of $U(W)$ in $U(W)\times U(W).$  We shall also occasionally use the shorthand $G:=U(V)=U(V^-), H:= U(W)$ and $G^\diamond := U(V\oplus V^-)$.

The dual reductive pairs above live inside $\widetilde{Sp}(2\mathbb{W})$, where
$$
2\mathbb{W} := \mathbb{W}\oplus\mathbb{W}^-
$$
and $\mathbb{W}^-:= V^-\otimes_E W$.

In order to have a Weil representation, we will have to fix splitting characters for each of the dual reductive pairs in the doubling seesaw.  After choosing splitting characters for one of the diagonal segments above, the splitting characters will be fixed for the other diagonal segment.  We shall choose splitting characters for the dual reductive pair $(U(V)\times U(V^-),U(W)\times U(W))$.

Note that this pair can be viewed as two dual reductive pairs: $(U(V), U(W))$ in $\widetilde{Sp}(\mathbb{W})$ and $(U(V^-),U(W))$ in $\widetilde{Sp}(\mathbb{W}^-)$.  Suppose that we choose splitting characters $(\gamma_W, \gamma_V)$ for $(U(V), U(W))$ and $(\gamma_W, \gamma_V')$ for $(U(V^-), U(W))$.  With these choices made, we are forced to take $(\gamma_W, \gamma_V\gamma_V')$ as our splitting characters for the dual reductive pair $(U(V\oplus V^-), U(W)^\Delta)$.\footnote{This choice is `forced' in the sense that seesaw duality does not hold otherwise.}

Suppose that $\pi$ and $\pi'$ are cuspidal, irreducible, automorphic representations of $U(V)$ and $U(V^-)$, respectively.  Let $f\in\pi$ and $f'\in\pi'$.  With the splitting data $(\gamma_W,\gamma_V)$, $(\gamma_W,\gamma_V')$ and the additive character $\psi$, and Schwartz functions $\phi$ and $\phi'$, we can consider the $\Theta$-lifts $\theta(f, \varphi)$ and $\theta(f', \varphi')$, which are both cusp forms on $U(W)$.

The Rallis Inner Product Formula computes the following ratio:
$$
\frac{\langle\theta(f, \varphi), \theta(f', \varphi')\rangle}{\langle f,f'\rangle}
$$
where the pairings are the respective Petersson inner products, which are defined using the respective Tamagawa measures.  However, the RHS of the Rallis Inner Product Formula will contain division by $\prod_v \langle f_v, f'_v\rangle_v$.  Since we make the assumption that the global and local inner products are chosen compatibly, for our purposes the Rallis Inner Product Formula simply computes $\langle\theta(f, \varphi), \theta(f', \varphi)\rangle$ in terms of an $L$-value.

Note that unless $\theta(f, \varphi)$ is in the contragredient of $\Theta(\pi)$, the Petersson pairing $\langle\theta(f, \varphi), \theta(f', \varphi)\rangle$ will vanish.  So, we see that $\pi$ and $\pi'$ cannot be chosen independently.

What then, is the required relationship between $\pi$ and $\pi'$ to ensure the non-triviality of the inner product?  We note that
\beqnan
\theta(f, \varphi)&\in&\Theta_{V,W,\gamma_W,\gamma_V,\psi}(\pi)\\
\theta(f', \varphi') &\in&\Theta_{V^-,W,\gamma_W,\gamma_V',\psi}(\pi').
\eeqnan
We have the following technical result:
\begin{lemma}
$$
\Theta_{V^-,W,\gamma_W,\gamma_V',\psi}(\pi') = (\gamma_V\gamma_V'\circ i_{U(1)}\circ\operatorname{det}_{U(W)})\cdot \overline{\Theta_{V,W,\gamma_W,\gamma_V,\psi}((\gamma_W^2\circ i_{U(1)}\circ\operatorname{det}_{U(V^-)})\cdot \overline{\pi'})}
$$
where $i_{U(1)}$ is the inverse of the isomorphism $i_{U(1)}^{-1}:E^\times/F^\times\overset\sim\to U(1)\subset E^\times$ given by $e\mapsto \frac{e}{e^\tau}$, where $\tau$ is the generator of $\operatorname{Gal}(E/F)$.
\end{lemma}
\begin{proof} We begin by considering the Weil representation $\omega_{V^-,W,\gamma_W,\gamma_V',\psi}$ as a representation of $U(V)\times U(W)$.  (A priori, it is a representation of $U(V^-)\times U(W)$.  We are identifying $U(V)$ with $U(V^-)$ via the identity map on $GL(V)$.)  We note that
\beqnan
\omega_{V^-,W,\gamma_W,\gamma_V',\psi} &\cong& \omega_{V^-,W,\gamma_W,\gamma_V,\psi}\otimes \left(1_{U(V)}\boxtimes \frac{\gamma_V'}{\gamma_V}\circ i_{U(1)}\circ \operatorname{det}_{U(W)}\right)\\
&\cong& \omega_{V,W,\gamma_W,\gamma_V,\psi^{-1}} \otimes \left(1_{U(V)}\boxtimes \frac{\gamma_V'}{\gamma_V}\circ i_{U(1)}\circ \operatorname{det}_{U(W)}\right)\\
&\cong& \omega_{V,W,\gamma_W^{-1},\gamma_V^{-1},\psi^{-1}}\otimes (\gamma_W^2\circ i_{U(1)}\circ \operatorname{det}_{U(V^-)}\boxtimes\gamma_V\gamma_{V}'\circ i_{U(1)}\circ \operatorname{det}_{U(W)})\\
&\cong& \overline{\omega_{V,W,\gamma_W,\gamma_V,\psi}}\otimes (\gamma_W^2\circ i_{U(1)}\circ \operatorname{det}_{U(V^-)}\boxtimes\gamma_V\gamma_{V}'\circ i_{U(1)}\circ \operatorname{det}_{U(W)}).
\eeqnan
There is a surjection
$$
\omega_{V^-,W,\gamma_W,\gamma_V',\psi}\twoheadrightarrow \pi'\boxtimes \Theta_{V^-,W,\gamma_W,\gamma_V',\psi}(\pi')
$$
and therefore, from the series of isomorphisms above, we have
$$
\overline{\omega_{V,W,\gamma_W,\gamma_V,\psi}}\twoheadrightarrow (\gamma_W^{-2}\circ i_{U(1)}\circ\operatorname{det}_{U(V^-)})\cdot\pi'\boxtimes ((\gamma_V\gamma_V')^{-1}\circ i_{U(1)}\circ\operatorname{det}_{U(W)})\cdot\Theta_{V^-,W,\gamma_W,\gamma_V',\psi}(\pi')
$$
which immediately gives
$$
\omega_{V,W,\gamma_W,\gamma_V,\psi}\twoheadrightarrow (\gamma_W^2\circ i_{U(1)}\circ\operatorname{det}_{U(V^-)})\cdot\overline{\pi'}\boxtimes (\gamma_V\gamma_V'\circ i_{U(1)}\circ\operatorname{det}_{U(W)})\cdot\overline{\Theta_{V^-,W,\gamma_W,\gamma_V',\psi}(\pi')}.
$$
The surjection above tells us that
$$
(\gamma_V\gamma_V'\circ i_{U(1)}\circ\operatorname{det}_{U(W)})\cdot\overline{\Theta_{V^-,W,\gamma_W,\gamma_V',\psi}(\pi')} \cong \Theta_{V,W,\gamma_W,\gamma_V,\psi}(\gamma_W^2\circ i_{U(1)}\circ\operatorname{det}_{U(V^-)}\cdot\overline{\pi'})
$$
and therefore
$$
\Theta_{V^-,W,\gamma_W,\gamma_V',\psi}(\pi') = (\gamma_V\gamma_V'\circ i_{U(1)}\circ\operatorname{det}_{U(W)})\cdot \overline{\Theta_{V,W,\gamma_W,\gamma_V,\psi}((\gamma_W^2\circ i_{U(1)}\circ\operatorname{det}_{U(V^-)})\cdot \overline{\pi'})}.
$$
\end{proof}

Now, if $e\in U(1)\subset E^\times$, we observe that $i_{U(1)}^{-1}(e) = e^2$, so that $\gamma_W(e) = \gamma_W(i_{U(1)}(e^2))=\gamma_W^2(i_{U(1)}(e)).$  This means that
$$
\Theta_{V^-,W,\gamma_W,\gamma_V',\psi}(\pi') = (\gamma_V\gamma_V'\circ i_{U(1)}\circ\operatorname{det}_{U(W)})\cdot \overline{\Theta_{V,W,\gamma_W,\gamma_V,\psi}((\gamma_W|_{U(1)}\circ\operatorname{det}_{U(V^-)})\cdot \overline{\pi'})}.
$$
So if we take
$$
\pi'= (\gamma_W\circ\operatorname{det}_{U(V^-)})\cdot \overline{\pi}
$$
then
$$
\theta(f', \varphi')\in (\gamma_V\gamma_V'\circ i_{U(1)}\circ \operatorname{det}_{U(W)})\cdot \overline{\Theta_{V,W,\gamma_W,\gamma_V,\psi}(\pi)}.
$$
Therefore, the integral
$$
\int_{[H^\Delta]}\theta(f, \varphi)(h)\theta(f', \varphi')(h)\cdot((\gamma_V\gamma_V')^{-1}\circ i_{U(1)}\circ\operatorname{det}_{U(W)})(h)dh
$$
is the Petersson inner product of two vectors in $\Theta_{V,W,\gamma_W,\gamma_V,\psi}(\pi)$.  Seesaw duality is the statement that this integral is equal to
\beqn\label{preeis}
\int_{[G\times G]}\overline{f(g_1)}\overline{f'(g_2)}\theta_{\varphi\otimes\varphi'}((\gamma_V\gamma_V')\circ i_{U(1)}\circ\operatorname{det}_{U(W)})(g_1,g_2)dg_1dg_2
\eeqn
where
$$
\theta_{\varphi\otimes\varphi'}((\gamma_V\gamma_V')\circ i_{U(1)}\circ\operatorname{det}_{U(W)})\in \Theta_{V\oplus V^-,W,\gamma_W,\gamma_V\gamma_V',\psi}(\gamma_V\gamma_V'\circ i_{U(1)}\circ \operatorname{det}_{U(W)}).
$$
Since $f'\in\pi'=(\gamma_{W}|_{U(1)}\cdot\operatorname{det}_{U(V^-)})\cdot\pi^\vee$, we can take $f' =(\gamma_W|_{U(1)}\cdot \operatorname{det}_{U(V^-)})\cdot  \overline{f_2}$ for some $f_2\in\pi$.  (For consistency of notation, $f_1:=f$.)

Now, the $\Theta$-lift in the integral in line \ref{preeis} can actually be identified with a certain Eisenstein series $E(\Phi_s, g)$.  More specifically, the integral above can be identified with
\beqn\label{zetaint}
\int_{[G\times G]}  \overline{f_1(g_1)} {f_2(g_2)} E(\Phi_s, (g_1, g_2)) \gamma_W^{-1}(\operatorname{det}_{U(V^-)} g_2)dg_1dg_2
\eeqn
where
$\Phi_s$ is a member of the degenerate principal series $\operatorname{Ind}_{P(\A_F)}^{G^\diamond(\A_F)} (\gamma_W\circ\det)\cdot |\det|^s$, with $P$ the Siegel parabolic that preserves the diagonal $V^\Delta\subset V\oplus V^-$, the determinants are taken with respect to $GL(V^\Delta)$ (which is isomorphic to the Levi of $P$), and
$$
E(\Phi_s, g) := \sum_{x\in P(F)\backslash G^\diamond(F)}\Phi_s(xg)
$$
for $g\in G^\diamond$.

\begin{definition} The Piatetski-Shapiro-Rallis zeta integral is defined as
$$
Z(s,f_1,f_2,\Phi_s,\gamma_W) := \int_{[G\times G]} f_1(g_1)\overline{f_2(g_2)} E(\Phi_s, \iota(g_1, g_2)), \gamma_W^{-1}(\operatorname{det}_{U(V^-)} g_2)dg_1 dg_2.
$$
\end{definition}
We remark that the integral in line \ref{zetaint} is equal to $Z(s,\bar{f}_1, \bar{f}_2, \Phi_s, \gamma_W).$

It is important to note that invoking seesaw duality requires one to deal with convergence of the relevant integrals.  Unfortunately, these integrals do not always converge.  Certain regularization methods are sometimes required to relate the Petersson inner product of the $\Theta$-lifts to the zeta integral.  However, once this relationship between the Petersson inner product and zeta integral are established, there is a very nice result which relates $Z$ to an $L$-function.  Before giving this result, we need a bit more discussion.

Recall that $m:= \dim_E V$ and $n:=\dim_E W$.  Set
$$
d_m(s,\gamma_W) := \prod_{r=0}^{m-1} L(2s + m -r,\chi_{E/F}^{n+r}).
$$
We assume that $\Phi_s=\otimes_v \Phi_{s,v}$ and $f_i=\otimes_v f_{i,v}$.  We take $S$ to be a sufficiently large finite set of places of $F$ such that for all $v\notin S$, all relevant data is unramified, and the local vectors $f_{i,v}$ are normalized spherical vectors, with the additional property that $\langle f_{1,v}, f_{2,v}\rangle_{\pi_v}=1$.\footnote{We remind the reader that the local pairings $\langle\cdot,\cdot\rangle_{\pi_v}$ are chosen so that the product over all places $v$ gives the Petersson inner product $\langle\cdot,\cdot\rangle_{\pi}$.}  Let the $d_{m,v}$ be the local factors of $d_m$.  We have the following factorization of the zeta-integral (see \cite{coho1} as well as \cite{coho2}):
\begin{theorem} For $\operatorname{Re}(s) >> 0$,
$$
Z(s, f_1, f_2, \Phi_s, \gamma_W) = \prod_v Z_v(s,f_{1,v}, f_{2,v},\Phi_{s,v},\gamma_{W,v})
$$
where
$$
Z_v(s,f_{1,v}, f_{2,v},\Phi_{s,v},\gamma_{W,v}) := \int_{U(V)_v} \Phi_{s,v}((g_v, 1)){\langle \pi_v(g_v) f_{1,v}, f_{2,v}\rangle_{\pi_v}} (\gamma_{W,v}^{-1} \det g_v)dg_v.
$$
\end{theorem}
We note that the integral defining the $Z_v$ in the theorem above only converges for $\operatorname{Re}(s)$ sufficiently large.  The definition of the $Z_v$ is extended to all of $\C$ by meromorphic continuation.

Now, if we set $Z_S := \prod_{v\in S} Z_v$, then we have the following so-called `basic identity' (again, see \cite{coho1} and \cite{coho2}):
\begin{theorem}[Basic Identity] For $f_1,f_2\in\pi$, we have:
$$
{Z(s, f_1, f_2, \Phi_s, \gamma_W)} = [d_m(s,\gamma_W)]^{-1} Z_S(s, f_1, f_2, \Phi_s, \gamma_W)\cdot L^S(s+1/2,\pi\otimes\gamma_W).
$$
\end{theorem}

\subsubsection{Lifting from $U(1)$ to $U(2)$} Here, $\dim V=1$ and $\dim W = 2$.  In this case, $\pi$ is just a character, which we will denote by $\mu$.  In this section, we will not only consider $\Theta_{V,W,\gamma_W,\gamma_V, \psi}(\mu)$, but also its transfer to $GU(2)$, \`a la Theorem \ref{lift}.  Denote this by $\widetilde{\Theta(\mu)}$.  We remind the reader that $\gamma_W = \gamma^2$ and $\gamma_V = \gamma$ for some Hecke character $\gamma$ of $\A_E^\times$ extending $\chi_{E/F}$, per our convention.

The first incarnation of the Rallis Inner Product Formula that we will need is as follows:

\begin{theorem}[RIPF for $\Theta$-lifts from $U(1)$ to $U(2)$]\label{ripfu1u2} Suppose that $f_i=\otimes_v f_{i,v}, \varphi_i=\otimes_v \varphi_{i,v}, \Phi_s=\otimes_v \Phi_{s,v}$, and that $\Phi$ is a holomorphic section given by $[\delta(\varphi_1,\varphi_2)]$ in the notation of \cite{nonvan}, page 182.  Then:
$$
{\langle \theta(\bar{f}_1, \varphi_1), \theta(\bar{f}_2, \varphi_2)\rangle_{\Theta(\bar\mu)}} = 2\cdot\frac{L_E(1, BC(\mu)\otimes \gamma^2)}{\zeta_F(2)}\prod_v Z_v^\sharp(1/2, {f}_{1, v}, {f}_{2, v}, \Phi_{1/2, v}, \gamma^2_v)
$$
where
$$
Z_v^\sharp := \frac{\zeta_{F_v}(2)}{L_{E_v}(1,BC(\mu_v)\otimes\gamma^2_v)}\cdot Z_v.
$$
\end{theorem}
\begin{proof} First, we remark that the presence of complex conjugation bars over the $f_i$ is due to our normalization of the theta-correspondence.  In line 6.4.8 on page 710 of \cite{2x2}, Harris has the following:
\beqn\label{rallis12}
{\langle \widetilde{\theta(\bar{f}_1, \varphi_1)}, \widetilde{\theta(\bar{f}_2,\varphi_2)}\rangle_{\widetilde{\Theta(\bar\mu)}}} = \frac{L_E^S(1, BC(\mu))}{\zeta_F^S(2)}\prod_{v\in S} Z_v(1/2, {f}_{1, v}, {f}_{2, v}, \Phi_{1/2, v})
\eeqn
Using Remark 4.20 of \cite{hiraga}, we have
$$
{\langle \theta(\bar{f}_1, \varphi_1), \theta(\bar{f}_2, \varphi_2)\rangle_{\Theta(\bar\mu)}} = |X(\Theta(\bar\mu))|\cdot\frac{L_E(1, BC(\mu))}{\zeta_F(2)}\prod_v Z_v^\sharp(1/2, {f}_{1, v}, {f}_{2, v}, \Phi_{1/2, v})
$$
where $X(\Theta(\bar\mu))$ is the set of automorphic characters $\omega$ of $GU(2)(\A_F)/U(2)(\A_F)$ such that $\widetilde{\Theta(\bar\mu)}\otimes\omega\cong \widetilde{\Theta(\bar\mu)}.$  We note that $|X(\Theta(\bar\mu))|=2.$

Now, this appears different from the claimed result for a simple (albeit subtle) reason:  Harris has used a particular normalization of the theta correspondence that is slightly different from ours.  When choosing a pair of splitting characters $(\gamma_V,\gamma_W)$, he chooses $\gamma_W$ to be the trivial character.  However, there is no real loss of generality here, since we can merely replace $\mu$ with $\mu\otimes\gamma|_{U(1)}$ and then follow Harris' conventions.
\end{proof}

\subsubsection{Lifting from $U(2)$ to $U(2)$}
Here, $\dim W=\dim V = 2$.  With a sufficiently large set of places $S$, and the same assumptions made for $v\notin S$ as in the previous section, the Rallis Inner Product Formula is stated in line 1.3.5 of \cite{close}.  Once again, the presence of complex conjugation bars in the result below is due to our normalization of the theta-correspondence.
\begin{theorem}[RIPF for lifts from $U(2)$ to $U(2)$]\label{ripfu2u2} Suppose that $f_i=\otimes_v f_{i,v}, \varphi_i=\otimes_v \varphi_{i,v}, \Phi_s=\otimes_v \Phi_{s,v}$, and that $\Phi$ is a holomorphic section given by $[\delta(\varphi_1,\varphi_2)]$ in the notation of \cite{nonvan}, page 182.  Then we have
$$
{\langle \theta(\bar{f}_1, \varphi_1), \theta(\bar{f}_2, \varphi_2)\rangle_{\Theta(\bar\pi)}}=\frac{L_E(1/2, BC(\pi)\otimes\gamma^2)}{L_F(1,\chi_{E/F})\zeta_F(2)}\prod_{v} Z^\sharp_v(0, {f}_{1,v}, {f}_{2,v}, \Phi_{0,v}, \gamma^2_v)
$$
where
$$
Z_v^\sharp := \frac{L_{F_v}(1, \chi_{E_v/F_v})\zeta_{F_v}(2)}{L_{E_v}(1/2, \pi_v\otimes\gamma_v^2)}Z_v.
$$
\end{theorem}

\subsubsection{Lifting from $U(2)$ to $U(3)$}
We have $\dim_E V =2$ and $\dim_E W=3$.  If $W$ is anisotropic, then $[H]$ is compact, and the Rallis Inner Product Formula we need follows from a Siegel-Weil Formula (Theorem 1.1 in \cite{sw}).  However, if $W$ is not anisotropic (so that $H(\A_F)$ is quasi-split), the theta integral in the Siegel-Weil formula does not converge.  In this case, the Rallis Inner Product Formula follows from Tan's Regularized Siegel-Weil formula.  What follows is a brief summary of \cite{tan}.  We encourage the interested reader to consult Tan's paper for further details.

We again have occasion to consider the following degenerate principal series representation of $G^\diamond(\A_F)$:
$$
I(s,\gamma) := \operatorname{Ind}_{P(\A_F)}^{G^\diamond(\A_F)} \gamma^3||\cdot||^s_{\A_E}\circ\det
$$
where we remind that reader that $P$ is the parabolic preserving the diagonal $$V^\Delta:=\{(v,v):v\in V\}\subset V\oplus V^-.$$
Given $\Phi_s\in I(s, \gamma)$, we define the Siegel-Eisenstein series
$$
\mathcal{E}(g,\Phi_s) := \sum_{\varepsilon\in P(F)\backslash G^\diamond(F)} \Phi_s(\varepsilon g).
$$

There is a maximal compact subgroup $K\subset G^\diamond(\A_F)$ such that we have the decomposition:
$$
G^\diamond(\A_F) = P(\A_F) K.
$$
We call $\Phi_s$ a \emph{standard section} if its restriction to $K$ is independent of $s$.  For a standard section $\Phi_s$, the Siegel-Eisenstein series $\mathcal{E}(g,\Phi_s)$ converges for $\operatorname{Re}(s)>1$, and has a meromorphic continuation to $\C$.  Furthermore, for each $s\in\C$ where it is holomorphic, $\mathcal{E}(g,\Phi_s)$ is an automorphic form on $G^\diamond(\A_F)$.  We take $\Phi_s$ to be a standard section in the sequel.

The Eisenstein series $\mathcal{E}(g,\Phi_s)$ has at worst a simple pole at $s=1/2$.  So we write its Laurent expansion as:
$$
\mathcal{E}(g,\Phi_s) = \frac{A_{-1}(g,\Phi)}{s-1/2} + A_0(g,\Phi) +\dots
$$
Before defining the theta integral, we remind the reader of the setup for the Weil representation.  We have
$$
\mathbb{W} := \operatorname{Res}_{E/F} 2V\otimes_E W
$$
where we remind the reader that $2V := V\oplus V^-$.  We set
$$
V^\nabla := \{(v,-v):v\in V\}\subset V\oplus V^-.
$$
Having fixed the characters $\psi$ and $\gamma$, we have a Schr\"odinger model of the Weil representation of $G^\diamond(\A_F)\times H(\A_F)$ realized on $\mathcal{S}((V^\nabla\otimes W)(\A_F)).$

Now, if we fix polarizations
\beqnan
V &=& X^+\oplus Y^+\\
V^- &=& X^-\oplus Y^-
\eeqnan
and denote
\beqnan
X &:=& X^+\oplus X^-\\
Y &:=& Y^+\oplus Y^-
\eeqnan
then we obtain another polarization of $\mathbb{W}.$  We have:
$$
2V = X\oplus Y
$$
and therefore
$$
\mathbb{W} = (X\otimes W)\oplus (Y\otimes W).
$$
We denote
\beqnan
\mathbb{X} &:=& X\otimes W\\
\mathbb{X}^+ &:=& X^+\otimes W\\
\mathbb{X}^- &:=& X^-\otimes W.
\eeqnan
There is a $U(V)(\A_F)\times U(V^-)(\A_F)$-intertwining map
$$
\sigma: \mathcal{S}(\mathbb{X}^+(\A_F))\otimes\mathcal{S}(\mathbb{X}^-(\A_F))\to \mathcal{S}(\mathbb{X}(\A_F))\to \mathcal{S}((V^\nabla\otimes W)(\A_F))
$$
where the first map is the obvious one, and the second map is given by a Fourier transform.

We now define the \emph{theta integral} as follows:
$$
I(g,\varphi) := \int_{[H]} \theta(g,h,\varphi)\gamma^{-2}(\det h)\ dh
$$
where $\varphi\in\mathcal{S}(V^\nabla\otimes W)(\A_F)),g\in G^\diamond(\A_F)$, and $dh$ is the Tamagawa measure on $H(\A_F)$.

For $g\in G^\diamond(\A_F)$, we have $g=pk$, with $p\in P(\A_F)$, and $k\in K$.  With the right choice of basis, we have
$
p = \begin{pmatrix} a & *\\ 0 & ^t\bar{a}^{-1} \end{pmatrix}
$
for some $a\in GL_2(\A_E).$  Write $|a(g)| := |\det a|_{\A_E}$.
If $W$ is anisotropic, and $\Phi_s$ is chosen such that $\Phi_s(g) = |a(g)|^{s-1/2}\omega_{\psi,\gamma}(g)\varphi(0)$, then $\mathcal{E}(g,\Phi_s)$ is holomorphic at $s=1/2$, and the theta integral defined above converges.  Indeed, Theorem 1.1 of \cite{sw} says that
$$
\mathcal{E}(g,\Phi_{1/2}) = I(g,\varphi).
$$
However, as mentioned above, if $H(\A_F)$ is quasi-split (i.e. $W$ is not anisotropic), the theta integral does not necessarily converge.  We'll have to `regularize' it so that we can think of it as a meromorphic function of a complex variable.  The Regularized Siegel-Weil Formula relates the Laurent expansions of this yet-to-be-defined regularized theta integral and the Siegel Eisenstein series.

Let $v$ be an odd place of $F$ such that all relevant data is unramified.  Tan finds a Hecke operator $z$ in the Hecke algebra of $G^\diamond_v$ that is used in the definition of the regularized theta integral.

We also need an auxiliary Eisenstein series.  Let $B_H$ be a Borel subgroup of $H$.  Then we have
$$
B_H = M_HN_H
$$
where $M_H$ is the Levi component of $B_H$, and $N_H$ is the unipotent radical.  We know that $M_H(\A_F)\cong \A_E^\times\times\A_E^{\times,1}.$  For $s\in\C$, let $\mu_s$ be the character of $M_H(\A_F)$ defined by $\mu_s(x,u):= ||x||^s_{\A_E}$.  We extend $\mu_s$ to all of $B_H$ by triviality on $N_H$.  We consider the induced representation
$$
I^{Aux}(s) := \operatorname{Ind}_{B_H}^H \mu_s.
$$
Let $K_H\subset H$ by a maximal compact subgroup such that $H=B_HK_H$.  Let $\Phi^{Aux}_s\in I^{Aux}(s)$ be the normalized $K_H$-fixed vector (i.e. $\Phi^{Aux}_s(k)=1$ for all $k\in K_H$).  Then the auxiliary Eisenstein series we need is defined by
$$
E(h,\Phi^{Aux}_s) := \sum_{\varepsilon\in B_H(F)\backslash H(F)} \Phi^{Aux}_s(\varepsilon h).
$$
It is known that $E(h,\Phi^{Aux}_s)$ converges for $\operatorname{Re}(s)$ sufficiently large, and has meromorphic continuation to all of $\C$.  Furthermore, $E(h,\Phi^{Aux}_s)$ has a simple pole at $s=1$ which is independent of $h$; we denote this residue by
$$
\kappa := \underset{s=1}{\operatorname{Res}}\ E(h,\Phi^{Aux}_s).
$$
We now define a new theta integral which incorporates both the auxiliary Eisenstein series and Hecke operator:
$$
I(g,s,\omega_{\psi,\gamma}(z)\varphi):=\int_{[H]} \theta(g,h,\omega_{\psi,\gamma}(z)\varphi)E(h,\Phi^{Aux}_s)\gamma^{-2}(\det h)\ dh.
$$
With all of this in place, we can define the \emph{regularized theta integral}.  The only modification from the theta integral above is that we multiply by an appropriate factor to cancel the effect of the Hecke operator.
\begin{definition}  For $g\in G^\diamond(\A_F),s\in\C$, and $\varphi\in\mathcal{S}(V^\nabla\otimes W)(\A_F))$, the \emph{regularized theta integral} is given by
$$
\mathcal{I}(g,s,\varphi) := \frac{1}{\kappa}\cdot\frac{I(g,s,\omega_{\psi,\gamma}(z)\varphi)}{P_z(s)}
$$
where
$$
P_z(s) :=q_F^s+q_F^{-s}-q_F-q_F^{-1}.
$$
The observant reader will notice that this definition differs by a constant from the definition given by Tan in \cite{tan}.
\end{definition}
The regularized integral $\mathcal{I}(g,s,\varphi)$ has a double pole at $s=1$; so we write the Laurent expansion as
$$
\mathcal{I}(g,s,\varphi) = \frac{B_{-2}(g,\varphi)}{(s-1)^2}+\frac{B_{-1}(g,\varphi)}{s-1}+B_0(g,\varphi)+\dots
$$
where the $B_i(g,\varphi)$ are automorphic forms on $G^\diamond(\A_F)$.

In order to prove the version of the Rallis Inner Product Formula that we'll use later, we need a result of Tan which relates the second terms in the Laurent expansions of $\mathcal{I}(g,s,\varphi)$ and $\mathcal{E}(g,\Phi_s)$.
\begin{theorem}[Second term identity]\label{2nd} Suppose that $\Phi_s(k)=(\omega(k)\varphi)(0)$ for all $k\in K$  Then
$$
2\cdot A_0(g,\Phi) = B_{-1}(g,\varphi) + \Psi(g)
$$
where $\Psi$ is an automorphic form on $G^\diamond(\A_F)$ which satisfies
$$
 \int_{[G\times G]} {f_1(g_1)} \overline{f_2(g_2)} \Psi(\iota(g_1,g_2))dg_1dg_2 = 0
$$
for cusp forms $f_i$ on $G(\A_F)$.
\end{theorem}
\begin{proof} This follows from Theorem 1.2 and Proposition 5.1.1 in \cite{tan}, as well as Lemma 4.9 in \cite{tanapp}.  Note that our renormalization in the definition of the regularized theta integral eliminates the need for the constant $c$ found in \cite{tan}.
\end{proof}
\begin{remark}
Instead of using the Tamagawa measure on $H$ to define the theta integral (and regularized theta integral), Tan uses the measure which gives $[H]$ volume $1$.  However, he also takes $c=\frac{1}{\kappa}$.  Since we are using the Tamagawa measure -- which  gives $[H]$ volume $2$ -- and we have renormalized the regularized theta integral by $\frac{1}{\kappa}$, the theorem above is correct.
\end{remark}
We are now equipped to state and prove the Rallis Inner Product Formula.
\begin{theorem}[Rallis Inner Product Formula]\label{rip} Let $\pi$ be an irreducible, cuspidal, automorphic representation of $G(\A_F)$.  Let $f_1,f_2\in\pi$.  Let $\varphi_1\in\mathcal{S}(\mathbb{X}^+(\A_F))$ and $\varphi_2\in\mathcal{S}(\mathbb{X}^-(\A_F))$.  Let $\theta(\bar{f}_1,\varphi_1)$ and $\theta(\bar{f}_2,\varphi_2)$ be the theta-lifts of $\bar{f}_1$ and $\bar{f}_2$ to $H(\A_F)$, and suppose these lifts are cuspidal.  (Once again, we have included complex conjugation here to compensate for our different normalization of the theta correspondence.)  Set $\varphi:=\sigma(\varphi_1\otimes\overline{\varphi_2})$, and let $\Phi_s\in I(s,\gamma)$ be such that $\Phi_s(k)=(\omega(k)\varphi)(0)$ for all $k\in K$.  Then the following equality holds:
$$
{\langle\theta(\bar{f}_1,\varphi_1),\theta(\bar{f}_2,\varphi_2)\rangle_{\Theta(\bar\pi)}}= \frac{2\cdot L_E(1,BC(\pi)\otimes\gamma^3)}{\zeta_F(2)L_F(3,\chi_{E/F})}\prod_{v}Z_v^\sharp(1/2,{f}_{1,v},{f}_{2,v},\Phi_{1/2,v}, \gamma_v^3)
$$
where
$$
Z_v^\sharp := \frac{\zeta_{F_v}(2)L_{F_v}(3,\chi_{E_v/F_v})}{L_{E_v}(1,BC(\pi_v)\otimes\gamma_v^3)} Z_v.
$$
\end{theorem}
\begin{proof}  We begin with the argument in the case that $W$ is not anisotropic.  In this case we require the regularized Siegel-Weil formula.

We have the following equalities from the definitions of $I$ and $\mathcal{I}$ given above:
\beqnan
(*):=\lefteqn{\underset{s=1}{\operatorname{Res}}\int_{[G\times G]} {f_1(g_1)}\overline{f_2(g_2)}\mathcal{I}(\iota(g_1,g_2), s,\varphi) dg_1dg_2}\\ &=& \frac{1}{\kappa P_z(s)}\underset{s=1}{\operatorname{Res}}\int_{[G\times G]} {f_1(g_1)}\overline{f_2(g_2)} I(g,s,\omega_{\psi,\gamma}(z)\varphi)dg_1 dg_2\\
&=& \frac{1}{\kappa P_z(s)}\underset{s=1}{\operatorname{Res}}\int_{[G\times G]} {f_1(g_1)}\overline{f_2(g_2)} \int_{[H]} \theta(\iota(g_1,g_2), h, \omega_{\psi,\gamma}(z)\varphi)\\&& \times E(h,\Phi_s^{Aux})\gamma^{-2}(\det h)dh dg_1dg_2.
\eeqnan
By Corollary 2.3.2 in \cite{tan}, $\theta(\iota(g_1,g_2), h, \omega_{\psi,\gamma}(z)\varphi)$ is rapidly decreasing on $[H]$, and we may change the order of integration, so that we have
\beqnan
(*) &=& \frac{1}{\kappa P_z(s)}\underset{s=1}{\operatorname{Res}} \int_{[H]}\int_{[G\times G]} {f_1(g_1)}\overline{f_2(g_2)} \theta(\iota(g_1,g_2),h,\omega_{\psi,\gamma}(z)\varphi)\\
&&\times E(h,\Phi_s^{Aux})\gamma^{-2}(\det h)\ dg_1dg_2dh.
\eeqnan
Now, there is a Hecke operator $z'$ in $H_v$ (for some place $v$ of $F$) corresponding to $z$ such that $\omega_{\psi,\gamma}(z)$ and $\omega_{\psi,\gamma}(z')$ have the same action on $\varphi$.  (See section 2.2 of \cite{tan} for details.)  So we have
\beqnan
(*) &=& \frac{1}{\kappa P_z(s)}\underset{s=1}{\operatorname{Res}} \int_{[H]}\int_{[G\times G]} {f_1(g_1)}\overline{f_2(g_2)} \theta(\iota(g_1,g_2),h,\omega_{\psi,\gamma}(z')\varphi)\\
&&\times E(h,\Phi_s^{Aux})\gamma^{-2}(\det h)\ dg_1dg_2dh.
\eeqnan
We have
$$
\theta(\iota(g_1,g_2),h,\omega_{\psi,\gamma}(z')\varphi) = \int_{H_v} z'(h_v)\theta(\iota(g_1,g_2), hh_v,\varphi)dh_v.
$$
By plugging this in to the previous equation for $(*)$ and making a change of variables, we have
\beqnan
(*) &=& \frac{1}{\kappa P_z(s)}\underset{s=1}{\operatorname{Res}} \int_{[H]}\int_{[G\times G]} {f_1(g_1)}\overline{f_2(g_2)} \theta(\iota(g_1,g_2),h,\varphi)\gamma^{-2}(\det h)\\
&&\times \int_{H_v} z'(h_v)E(hh_v^{-1},\Phi_s^{Aux})\gamma^2(\det h_v)dh_v\ dg_1dg_2dh.
\eeqnan
But since
$$
\int_{H_v} z'(h_v)E(hh_v^{-1},\Phi_s^{Aux})\gamma^2(\det h_v)dh_v = P_z(s) E(h,\Phi_s^{Aux})
$$
(see the top of page 351 of \cite{tan}) we have
$$
(*) = \frac{1}{\kappa}\underset{s=1}{\operatorname{Res}} \int_{[H]}\int_{[G\times G]} {f_1(g_1)}\overline{f_2(g_2)} \theta(\iota(g_1,g_2),h,\varphi)E(h,\Phi_s^{Aux})\gamma^{-2}(\det h)dg_1dg_2dh.
$$
Now we use a Poisson summation formula
$$
\theta(\iota(g_1,g_2),h,\varphi) = \theta(g_1, h, \varphi_1)\theta(g_2,h,\overline{\varphi_2})\gamma^{-2}(\det h)
$$
to obtain
$$
(*) = \frac{1}{\kappa}\underset{s=1}{\operatorname{Res}} \int_{[H]} \theta(\bar{f}_1,\varphi_1)(h)\overline{\theta(\bar{f}_2, \varphi_2)(h)} E(h,\Phi_s^{Aux})dh.
$$
Then, since $\kappa := \underset{s=1}{\operatorname{Res}}\ E(h,\Phi_s^{Aux})$, we see that
$$
(*) = {\langle\theta(\bar{f}_1,\varphi_1),\theta(\bar{f}_2,\varphi_2)\rangle_{\Theta(\pi)}}.
$$

Now, returning to the definition of $(*)$, we see that
$$
(*) =  \int_{[G\times G]} {f_1(g_1)}\overline{f_2(g_2)} B_{-1}(\iota(g_1,g_2),\varphi)dg_1dg_2
$$
and by Theorem \ref{2nd} we have
\beqnan
(*) &=& 2\int_{[G\times G]} {f_1(g_1)}\overline{f_2(g_2)} A_0(\iota(g_1, g_2), \Phi)dg_1dg_2\\
&& - \int_{[G\times G]} {f_1(g_1)}\overline{f_2(g_2)} \Psi(\iota(g_1,g_2))dg_1dg_2\\
&=& 2\int_{[G\times G]} {f_1(g_1)}\overline{f_2(g_2)} A_0(\iota(g_1, g_2), \Phi)dg_1dg_2.
\eeqnan
Finally, Lemma 4.7 and line 4.8 of \cite{tanapp} gives that the integral above is just $Z(1/2, f_1,f_2,\Phi_{1/2},\gamma^3)$.  Therefore, after applying the basic identity, we are done.

If $W$ is anisotropic (so that $[H]$ is compact), then the proof is much simpler; there is no need to regularize the Siegel-Weil formula in this case.  The result follows from Theorem 1.1 of \cite{sw}.
\end{proof}

\section{A local seesaw identity}\label{seesawchapt}

Everything in this chapter is local in nature, though we omit $v$ from the notation.  Whenever we refer to a group in this chapter, we really mean the $F_v$ points of the underlying algebraic group.

The purpose of this chapter is to provide an identity between the local integrals $\mathcal{P}_v$ considered in Conjecture \ref{theconjecture} and the $\mathcal{J}_v$ from Ichino's triple product formula.  The seesaw diagram which motivates the identity is:
$$
\xymatrix{U(V\oplus L)\ar @{-}[d]\ar @{-}[dr]& U(W)\times U(W)\ar @{-}[dl]\ar @{-} [d]\\ U(V)\times U(L) & U(W)}
$$
Here, $V$ is a two-dimensional hermitian space, $W$ is two-dimensional skew-hermitian spaces, and $L$ is a one-dimensional hermitian space.

We fix representations for the groups on the `bottom row' of the seesaw.  That is, let $\pi,\mu$, and $\sigma$ be irreducible, cuspidal representations of $U(V), U(L)$, and $U(W)$, respectively.  After fixing the appropriate splitting characters (which we suppress, for now), we also consider the Weil representation $\omega$ of $U(V\oplus L)\times U(W)$.

Let $\mathcal{B}_\pi,\mathcal{B}_\mu,\mathcal{B}_\sigma$ and $\mathcal{B}_\omega$ be pairings for the relevant representation.  Inspired by the analogous global seesaw duality property, one hopes to consider matrix coefficients for $\sigma,\pi,\mu$, and $\omega$.  Then, by showing that the integral
\beqnan
\int_{U(V)\times U(L)\times U(W)} &\mathcal{B}_\pi(\pi(g)f_\pi, f_\pi) \mathcal{B}_\mu(\mu(\ell)f_\mu, f_\mu)\mathcal{B}_\sigma(\sigma(h)f_\sigma, f_\sigma)\\
&\mathcal{B}_\omega(\omega((g,\ell)h)f_\omega, f_\omega)\ dgd\ell dh
\eeqnan
converges absolutely, once can use Fubini's theorem to arrive at a local seesaw identity.  On one side of this hypothetical local seesaw identity would be an integral of matrix coefficients for $\Theta(\sigma),\pi$, and $\mu$, and on the other would be an integral of matrix coefficients for $\sigma,\Theta(\pi)$, and $\Theta(\mu)$.  Alas, the convergence of the integral above does not hold.  However, all is not lost.  By ignoring $\mu$ and $U(L)$, and integrating only over $U(V)\times U(W)$, we obtain a convergent integral.

\begin{proposition}\label{localconv}
Suppose that $\pi$ and $\sigma$ are tempered.  Then the integral
$$
\int_{U(V)\times U(W)}\mathcal{B}_\pi(\pi(g)f_\pi, f_\pi)\mathcal{B}_\sigma(\sigma(h)f_\sigma, f_\sigma) \mathcal{B}_\omega(\omega(g,h)f_\omega, f_\omega)dgdh
$$
converges absolutely.
\end{proposition}
\begin{proof} First, we suppose that $E$ is a quadratic field extension of $F$.  We also suppose that neither $V$ nor $W$ is anisotropic.  We recall that we have Cartan decompositions
\beqnan
U(V) &=& K_VM_V^+K_V\\
U(W) &=& K_WM_W^+K_W
\eeqnan
where in this case we have
$$
M_V^+\cong M_W^+\cong \{x\in E^\times :|x|\leq 1\}.
$$
Following the proof of Proposition \ref{conv}, we see that the integral above is reduced to the convergence of
$$
\int_{M_V^+\times M_W^+} \mu_1(m_1)\mu_2(m_2)\mathcal{B}_\pi(\pi(m_1)f_\pi,f_\pi)\mathcal{B}_\sigma(\sigma(m_2)f_\sigma,f_\sigma)\mathcal{B}_\omega(\omega(m_1, m_2)f_\omega,f_\omega)dm_1dm_2
$$
where
$$
\mu_1(m) := \operatorname{Vol}(K_VmK_V)/\operatorname{Vol}(K_V)
$$
for $m\in M_V^+$, and
$$
\mu_2(m) := \operatorname{Vol}(K_WmK_W)/\operatorname{Vol}(K_W)
$$
for $m\in M_W^+$.
We know that $|\mu_1(m_1)|\leq A_1 |m_1|^{-1}$ and $|\mu_2(m_2)|\leq A_2 |m_2|^{-1}$, where $A_1,A_2$ are positive constants.  Furthermore, since $\pi$ and $\sigma$ are tempered, we know that
$$
|\mathcal{B}_\pi(\pi(m_1)f_\pi, f_\pi)|\leq C_1 |m_1|^{1/2} (1-\log |m_1|)^{r_1}
$$
and
$$
|\mathcal{B}_\sigma(\pi(m_2)f_\sigma, f_\sigma)|\leq C_2 |m_2|^{1/2} (1-\log |m_2|)^{r_2}
$$

where the $C_i$ and $r_i$ are positive constants.

For any $x\in E^\times$, we set
$$
\Upsilon(x) := \min(1, |x|^{-1}).
$$
We recall that for $\phi,\phi'\in \mathcal{S}(E)$, there is some $C>0$ such that for any $a\in E^\times$ we have
$$
\left|\int_E \phi(ax)\overline{\phi'(x)}dx\right|\leq C\cdot\Upsilon(a).
$$
Realizing $\omega$ on $\mathcal{S}(V\oplus L)$, we write $x=\begin{bmatrix} x_1\\ x_2\\ x_3\end{bmatrix}\in V\oplus L$ with $x_2\in L, (x_1,x_3)\in V$, and note that $\{(x_1,0)\},\{(0,x_3)\}\subset V$ are isotropic lines.  Then we have
$$
\omega(m_1,m_2)\Phi\left(\begin{bmatrix} x_1\\x_2\\x_3\end{bmatrix}\right) = \gamma(m_1) |m_1|^{3/2}\Phi\left(\begin{bmatrix} m_1m_2^{-1} x_1\\ m_1 x_2\\ m_1\overline{m_2} x_3\end{bmatrix}\right).
$$
So, we see that there is a positive constant $C$ such that for all $m_1\in M_V^+$ and $m_2\in M_W^+$ we have
$$
|B_\omega(\omega(m_1,m_2))f_\omega, f_\omega|\leq C |m_1|^{3/2} \Upsilon(m_1m_2^{-1})\Upsilon(m_1)\Upsilon(m_1\overline{m_2}).
$$
Note that for $|m_1|,|m_2|\leq 1$, we have $\Upsilon(m_1)=\Upsilon(m_1\overline{m_2})=1$.

So, we are reduced to checking the convergence of
$$
\int_{|m_1|, |m_2|\leq 1} |m_1|\cdot |m_2|^{-1/2}\cdot \Upsilon(m_1m_2^{-1})(1-\log |m_1|)^{r_1}(1-\log |m_2|)^{r_2} d^\times m_1 d^\times m_2.
$$
When $|m_1|\geq |m_2|$, the integrand is
$$
|m_2|^{1/2}(1-\log |m_1|)^{r_1}(1-\log |m_2|)^{r_2}
$$
which is bounded above by, say
$$
|m_1|^{1/4}|m_2|^{1/4}(1-\log |m_1|)^{r_1}(1-\log |m_2|)^{r_2}.
$$
When $|m_1|\leq |m_2|$, the integrand is
$$
|m_1|\cdot |m_2|^{-1/2}(1-\log |m_1|)^{r_1}(1-\log |m_2|)^{r_2}
$$
which is also bounded above by
$$
|m_1|^{1/4}|m_2|^{1/4}(1-\log |m_1|)^{r_1}(1-\log |m_2|)^{r_2}.
$$
The convergence of
$$
\int_{|m_1|,|m_2|\leq 1}|m_1|^{1/4}|m_2|^{1/4}(1-\log |m_1|)^{r_1}(1-\log |m_2|)^{r_2}d^\times m_1 d^\times m_2.
$$
follows from Lemma \ref{calc}.

Now we suppose that $W$ is anisotropic, but $V$ is not.  In that case, we need only check the convergence of
$$
\int_{|m_1|\leq 1} |m_1|(1-\log |m_1|)^{r_1}d^\times m_1,
$$
which follows from Lemma \ref{calc}.

If $V$ is anisotropic, but $W$ is not, then we are reduced to checking the convergence of
$$
\int_{|m_2|\leq 1} |m_2|^{1/2}(1-\log |m_2|)^{r_2}d^\times m_2,
$$
which also follows from Lemma \ref{calc}.

If both $V$ and $W$ are anisotropic, then there is nothing to check.

We now assume that $E=F\times F$.  In this case, we recall that both $U(V)$ and $U(W)$ are isomorphic to $GL_2(F)$.  With the right choice of bases, we have
$$
M_V^+\cong M_W^+\cong \{\operatorname{diag}(x,y): |x|\leq |y|\}.
$$
The proof in this case requires us to check many cases.  Before we can make use of Lemma \ref{calc}, we note that because both $\pi$ and $\sigma$ are tempered, we have constants $A,A'>0$ such that
\beqnan
|\mathcal{B}_\pi(\pi(g)f_\pi,f_\pi)|&\leq& A\cdot |g_1|^{1/2}|g_2|^{-1/2}\\
|\mathcal{B}_\sigma(\sigma(h)f_\sigma, f_\sigma)| &\leq& A'\cdot |h_1|^{1/2}|h_2|^{-1/2}
\eeqnan
where $g=\begin{pmatrix} g_1&\\&g_2\end{pmatrix}\in M_V^+$ and $h=\begin{pmatrix}h_1&\\&h_2\end{pmatrix}\in M_W^+$. Recall that
\beqnan
\mu_1(g) &=& \operatorname{Vol}(K_VgK_V)/\operatorname{Vol}(K_V)\\
\mu_2(2) &=& \operatorname{Vol}(K_WhK_W)\operatorname{Vol}(K_W)
\eeqnan
and that we have constants $B,B'>0$ such that
\beqnan
|\mu_1(g)| &\leq& B\cdot |g_1|^{-1}|g_2|\\
|\mu_2(h)| &\leq& B'\cdot |h_1|^{-1}|h_2|.
\eeqnan
Realizing $\omega$ on $\mathcal{S}(M_{2,3}(F))$, we have
$$
\omega(g,h)\phi\left(\begin{bmatrix}x_1& x_2& x_3\\ x_4 & x_5 & x_6\end{bmatrix}\right) = \det(g)^{-3/2}\det(h)\phi\left(\begin{bmatrix} g_1^{-1} h_1 x_1 & g_1^{-1} x_2 & g_1^{-1} h_2 x_3\\ g_2^{-1}h_1 x_4 & g_2^{-1} x_5 & g_2^{-1}h_2 x_3\end{bmatrix}\right).
$$

So, combining the various bounds mentioned above, we see that the integral whose convergence we must check is
\beqnan
\int_{|g_1|\leq |g_2|, |h_1|\leq |h_2|} &&|g_1|^{-2} |g_2|^{-1} |h_1|^{1/2} |h_2|^{3/2}(1-\sum\log |g_i|)^r(1-\sum\log |h_i|)^s\\ &&\Upsilon(h_1g_1^{-1})\Upsilon(h_1g_2^{-1})\Upsilon(h_2g_1^{-1})\Upsilon(h_2g_2^{-1})\Upsilon(g_1^{-1})\Upsilon(g_2^{-1})\\ &&d^\times g_1 d^\times g_2 d^\times h_1 d^\times h_2
\eeqnan
where $r,s>0$.
We will cut the region of integration into thirty regions, and verify the convergence of the integral in each of these regions.  The convergence of the integral above follows from cutting the region of integration into thirty regions (one for each comparative order of $g_1,g_2,h_1,h_2$ and $1$), and applying Lemma \ref{calc} in each case.  We omit the details.
\end{proof}

\begin{remark}\label{ram}
It is actually not necessary here to assume that $\pi$ and $\sigma$ are tempered.  Indeed, if we merely assume that $\pi$ and $\sigma$ are local components of some global cuspidal representations, then the Ramanujan bounds on matrix coefficients for $GL(2)$ are sufficient to prove that the integral in question converges absolutely.  We refer the reader to \cite{ks1} and \cite{ks2} for the bounds in question, as well as \cite{so4so5} for a similar argument.
\end{remark}

Having settled the issue of convergence, our local seesaw identity follows from Fubini's theorem.  In addition to the pairings already discussed, we consider pairings $\mathcal{B}_{\Theta(\sigma)},\mathcal{B}_{\Theta(\pi)}$, and $\mathcal{B}_{\Theta(\mu)}$ which are `inherited' from the pairings $\mathcal{B}_\sigma,\mathcal{B}_\pi, \mathcal{B}_\mu, \mathcal{B}_\omega$ and the local theta-correspondence.  For $\tau=\mu,\pi,\sigma$, we have surjective maps
$$
\omega \to \tau\boxtimes\Theta(\tau)
$$
(defined up to scaling) where $\Theta(\tau)$ is the `big' theta-lift from the previous chapter.  This induces a map
$$
\Theta:\tau^\vee\otimes\omega\to\Theta(\tau).
$$
Note that for the $\tau$ considered above, we have $\tau^\vee=\bar{\tau}$ since $\tau$ is unitary.  Now, we set
\beqnan
\mathcal{B}_{\Theta(\mu)}(\Theta(f_1,\varphi_1), \Theta(f_2, \varphi_2)) &:=& \int_{U(L)} \overline{\mathcal{B}_\mu(\mu(z) f_1, f_2)}\mathcal{B}_\omega(\omega(z)\varphi_1,\varphi_2) dz\\
\mathcal{B}_{\Theta(\pi)}(\Theta(f_1,\varphi_1), \Theta(f_2, \varphi_2)) &:=& \int_{U(V)} \overline{\mathcal{B}_\pi(\pi(g) f_1, f_2)}\mathcal{B}_\omega(\omega(g) \varphi_1, \varphi_2) dg\\
\mathcal{B}_{\Theta(\sigma)}(\Theta(f_1,\varphi_1), \Theta(f_2, \varphi_2)) &:=& \int_{U(W)} \overline{\mathcal{B}_\sigma(\sigma(h) f_1, f_2)} \mathcal{B}_\omega(\omega(h)\varphi_1, \varphi_2) dh
\eeqnan
where $f_i\in \mu,\pi$ or $\sigma$, respectively, and $\varphi_i\in\omega$.

We emphasize that the pairings above are defined on the `big' theta-lifts.  In order to make use of the following theorem, some work is required to show that these pairings descend to pairings on the `small theta-lifts.  This will be addressed in the following chapter.

\begin{theorem}\label{localseesaw} Let $Z_W\subset U(W)$ denote the center.  We let $\omega_{W, V\oplus L}$ denote the Weil representation of $U(W)\times U(V\oplus L)$.  For $\varphi = \varphi_V \otimes \varphi_L \in \omega_{W, V \oplus L} = \omega_{W,V} \otimes \omega_{W,L}$, and with the pairings as above, we have
\beqnan
\lefteqn{\int_{U(V)} \mathcal{B}_{\Theta(\sigma)}(\Theta(\sigma)(g) \Theta(f_\sigma, \varphi), \Theta(f_\sigma, \varphi)) \overline{\mathcal{B}_\pi(\pi(g) f_\pi, f_\pi)}dg}\\
&=& \int_{Z_W\backslash U(W)} \mathcal{B}_{\Theta(\pi)}(\Theta(\pi)(h) \Theta(f_\pi, \varphi_V), \Theta(f_\pi, \varphi_V))\mathcal{B}_{\Theta(\mu)}(\Theta(\mu)(h)\Theta(f_\mu, \varphi_L), \Theta(f_\mu, \varphi_L))\\
&& \overline{\mathcal{B}_\sigma(\sigma(h) f_\sigma, f_\sigma)} dh.
\eeqnan
Here, the vector $f_\mu\in\mu$ is chosen such that $\mathcal{B}_\mu(f_\mu,f_\mu)=1.$
\end{theorem}
\begin{proof}
We start with
$$
\int_{U(V)\times U(W)}\overline{\mathcal{B}_\pi(\pi(g)f_\pi, f_\pi)\mathcal{B}_\sigma(\sigma(h)f_\sigma, f_\sigma) }\mathcal{B}_\omega(\omega(g,h)\varphi, \varphi)dgdh.
$$
By Proposition \ref{localconv}, we may view the integral above as an iterated integral.  By first integrating out the $U(W)$ variable, we have that the above is equal to
$$
\int_{U(V)} \mathcal{B}_{\Theta(\sigma)}(\Theta(\sigma)(g)\Theta(f_\sigma, \varphi), \Theta(f_\sigma, \varphi)) \overline{\mathcal{B}_\pi(\pi(g)f_\pi, f_\pi)}dg.
$$

Before we proceed to integrate the $U(V)$ variable, we remind the reader that $\omega_{W,V\oplus L}\cong\omega_{W,V}\otimes\omega_{W,L}$.  Using this decomposition, we see that the pairing $\mathcal{B}_\omega$ can be written as the product $\mathcal{B}_{\omega, V}\cdot \mathcal{B}_{\omega, L}$.

If we instead integrate out the $U(V)$ variable, we obtain
$$
\int_{U(W)} \mathcal{B}_{\Theta(\pi)}(\Theta(\pi)(h) \Theta(f_\pi, \varphi_V), \Theta(f_\pi, \varphi_V)) \overline{\mathcal{B}_\sigma(\sigma(h) f_\sigma, f_\sigma)} \mathcal{B}_{\omega, L}(\omega(h) \varphi_{L}, \varphi_{L}) dh,
$$
which we rewrite as
\beqnan
\int_{Z_W\backslash U(W)} \mathcal{B}_{\Theta(\pi)}(\Theta(\pi)(h) \Theta(f_\pi, \varphi_V), \Theta(f_\pi, \varphi_V)) \overline{\mathcal{B}_\sigma(\sigma(h) f_\sigma, f_\sigma)}\\ \int_{Z_W} \omega_\pi(z)\omega_\sigma^{-1}(z) \mathcal{B}_{\omega, L}(\omega(zh) \varphi_{L}, \varphi_{L})dz\ dh.
\eeqnan
Noting that $\omega_\pi\omega_\sigma^{-1} = \mu^{-1}$, we see that this is equal to
\beqnan
\int_{Z_W\backslash U(W)} \mathcal{B}_{\Theta(\pi)}(\Theta(\pi)(h) \Theta(f_\pi, \varphi_V), \Theta(f_\pi, \varphi_V)) \overline{\mathcal{B}_\sigma(\sigma(h) f_\sigma, f_\sigma)}\\ \int_{Z_W} \overline{\mu(z)} \mathcal{B}_{\omega, L}(\omega(zh) \varphi_{L}, \varphi_{L})dz\ dh.
\eeqnan
Finally, we see that the inner integral gives the pairing $\mathcal{B}_{\Theta(\mu)}$ on $\Theta(\mu)$.  This completes the proof.
\end{proof}


\section{The Refined Gross-Prasad Conjecture for $U(2)\times U(3)$}\label{finalchapt}
With everything that we've developed so far, we can now prove Conjecture \ref{theconjecture} for $n=2$, provided that $\pi_{n+1}=\Theta(\bar{\sigma})$, where $\sigma$ is a cuspidal, irreducible automorphic representation of $U(2)$.  (The theta-lift is $\Theta(\bar{\sigma})$ because of our normalization of the theta correspondence, and the fact that the seesaw identities we use do not involve complex conjugation.)  We will employ the various Rallis inner product formulae developed in Chapter \ref{thetachapt}, as well as Ichino's triple product formula from Chapter \ref{triplechapt}.

\subsection{The Setup}
We remind the reader of the following seesaw diagram:
\beqn\label{mainss}
\xymatrix{U(V\oplus L)\ar @{-}[d]\ar @{-}[dr]& U(W)\times U(W)\ar @{-}[dl]\ar @{-} [d]\\ U(V)\times U(L) & U(W)}
\eeqn

Here, $V$ is a $2$-dimensional hermitian space over $E/F$, $W$ is a $2$ dimensional skew-hermitian space over $E/F$, and $L$ is a hermitian line over $E/F$.  Using the theory of $\Theta$-correspondence and seesaw duality, we will relate the period integral in Conjecture \ref{theconjecture} (with $n=2$) to the so-called triple product integral considered by Ichino in \cite{triple}.

We fix the following:
\begin{itemize}
\item $\pi$ is an irreducible, cuspidal, tempered, automorphic representation of\\ $U(V)(\A_F).$
\item $\sigma$ is an irreducible, cuspidal, tempered, automorphic representation of\\ $U(W)(\A_F).$
\item $\mu := \omega_\sigma\omega_\pi^{-1}$ is an automorphic character of $U(L)(\A_F)$, where $\omega_\sigma$ and $\omega_\pi$ are the central characters of $\sigma$ and $\pi$, respectively.
\item $(\omega_\psi,\mathcal{S})$ is a Weil representation of $\widetilde{Sp}(\mathbb{W})(\A_F)$.  (See Chapter \ref{thetachapt} for notation.)
\end{itemize}
We also fix local pairings $\mathcal{B}_{\pi_v}, \mathcal{B}_{\sigma_v}, \mathcal{B}_{\mu_v}$ such that $\prod_v \mathcal{B}_{\pi_v}, \prod_v\mathcal{B}_{\sigma_v}$ and $\prod_v\mathcal{B}_{\mu_v}$ give the respective Petersson inner products on the global representation.

After fixing splitting data $(\gamma_V,\gamma_L,\gamma_W)$ as in Chapter \ref{thetachapt}, we consider $\Theta(\bar\pi):=\Theta_{V,W\gamma_W,\gamma_V,\psi}(\bar\pi)$ on $U(W)(\A_F)$, $\Theta(\bar\sigma):=\Theta_{W,V\oplus L,\gamma_V\gamma_L,\gamma_W,\psi}(\bar\sigma)$ on $U(V\oplus L)(\A_F)$, and $\Theta(\bar\mu):=\Theta_{L,W\gamma_W,\gamma_L,\psi}(\bar\mu)$ on $U(W)(\A_F)$.  We take $\gamma_W,\gamma_V = \gamma^2$ and $\gamma_L=\gamma$, where $\gamma$ is a character of $\A_E^\times/E^\times$ such that $\gamma|_{\A_F^\times} = \chi_{E/F}$.  We assume that these $\Theta$-lifts are cuspidal.

By using Ichino's triple product formula from Chapter \ref{triplechapt}, the various Rallis Inner Product formulae from Chapter \ref{thetachapt}, the explicit local seesaw identity from Chapter \ref{seesawchapt}, and some $L$-function identities in the appendix, we can establish the Refined Gross-Prasad Conjecture for $n=2$ with $\pi_2=\pi$ and $\pi_3=\Theta(\bar\sigma)$:
\begin{theorem}\label{finalcountdown} Let $\mu,\pi,\sigma$ and $\Theta(\bar\sigma)$ be as above.  Let $f_{\Theta(\bar\sigma)}\in \Theta(\bar\sigma)$ and $f_\pi\in \pi$ be cusp forms such that $f_{\Theta(\bar\sigma)} = \otimes_v f_{\Theta(\bar\sigma)_v}$ and $f_\pi=\otimes_v f_{\pi_v}$.  Then
$$
\mathcal{P}(f_{\Theta(\bar\sigma)}, f_\pi) = \frac{\Delta_{G_3}L_E(1/2, BC(\Theta(\bar\sigma))\boxtimes BC(\pi))}{|S_{\psi_{\Theta(\bar\sigma)}}|\cdot |S_{\psi_\pi}| L_F(1,\Theta(\bar\sigma),\operatorname{Ad})L_F(1,\pi,\operatorname{Ad})}\prod_v\mathcal{P}_v(f_{\Theta(\bar\sigma)_v}, f_{\pi_v}).
$$
Here, $\psi_{\Theta(\bar\sigma)}$ and $\psi_\pi$ are the `$L$-parameters' for $\Theta(\bar\sigma)$ and $\pi$,\footnote{We remark that this makes sense in this case, in light of the work of Rogawski in \cite{rog}.} and $S_{\psi_{\Theta(\bar\sigma)}}$ and $S_{\psi_\pi}$ are the associated component groups.
\end{theorem}
The proof of the above will occupy the rest of this chapter.

\begin{remark}
By Remark \ref{ram}, we note that assuming that $\pi$ and $\sigma$ are tempered is not necessary.  The only place where temperedness was necessary was in proving Proposition \ref{localconv}.  Using the Ramanujan bounds from \cite{ks1} and \cite{ks2} obviates this assumption.
\end{remark}

\subsection{Proof of Theorem \ref{finalcountdown}}  The proof of Theorem \ref{finalcountdown} involves using both a global and local seesaw identity, as well as all of the various Rallis Inner Product Formulae.

The first global seesaw identity that we need is
\beqn
\mathcal{P}'\circ\mathcal{T}_1' = \mathcal{I}\circ\mathcal{T}_2'
\eeqn
where the maps are defined as follows:
$$
\mathcal{T}_1': (V_\pi\boxtimes \bar{V}_\pi)\otimes (V_\mu\boxtimes \bar{V}_\mu)\otimes (V_\sigma\boxtimes \bar{V}_\sigma)\otimes (\mathcal{S}\boxtimes\bar{\mathcal{S}})\to (V_\pi\boxtimes \bar{V}_\pi)\otimes (V_\mu\boxtimes \bar{V}_\mu)\otimes (V_{\Theta(\bar\sigma)}\boxtimes \bar{V}_{\Theta(\bar\sigma)})
$$
is the map induced by the global theta integral for $\sigma$.  Similarly, 
$$
\mathcal{T}_2' : (V_\pi\boxtimes \bar{V}_\pi)\otimes (V_\mu\boxtimes \bar{V}_\mu)\otimes (V_\sigma\boxtimes \bar{V}_\sigma)\otimes (\mathcal{S}\boxtimes\bar{\mathcal{S}})\to (V_{\Theta(\bar\pi)}\boxtimes \bar{V}_{\Theta(\bar\pi)})\otimes (V_{\Theta(\bar\mu)}\boxtimes \bar{V}_{\Theta(\bar\mu)})\otimes (V_\sigma\boxtimes \bar{V}_\sigma)
$$
is the map induced by the global theta integrals for $\pi$ and $\mu$.  Also, the global period
$$
\mathcal{P}': (V_{\Theta(\bar\sigma)}\boxtimes \bar{V}_{\Theta(\bar\sigma)})\otimes ({V}_\pi\boxtimes \bar{V}_\pi)\otimes ({V}_\mu\boxtimes \bar{V}_\mu)\to\C
$$
is defined by
\beqnan
\mathcal{P}'(f_{\Theta(\bar\sigma)}, \bar{f}'_{\Theta(\bar\sigma)}, {f}_\pi, \bar{f'}_\pi, {f}_\mu, \bar{f'}_\mu) &:=& \int_{[U(V)\times U(L)]} f_{\Theta(\bar\sigma)}((g_1,g_2)){f_\pi(g_1) f_\mu(g_2)}dg_1dg_2\\&&\times\int_{[U(V)\times U(L)]} \overline{f'_{\Theta(\bar\sigma)}((g_1, g_2))f'_\pi(g_1) f'_\mu(g_2)}dg_1dg_2.
\eeqnan
Here, we view $(g_1,g_2)\in U(V\oplus L)$ via the natural embedding $U(V)\times U(L)\hookrightarrow U(V\oplus L)$.  Finally, we consider the map
$$
\mathcal{I}: (V_{\Theta(\bar\pi)}\boxtimes \bar{V}_{\Theta(\bar\pi)})\otimes (V_{\Theta(\bar\mu)}\boxtimes \bar{V}_{\Theta(\bar\mu)})\otimes (V_\sigma\boxtimes \bar{V}_\sigma)\to\C
$$
which is given by
\beqnan
\mathcal{I}(f_{\Theta(\bar\pi)}, \bar{f}'_{\Theta(\bar\pi)}, f_{\Theta(\bar\mu)}, \bar{f}'_{\Theta(\bar\mu)}, f_\sigma, \bar{f}'_\sigma) &:=& \int_{[U(W)]}f_{\Theta(\bar\pi)}(g)f_{\Theta(\bar\mu)}(g)f_\sigma(g)dg \\&&\times \int_{[U(W)]} \overline{f'_{\Theta(\bar\pi)}(g)f'_{\Theta(\bar\mu)}(g)f'_\sigma(g)}dg
\eeqnan
and is closely related to Ichino's triple product integral.

We follow the convention of Chapter \ref{intro} and set $$\mathcal{P}'(f_{\Theta(\bar\sigma)}, f_\pi, f_\mu) := \mathcal{P}'(f_{\Theta(\bar\sigma)}, \bar{f}_{\Theta(\bar\sigma)}, {f}_\pi, \bar{f}_\pi, {f}_\mu, \bar{f}_\mu).$$  We follow a similar convention for $\mathcal{I}$ and set $$ \mathcal{I}(f_{\Theta(\bar\pi)}, f_{\Theta(\bar\mu)}, f_\sigma) := \mathcal{I}(f_{\Theta(\bar\pi)}, \bar{f}_{\Theta(\bar\pi)}, f_{\Theta(\bar\mu)}, \bar{f}_{\Theta(\bar\mu)}, f_\sigma, \bar{f}_\sigma).$$

We note that $\mathcal{P}'$ is not quite the period integral in Theorem \ref{finalcountdown}; however, the two are related by a constant.  Indeed, if we denote by $\mathcal{P}$ the LHS of Theorem \ref{finalcountdown}, then we have the following:
\begin{lemma}\label{periodcompare}
$$
\mathcal{P}'(f_{\Theta(\bar\sigma)}, f_\pi, \mu)=4\cdot\mathcal{P}(f_{\Theta(\bar\sigma)}, f_\pi).
$$
\end{lemma}
\begin{proof} We see that by the change of variables $g_1\mapsto g_1g_2$ we have
\beqnan
 &&\int_{[U(V)\times U(L)]} f_{\Theta(\bar\sigma)}((g_1, g_2)){f_\pi(g_1)\mu(g_2)}dg_1dg_2\\&&= \int_{[U(V)\times U(L)]} f_{\Theta(\bar\sigma)}((g_1g_2, g_2)){f_\pi(g_1g_2)\mu(g_2)dg_1dg_2}.
\eeqnan
Note that $\Theta(\bar{\sigma})$ has central character $\omega_{\Theta(\bar\sigma)}=\omega_{\Theta(\sigma)}^{-1}=\omega_\sigma^{-1}$.  So, after observing that $(g_2, g_2)$ is in the center of $U(V\oplus L)$ and $g_2$ is in the center of $U(V)$, we have
\beqnan
&&\int_{[U(V)\times U(L)]} f_{\Theta(\bar\sigma)}((g_1g_2, g_2)){f_\pi(g_1g_2)\mu(g_2)dg_1dg_2}\\ 
&&= \int_{[U(V)\times U(L)]} \omega_{\Theta(\bar\sigma)}(g_2)\omega_\pi(g_2)\mu(g_2) f_{\Theta(\bar\sigma)}|_{U(V)}(g_1) {f_\pi(g_1)}dg_1dg_2\\ 
&&= \int_{[U(V)\times U(L)]} f_{\Theta(\bar\sigma)}|_{U(V)}(g_1) {f_\pi(g_1)}dg_1dg_2\\
&&= \operatorname{Vol}([U(L)]) \int_{[U(V)]} f_{\Theta(\bar\sigma)}|_{U(V)}(g) {f_\pi(g)}dg\\
&&= 2\int_{[U(V)]} f_{\Theta(\bar\sigma)}|_{U(V)}(g) {f_\pi(g)}dg.
\eeqnan
\end{proof}
With this, we can restate our global seesaw identity as
\beqn\label{newss}
4\cdot \mathcal{P}\circ \mathcal{T}_1=\mathcal{I}\circ\mathcal{T}_2
\eeqn
where
$$
\mathcal{T}_1:(V_\pi\boxtimes\bar{V}_\pi)\otimes (V_\sigma\boxtimes\bar{V}_\sigma)\otimes (\mathcal{S}\boxtimes\bar{\mathcal{S}})\to (V_\pi\boxtimes\bar{V}_\pi)\otimes (V_{\Theta(\bar\sigma)}\boxtimes \bar{V}_{\Theta(\bar\sigma)})
$$
is the map induced by the global theta-lift of $\sigma$.  More explicitly, we have
$$
(f_{\pi,1}\boxtimes \bar{f}_{\pi,2})\otimes (f_{\sigma,1}\boxtimes \bar{f}_{\sigma,2})\boxtimes (\varphi_1\boxtimes \bar{\varphi}_2) \mapsto (f_{\pi,1} \boxtimes \bar{f}_{\pi,2}) \otimes (\Theta(f_{\sigma,1}, \varphi_1)\boxtimes\Theta(\bar{f}_{\sigma,2},\bar{\varphi}_2)).
$$
Here, $\Theta$ is the global theta integral.  Also,
$$
\mathcal{T}_2:(V_\pi\boxtimes\bar{V}_\pi)\otimes (V_\sigma\boxtimes\bar{V}_\sigma)\otimes (\mathcal{S}\boxtimes\bar{\mathcal{S}})\to (V_{\Theta(\bar\pi)}\boxtimes \bar{V}_{\Theta(\bar\pi)})\otimes (V_{\Theta(\bar\mu)}\boxtimes \bar{V}_{\Theta(\bar\mu)})\otimes (V_\sigma\boxtimes \bar{V}_{\bar\sigma})
$$
is the map induced from $\mathcal{T}_2'$ by fixing the argument $\mu\otimes\bar{\mu}\in V_{\mu}\boxtimes \bar{V}_\mu.$  Here, we use the canonical decomposition $\mu=\otimes \mu_v$, where the $\mu_v$ are characters of $U(L)_v$.

We define the corresponding local maps analogously by fixing decompositions $\Theta(\bar{\sigma})=\otimes_v\theta_v(\bar{\sigma}_v)$, $\Theta(\bar{\pi})=\otimes_v\theta_v(\bar{\pi}_v)$, and $\Theta(\bar{\mu})=\otimes_v\theta_v(\bar{\mu}_v).$  Globally, we have:
$$
\Theta: \mathcal{S}\otimes \sigma \to \Theta(\bar{\sigma}),
$$
and locally we have:
$$
\theta_v: \mathcal{S}_v\otimes\sigma_v\to \theta(\bar{\sigma}_v).
$$
(Of course, the definitions for $\pi$ and $\mu$ are analogous.) With these, we see that $\mathcal{T}_i=\otimes_v\mathcal{T}_{i,v}$.  

Before proceeding, we note that $\Theta(\bar\mu)$ is dihedral with respect to $E/F$.  Indeed, we have $BC(\Theta(\bar\mu)) \cong \pi(\gamma^{-1}BC(\bar\mu), \gamma)$, the principal series representation of $GL_2(\A_E)$.  By Corollary \ref{ichtrip} and line \ref{newss} we have
\beqna
\mathcal{P}\circ\mathcal{T}_1 &=& \frac{\zeta_F(2)^2L_F(1,\chi_{E/F})^3}{8\cdot |X(\Theta(\bar\pi))|\cdot |X(\sigma)|\cdot |X(\Theta(\bar\mu))|}\times\label{prerallis}\\
&&\frac{L_E(1/2, BC(\Theta(\bar\pi))\boxtimes BC(\sigma)\boxtimes\gamma^{-1})}{L_F(1,\Theta(\bar\pi),\operatorname{Ad})L_F(1,\sigma,\operatorname{Ad})L_F(1,\Theta(\bar\mu),\operatorname{Ad})}\prod_v\mathcal{I}_v\circ\mathcal{T}_{2,v}\nonumber
\eeqna
where the $\mathcal{I}_v$ are defined with suitably chosen local pairings as in Chapter \ref{triplechapt}.

The next step in the argument is to use the Rallis inner product formulae from Chapter \ref{thetachapt} as well as the local seesaw identity from Chapter \ref{seesawchapt}.  However, note that the matrix coefficients in Chapter \ref{seesawchapt} are attached to the `big' theta-lifts.  Before we can make use of our local seesaw identity, we must prove a lemma which allows us to relate local integrals of matrix coefficients of `big' theta-lifts to those for the corresponding `small' theta-lifts.

Recall that in Chapter \ref{seesawchapt} we considered the following pairings $\mathcal{B}_{\Theta(\bar\tau_v)}$ on the big local theta-lifts $\Theta(\bar\tau_v)$ for $\tau=\mu,\pi,\sigma$:
$$
\mathcal{B}_{\Theta(\bar\tau_v)}(\Theta(\bar{f}_{1,v}, \varphi_{1,v}), \Theta(\bar{f}_{2,v}, \varphi_{2,v})) := \int_{G_v} \mathcal{B}_{\omega_v}(\omega_v(g_v)\varphi_{1,v}, \varphi_{2,v})\mathcal{B}_{\tau_v}(\tau_v(g_v)f_{1,v}, f_{2,v})dg_v.
$$
Here, $G_v$ is $U(1)_v$ if $\tau=\mu$, and $G_v=U(2)_v$ if $\tau=\pi$ or $\tau=\sigma$.  Now, we observe that
$$
\mathcal{B}_{\Theta(\bar\tau_v)}(\Theta(\bar{f}_{1,v}, \varphi_{1,v}), \Theta(\bar{f}_{2,v}, \varphi_{2,v})) = Z_v(s_0, f_{1,v}, f_{2,v}, \Phi_{s_0, v}, \chi_v)
$$
where $s_0=0$ if $\tau=\pi$ and $s_0=1/2$ if $\tau=\mu$ or $\tau=\sigma$, $\Phi_{s,v}=\delta(\varphi_{1,v}, \varphi_{2,v})$, and $\chi_v$ is the appropriate power of $\gamma_v$.  Fix an isomorphism $\Theta(\bar\tau)\cong \otimes_v \theta(\bar\tau_v)$.  The observation above and the Rallis inner product formulae from Chapter \ref{thetachapt} give
\beqn\label{descender}
\langle\theta(\bar{f}_1, \varphi_1), \theta(\bar{f}_2, \varphi_2)\rangle_{\Theta(\bar{\tau})} = \prod_v \mathcal{B}_{\Theta(\bar{\tau}_v)}(\Theta(\bar{f}_{1,v}, \varphi_{1,v}), \Theta(\bar{f}_{2,v}, \varphi_{2,v})).
\eeqn

Since we are assuming that the $\Theta(\bar\tau)$ are cuspidal, and therefore semisimple, we know that $\langle\theta(\bar{f}_1, \varphi_1), \theta(\bar{f}_2, \varphi_2)\rangle_{\Theta(\bar\tau)}$ factors as a map
$$
\bar\tau\times\omega\times\tau\times\bar\omega\to \Theta(\bar{\tau})\times\overline{\Theta(\bar\tau)}
$$
composed with the Petersson inner product on $\Theta(\bar\tau)$.  This, along with \ref{descender} above implies that $\mathcal{B}_{\Theta(\bar{\tau}_v)}$ descends to a pairing on the small theta-lift.

So, at each place $v$, we are entitled to define the following pairings on the small theta-lifts.  We remark that here the $\theta_v$ are those from the definition of the maps $\mathcal{T}_{i,v}$ above.
\beqnan
&&\mathcal{B}_{\theta(\bar\sigma_v)}^\sharp(\theta(\bar{f}_{1,v},\varphi_{1,v}), \theta(\bar{f}_{2,v},\varphi_{2,v})) :=\\
&&\left(\frac{L_{E_v}(1, BC(\sigma_v)\otimes\gamma_v^3)}{\zeta_{F_v}(2) L_{F_v}(3,\chi_{E_v/F_v})}\right)^{-1}\int_{U(W)_v} \mathcal{B}_{\omega_v}(\omega_v(h) \varphi_{1,v},\varphi_{2,v})\mathcal{B}_{\sigma_v}(\sigma_v(h)f_{1,v}, f_{2,v})dh,
\eeqnan
for $\varphi_{i,v}\in\omega_v$ and $f_{i,v}\in\sigma_v$,
\beqnan
&&\mathcal{B}_{\theta(\bar\pi_v)}^\sharp(\theta(\bar{f}_{1,v},\varphi_{1,v}),\theta(\bar{f}_{2,v},\varphi_{2,v})) :=\\
&&\left(\frac{L_{E_v}(1/2, BC(\pi_v)\otimes \gamma_v^2)}{L_{F_v}(1,\chi_{E_v/F_v})\zeta_{F_v}(2)}\right)^{-1}\int_{U(V)_v} \mathcal{B}_{\omega_v}(\omega_v(g) \varphi_{1,v},\varphi_{2,v})\mathcal{B}_{\pi_v}(\pi_v(g) f_{1,v}, f_{2,v})dg,
\eeqnan
for $\varphi_{i,v}\in\omega_v$ and $f_{i,v}\in\pi_v$,
\beqnan
&&\mathcal{B}_{\theta(\bar\mu_v)}^\sharp(\theta(\bar{f}_{1,v},\varphi_{1,v}), \theta(\bar{f}_{2,v}, \varphi_{2,v})) :=\\
&& \left(\frac{L_{E_v}(1, BC(\mu_v)\otimes \gamma_v^2)}{\zeta_{F_v}(2)}\right)^{-1}\int_{U(L)} \mathcal{B}_{\omega_v}(\omega_v(g), \varphi_{1,v},\varphi_{2,v})\mathcal{B}_{\mu_v}(\mu_v(g) f_{1,v}, f_{2,v})dg,
\eeqnan
for $\varphi_{i,v}\in\omega_v$ and $f_{i,v}\in\mu_v$.
We remark that we have made normalizations so that the parings take value $1$ for unramified data.

With these in place, we set $\mathcal{I}_v^\sharp(f_{\sigma_v}, f_{\theta(\bar\pi_v)}, f_{\theta(\bar\mu_v)}) := \mathcal{I}_v^\sharp$ where
\beqnan
\mathcal{I}_v^\sharp&:=&\left(\frac{\zeta_{F_v}(2)^2L_{E_v}(1/2, BC(\bar\sigma_v)\boxtimes BC(\pi_v)\boxtimes\gamma_v^{-1})L_{F_v}(1,\chi_{E_v/F_v})^3}{L_{F_v}(1,\sigma_v,\operatorname{Ad})L_{F_v}(1,\theta(\bar\pi_v),\operatorname{Ad})L_{F_v}(1,\theta(\bar\mu_v),\operatorname{Ad})}\right)^{-1}\\
&&\int_{Z_W\backslash U(W)_v} {\mathcal{B}_{\sigma_v}(\sigma(g_v)f_{\sigma_v}, f_{\sigma_v})}\\
&&\mathcal{B}^\sharp_{\theta(\bar\pi_v)}(\theta(\bar\pi_v)(g)f_{\theta(\bar\pi_v)}, f_{\theta(\bar\pi_v)})\mathcal{B}^\sharp_{\theta(\bar\mu_v)}(\theta(\bar\mu_v)(g_v) f_{\theta(\bar\mu_v)},f_{\theta(\bar\mu_v)} ) dg_v.
\eeqnan
We also set $\mathcal{P}_v^\sharp(f_{\theta(\bar\sigma_v)}, f_{\pi_v}) := \mathcal{P}_v^\sharp$ where
\beqnan
\mathcal{P}_v^\sharp&:=& \left(\Delta_{G_{3,v}}\frac{L_{E_v}(1/2, BC(\theta(\bar\sigma_v))\boxtimes BC(\pi_v))}{L_{F_v}(1,\theta(\bar\sigma_v),\operatorname{Ad})L_{F_v}(1,\pi_v,\operatorname{Ad})}\right)^{-1}\\
&&\int_{U(V)_v} \mathcal{B}^\sharp_{\theta(\bar\sigma_v)}(\theta(\bar\sigma_v)(g_{v})f_{\theta(\bar\sigma_v)}, f_{\theta(\bar\sigma_v)}){\mathcal{B}_{\pi_v}(\pi_{v}(g_{v})f_{\pi_v}, f_{\pi_v}) }dg_{v}.
\eeqnan
Once again, the normalizations made in the definitions above ensure that the functionals $\mathcal{P}_v^\sharp$ and $\mathcal{I}_v^\sharp$ take value $1$ for unramified data.

The local seesaw identity of the previous chapter, along with the $L$-function identities in the appendix give:
$$
\mathcal{P}_v^\sharp \circ\mathcal{T}_{1,v} = \mathcal{I}_v^\sharp\circ\mathcal{T}_{2,v}
$$
assuming that the $f_{\mu_v}$ are chosen so that $\mathcal{B}_{\mu_v}(f_{\mu_v}, f_{\mu_v})=1$.  However, recall that we have taken $f_{\mu_v}=\mu_v$, and that the pairings have been chosen so that $\prod_v\mathcal{B}_{\mu_v}=\mathcal{B}_\mu$.  So, we see that in our situation, we actually have
$$\prod_v\mathcal{B}_{\mu_v}(f_{\mu_v}, f_{\mu_v}) = \prod_v\mathcal{B}_{\mu_v}({\mu_v}, {\mu_v}) = \mathcal{B}_\mu(\mu,\mu) = 2.$$
We must account for this factor of two.  With the pairings as we have chosen them, the choice $f_{\mu_v}=\mu_v$, and the remarks above, we see that
\beqn\label{newlocalseesaw}
2\cdot\prod_v\mathcal{P}^\sharp_v\circ\mathcal{T}_{1,v}=\prod_v\mathcal{I}^\sharp_v\circ\mathcal{T}_{2,v}.
\eeqn

However, neither side of line \ref{newlocalseesaw} appears in line \ref{prerallis}.  This is where we use two of the three Rallis inner product formulae from Chapter \ref{thetachapt}.  Indeed, by Theorems \ref{ripfu1u2} and \ref{ripfu2u2}, we have:
$$
\prod_v\mathcal{I}_v\circ\mathcal{T}_{2,v} = |X(\Theta(\bar\mu))|\frac{L_E(1,BC(\mu)\otimes\gamma^2)L_E(1/2, BC(\pi)\otimes\gamma^2)}{\zeta_F(2)^2L_F(1,\chi_{E/F})}\prod_v \mathcal{I}_v^\sharp\circ \mathcal{T}_{2,v}.
$$
So, using this, along with lines \ref{prerallis} and \ref{newlocalseesaw}, we have
\beqnan
\mathcal{P}\circ\mathcal{T}_1&=& \frac{L_F(1, \chi_{E/F})^2L_E(1/2, BC(\Theta(\pi))\boxtimes BC(\sigma)\boxtimes \gamma^{-1})}{4\cdot |X(\Theta(\bar\pi))|\cdot |X(\sigma)|}\\
&& \times\frac{L_E(1,BC(\mu)\otimes\gamma^2)L_E(1/2, BC(\pi)\otimes\gamma^2)}{L_F(1,\sigma,\operatorname{Ad})L_F(1,\Theta(\bar\pi),\operatorname{Ad})L_F(1,\Theta(\bar\mu),\operatorname{Ad})}\\
&&\times\prod_v \mathcal{P}^\sharp_v \circ\mathcal{T}_{1,v}.
\eeqnan
To finish the proof, we need to use the third Rallis inner product formula to replace the equation above with one involving $\prod_v\mathcal{P}_v\circ \mathcal{T}_{1,v}$ instead of $\prod_v \mathcal{P}^\sharp_v \circ\mathcal{T}_{1,v}$.  From Theorem \ref{rip}, we have
$$
\prod_v \mathcal{P}_v\circ \mathcal{T}_{1,v} = \frac{2\cdot L_E(1, BC(\sigma)\otimes\gamma^3)}{\zeta_F(2)L_F(3,\chi_{E/F})}\prod_v \mathcal{P}^\sharp_v\circ \mathcal{T}_{1,v}.
$$
This, along with the $L$-function identities in the appendix gives
$$
\mathcal{P}\circ\mathcal{T}_1 = \frac{\Delta_{G_3}}{8\cdot |X(\Theta(\bar\pi))|\cdot |X(\sigma)|}\cdot\frac{ L_E(1/2, BC(\Theta(\bar\sigma))\boxtimes BC(\pi))}{L_F(1, \Theta(\bar\sigma),\operatorname{Ad}) L_F(1,\pi,\operatorname{Ad})}\prod_v \mathcal{P}_v \circ\mathcal{T}_{1,v}.
$$
Finally, we note that
$$
8\cdot |X(\Theta(\bar\pi))|\cdot |X(\sigma)| = |S_{\psi_{\Theta(\bar\sigma)}}|\cdot |S_{\psi_\pi}|,
$$
which completes the proof.

\appendix
\section{$L$-function identities} We collect several $L$-function identities used above.  In particular, by using known relationships between $L$-parameters of representations and their $\Theta$-lifts, we compare some associated $L$-functions.

As in Chapter \ref{finalchapt}, we take $\sigma$ and $\pi$ to be irreducible, tempered, cuspidal automorphic representations of $U(2)$, and $\mu=\omega_\sigma\omega_\pi^{-1}$.  As in Chapter \ref{finalchapt}, we fix splitting characters which are powers of a fixed character $\gamma$, and consider the $\Theta$-lifts $\Theta(\bar{\pi})$ and $\Theta(\bar{\mu})$ on $U(2)$, and $\Theta(\bar{\sigma})$ on $U(3)$.  Once again, we remind the reader that the need to consider $\Theta(\bar{\pi}),\Theta(\bar{\sigma})$, and $\Theta(\bar{\mu})$ comes from our normalization of the theta correspondence, and the fact that in the seesaw identities we used, there was no complex conjugation.

The first identity we record is the following:
\begin{proposition}\label{piad}
$$
L_F(s,\Theta(\bar{\pi}),\operatorname{Ad}) = L_F(s,\pi,\operatorname{Ad}).
$$
\end{proposition}
\begin{proof} By Theorem $11.2$ in \cite{ggp}, we see that $\bar{\pi}$ and $\Theta(\bar{\pi})$ have the same $L$-parameters.  The result above then follows from the fact that $L_F(s, \bar{\pi}, \operatorname{Ad}) = L_F(s, \pi,\operatorname{Ad}).$
\end{proof}
Now we compute the adjoint $L$-function for $\Theta(\sigma)$:
\begin{proposition}\label{sigmaad}
$$
L_F(s, \Theta(\bar{\sigma}), \operatorname{Ad}) = L_F(s, \chi_{E/F})L_F(s, \sigma,\operatorname{Ad})L_E(s, BC({\sigma})\otimes\gamma^3).
$$
\begin{proof} Let $M$ and $N$ denote the restrictions of the $L$-parameters of $\bar{\sigma}$ and $\Theta(\bar{\sigma})$ to $WD(E)$.  By Theorem $8.1$ in \cite{ggp2}, this restriction inflicts no loss of information.  From \cite{grs}, we have that $M$ and $N$ are related as follows:
$$
N = \gamma^{-1}M\oplus \gamma^2.
$$
The result follows from this and the fact that $L_F(s, \bar{\sigma},\operatorname{Ad}) = L_F(s,\sigma,\operatorname{Ad})$.
\end{proof}
\end{proposition}
Finally, we compute the adjoint $L$-function for $\Theta(\mu)$:
\begin{proposition}\label{muad}
$$
L_F(s,\Theta(\bar{\mu}),\operatorname{Ad}) = L_F(s,\chi_{E/F})^2L_E(s, BC({\mu})\otimes\gamma^2).
$$
\end{proposition}
\begin{proof} Let $M$ and $N$ be as in the previous proposition.  Then it is easy to see that once again we have
$$
N = \gamma^{-1}M \oplus \gamma.
$$
As before, the result follows.
\end{proof}

The last identity we record is one that relates the $L$-function in Ichino's triple product formula to $L$-functions attached to representations of $U(2)$.  In the notation of Chapter \ref{thetachapt}, we have $\tau_1 = BC(\Theta(\bar{\pi})), \tau_2 = BC(\sigma)$, and $\tau_3 = BC(\Theta(\bar{\mu}))$.
\begin{proposition}\label{bigsig}
$$
L_F(s,\Sigma') = L_E(s, BC(\pi)\boxtimes BC(\bar{\sigma})\boxtimes\gamma^{-1}).
$$
\end{proposition}
\begin{proof} We compare $L$-parameters.  Let $N_{\Theta(\bar{\pi})}, N_\sigma,$ and $N_{\Theta(\bar{\mu})}$ denote the $L$-parameters for $\Sigma'_{\Theta(\bar{\pi})},\Sigma'_\sigma,$ and $\Sigma'_{\Theta(\bar{\mu})}$, respectively.  Now, since $\Sigma_{\Theta(\bar{\mu})}$ is dihedral with respect to $E/F$, we know that
$$
N_{\Theta(\bar{\mu})} = \operatorname{Ind}_{WD(E)}^{WD(F)} M
$$
for some one-dimensional representation $M$ of $WD(E)$.  Now observe that we have the following:
$$
N_{\Theta(\bar{\pi})}\otimes N_{\sigma}\otimes \operatorname{Ind}_{WD(E)}^{WD(F)} M = \operatorname{Ind}_{WD(E)}^{WD(F)} \left(N_{\Theta(\bar{\pi})}|_{WD(E)}\otimes N_{\sigma}|_{WD(E)} \otimes M\right).
$$
Now, identifying automorphic representations with their $L$-parameters, we have
\beqnan
L_F(s,\Sigma') &=& L_F(s, N_{\Theta(\bar{\pi})}\otimes N_{\sigma}\otimes \operatorname{Ind}_{WD(E)}^{WD(F)} M)\\
&=& L_F\left(s,  \operatorname{Ind}_{WD(E)}^{WD(F)} \left(N_{\Theta(\bar{\pi})}|_{WD(E)}\otimes N_{\sigma}|_{WD(E)} \otimes M\right)\right)\\
&=& L_F\left(s, \operatorname{Ind}_{WD(E)}^{WD(F)} \left(N_{\Theta(\bar\pi)}|_{WD(E)}\eta_{\Theta(\bar\pi)}\otimes N_\sigma|_{WD(E)}\eta_\sigma\otimes M\eta_{\Theta(\bar\mu)}\right)\right)\\
&=& L_E(s, BC(\Theta(\bar\pi))\boxtimes BC(\sigma)\boxtimes\gamma)\\
&=& L_E(s, BC(\bar\pi)\boxtimes BC(\sigma)\boxtimes\gamma)\\
&=& L_E(s, BC(\pi)\boxtimes BC(\bar\sigma)\boxtimes\gamma^{-1}).
\eeqnan
Note that we have used several facts above.  We have used the fact that $\Theta(\bar\pi)$ and $\bar\pi$ have the same $L$-parameters.  We have also used the fact that $M\eta_{\Theta(\bar{\mu})}$ is one of the summands in the $L$-parameter for $BC(\Theta(\bar\mu))$, which we know to be $\gamma^{-1}M_{\bar\mu}\oplus\gamma$, where $M_{\bar\mu}$ is the $L$-parameter for $\bar\mu$.  We note that we can replace $M\eta_{\Theta(\bar\mu)}$ with either of these two summands without changing the $L$-function.  Finally, for the last equality above we have used Lemma 3.5 in \cite{ggp2}, which simply says that if $M$ is conjugate-self dual, then $\operatorname{Ind}_{WD(E)}^{WD(F)} M$ is self-dual.
\end{proof}

Finally, we have a corollary of the result above.
\begin{corollary}\label{sigcor}
$$
L_F(s,\Sigma')L_E(s, BC(\pi)\otimes\gamma^2) = L_E(s,BC(\Theta(\bar\sigma))\boxtimes BC(\pi)).
$$
\end{corollary}
\begin{proof} This follows from the previous result, along with the fact that the $L$-parameters $N$ of $\Theta(\bar{\sigma})$ and $M$ of $\bar{\sigma}$ are related by:
$$
N = \gamma^{-1}M\oplus \gamma^{2}.
$$
\end{proof}

\bibliography{myrefs}

\begin{thebibliography}{10}

\bibitem{borel}
Armand Borel.
\newblock Automorphic {$L$}-functions.
\newblock {\em Proceedings of Symposia in Pure Mathematics}, 33:27--61, 1979.

\bibitem{ggp}
Wee~Teck Gan, Benedict Gross, and Dipendra Prasad.
\newblock Restrictions of representations of classical groups: Examples.
\newblock {\em Asterisque}, 118, to appear.

\bibitem{ggp2}
Wee~Teck Gan, Benedict Gross, and Dipendra Prasad.
\newblock Symplectic local root numbers, central critical {$L$}-values, and
  restriction problems in the representation theory of classical groups.
\newblock {\em Asterisque}, 118, to appear.

\bibitem{so4so5}
Wee~Teck Gan and Atsushi Ichino.
\newblock On endoscopy and the refined {G}ross-{P}rasad conjecture for
  $({SO}_4, {SO}_5)$.
\newblock {\em Journal of the Institute of Mathematics of Jussieu},
  10(2):235--324, 2011.

\bibitem{grs}
S.~Gelbart, J.~Rogawski, and D.~Soudry.
\newblock Endoscopy, theta-liftings, and period integrals for the unitary group
  in three variables.
\newblock {\em Annals of Math}, 145:419--476, 1997.

\bibitem{l:orthog}
David Ginzburg, Illya Piatetski-Shapiro, and Stephen Rallis.
\newblock ${L}$-functions for the orthogonal group.
\newblock {\em Memoirs of the American Mathematical Society}, 128(611), July
  1997.

\bibitem{motive}
Benedict Gross.
\newblock On the motive of a reductive group.
\newblock {\em Inventiones mathematicae}, 1997.

\bibitem{gp}
Benedict Gross and Dipendra Prasad.
\newblock On the decomposition of a representation of ${SO}_n$ when restricted
  to ${SO}_{n-1}$.
\newblock {\em Canadian Journal of Math}, 44(5):974 -- 1002, 1992.

\bibitem{2x2}
Michael Harris.
\newblock {$L$}-functions of $2\times 2$ unitary groups and factorizations of
  periods of {H}ilbert modular forms.
\newblock {\em Journal of the American Mathematical Society}, 6(3):637--719,
  July 1993.

\bibitem{coho1}
Michael Harris.
\newblock Cohomological automorphic forms on unitary groups, i: Rationality of
  the theta correspondence.
\newblock {\em Proceedings of Symposia in Pure Mathematics}, 66(2):103--200,
  1999.

\bibitem{coho2}
Michael Harris.
\newblock Cohomological automorphic forms on unitary groups, ii: Period
  relations and values of {$L$}-functions.
\newblock In Jian-Shu Li, Eng-Chye Tan, and Nolan Wallach, editors, {\em
  Harmonic analysis, group representations, automorphic forms, and invariant
  theory}, volume~12. World Scientific Publishing, 2008.

\bibitem{adjoint}
Michael Harris.
\newblock $l$-functions and periods of adjoint motives.
\newblock {\em Algebra and Number Theory}, to appear.

\bibitem{hks}
Michael Harris, Stephen Kudla, and William Sweet.
\newblock Theta dichotomy for unitary groups.
\newblock {\em Journal of the American Mathematical Society}, 9(4):941--1004,
  1996.

\bibitem{close}
Michael Harris, Jian-Shu Li, and Binyong Sun.
\newblock Theta correspondence for close unitary groups.
\newblock In {\em Arithmetic Geometry and Automorphic forms, Volume in honor of
  the 60th birthday of {S}tephen {S}. {K}udla}. Int. Press and the Higher
  Education Press of China, 2011.

\bibitem{hiraga}
Kaoru Hiraga and Hiroshi Saito.
\newblock On ${L}$-packets for inner forms of ${SL}_n$.
\newblock {\em Memoirs of the American Mathematical Society}, to appear.

\bibitem{sw}
Atsushi Ichino.
\newblock On the {S}iegel-{W}eil formula for unitary groups.
\newblock {\em Mathematische Zeitschrift}, 255(4):721--729, 2007.

\bibitem{triple}
Atsushi Ichino.
\newblock Trilinear forms and the central values of triple product
  ${L}$-functions.
\newblock {\em Duke Math Journal}, 145(2):281 -- 307, 2008.

\bibitem{rgp}
Atsushi Ichino and Tamotsu Ikeda.
\newblock On the periods of automorphic forms on special orthogonal groups and
  the {G}ross-{P}rasad conjecture.
\newblock {\em Geometric Functional Analysis}, 19(5):1378--1425, 2010.

\bibitem{kmsgln}
Shinichi Kato, Atsushi Murase, and Takashi Sugano.
\newblock Shintani functions for ${GL}_n$: An explicit formula.
\newblock unpublished article.

\bibitem{khoury}
Michael Khoury.
\newblock {\em Multiplicity-One Results and Explicit Formulas for Quasi-Split
  $p$-Adic Unitary Groups}.
\newblock PhD thesis, The Ohio State University, 2008.

\bibitem{ks1}
H.~H. Kim.
\newblock On local {$L$}-functions and normalized intertwining operators.
\newblock {\em Canadian Journal of Math}, 57:535--597, 2005.

\bibitem{ks2}
H.~H. Kim and Freydoon Shahidi.
\newblock Cuspidality of symmetric powers with applications.
\newblock {\em Duke Math Journal}, 112:177--197, 2002.

\bibitem{kudla}
S.~Kudla.
\newblock On the local theta correspondence.
\newblock {\em Inventiones mathematicae}, 83:229--255, 1986.

\bibitem{lr}
Erez Lapid and Stephen Rallis.
\newblock On the local factors of representations of classical groups.
\newblock In James Cogdell, Dihua Jiang, Stephen Kudla, David Soudry, and
  Robert Stanton, editors, {\em Automorphic Representations, ${L}$-Functions
  and Applications: Progress and Prospects}, 2005.

\bibitem{nonvan}
Jian-Shu Li.
\newblock Non-vanishing theorems for the cohomology of certain arithmetic
  quotients.
\newblock {\em Journal f{\"u}r die reine und angewandte Mathematik (Crelles
  Journal)}, 428:177--217, 1992.

\bibitem{mvw}
C.~Moeglin, M.-F. Vign\'eras, and Jean-Loup Waldspurger.
\newblock {\em Correspondances de {H}owe sur un corps $p$-adique}, volume 1291
  of {\em Lecture Notes in Mathematics}.
\newblock Springer-Verlag, 1987.

\bibitem{rog}
J.~Rogawski.
\newblock {\em Automorphic Representations of Unitary Groups in Three
  Variables}.
\newblock Princeton University Press, 1990.

\bibitem{satake}
I.~Satake.
\newblock Theory of spherical functions on reductive algebraic groups over
  theory of spherical functions on reductive algebraic groups over $p$-adic
  fields.
\newblock {\em IHES}, 18:1--69, 1963.

\bibitem{shahidi}
Freydoon Shahidi.
\newblock On certain {$L$}-functions.
\newblock {\em Amer. J. Math.}, 103(2):297--355, 1981.

\bibitem{silb}
A.J. Silberger.
\newblock Introduction to harmonic analysis on reductive $p$-adic groups.
\newblock {\em Mathematical Notes, Princeton University Press}, 23, 1979.

\bibitem{tan}
Victor Tan.
\newblock A regularized {S}iegel-{W}eil formula on ${U}(2,2)$ and ${U}(3)$.
\newblock {\em Duke Math Journal}, 94(2):341--378, 1998.

\bibitem{tanapp}
Victor Tan.
\newblock An application of the regularized {S}iegel-{W}eil formula on unitary
  groups to a theta lifting problem.
\newblock {\em Proceedings of the AMS}, 127(10):2811--2820, October 1999.

\bibitem{weak}
V.~S. Varadarajan.
\newblock {\em An Introduction to Harmonic Analysis on Semisimple {L}ie
  Groups}.
\newblock Number~16 in Cambridge studies in advanced mathematics. Cambridge
  Univeristy Press, 1989.

\bibitem{wald}
Jean-Loup Waldspurger.
\newblock Sur les valeurs de certaines fonctions ${L}$ automorphes en leur
  centre de sym{\'e}trie.
\newblock {\em Compositio Mathematica}, 54(2):173--242, 1985.

\bibitem{wald2}
Jean-Loup Waldspurger.
\newblock D\'emonstartion d'une conjecture de dualit\'e de {H}owe dans le cas
  $p$-adique, $p\neq 2$.
\newblock In {\em Festschrift in honor of {P}iatetski-Shapiro}, volume~2, pages
  267--324. Israel Math. Conf. Proc., 1990.

\bibitem{grosszag}
Xinyi Yuan, Shou wu~Zhang, and Wei Zhang.
\newblock Gross-{Z}agier {F}ormula.
\newblock {\em Annals of Mathematics Studies}, to appear.

\end{thebibliography}
\bibliographystyle{plain}
\end{document}